\documentclass[10pt, reqno, oneside, english]{smfart}

\usepackage{smfthm}

\usepackage[height=22cm, bottom=3cm]{geometry}
\usepackage{amsmath, amsthm, amssymb,amsfonts}
\usepackage[utf8]{inputenc}
\usepackage[frenchb]{babel}
\usepackage[T1]{fontenc}
\usepackage{url}
\usepackage{braket}
\usepackage{enumerate}
\usepackage{xcolor}
\usepackage[colorlinks=true, linktoc=page, citecolor=blue, linkcolor=blue, urlcolor=blue]{hyperref}
\usepackage{lipsum}

\usepackage{setspace}
\numberwithin{equation}{section}

\DeclareFontFamily{U}{mathx}{\hyphenchar\font45}
\DeclareFontShape{U}{mathx}{m}{n}{
      <5> <6> <7> <8> <9> <10>
      <10.95> <12> <14.4> <17.28> <20.74> <24.88>
      mathx10
      }{}
\DeclareSymbolFont{mathx}{U}{mathx}{m}{n}
\DeclareFontSubstitution{U}{mathx}{m}{n}
\DeclareMathAccent{\widecheck}{0}{mathx}{"71}

\title{$C^\ast$-blocks and crossed products for classical $p$-adic groups}
\author{Alexandre Afgoustidis}
\address{CNRS, Institut Élie Cartan de Lorraine, Nancy \& Metz, France}
\email{alexandre.afgoustidis@math.cnrs.fr}
\author{Anne-Marie Aubert}
\address{CNRS, Sorbonne Universit\'e, Universit\'e de Paris, Institut de Math\'ematiques de Jussieu – Paris Rive Gauche, 75005 Paris, France}
\email{anne-marie.aubert@imj-prg.fr}
\begin{document}

%-------------------------- Frontmatter------------------------%
\frontmatter
\begin{abstract} Let $G$ be a real or $p$-adic reductive group. We consider the tempered dual of $G$, and its connected components. For real groups, Wassermann proved in 1987, by noncommutative-geometric methods, that each connected component has a simple geometric structure which encodes the reducibility of induced representations. For $p$-adic groups, each connected component of the tempered dual comes with a compact torus equipped with a finite group action, and we prove that a version of Wassermann's theorem holds true under a certain geometric assumption on the structure of stabilizers for that action. We then focus on the case where $G$ is a quasi-split  symplectic, orthogonal or unitary group, and explicitly determine the connected components for which the geometric assumption is satisfied. 
\end{abstract}

%-------------------------- Mainmatter------------------------%

\maketitle

\section{Introduction}

\subsection*{The tempered dual and its topology} Let $F$ be a local field of characteristic zero. Let $G$ be the group of $F$-points of a connected reductive algebraic group defined over $F$, and let $\widehat{G}_{\mathrm{temp}}$ denote the tempered dual of $G$. 

Fix a Levi subgroup $M$ of $G$. We shall write $\mathcal{E}_2(M)$ for the set of equivalence classes of {irreducible representations in the discrete series of $M$}, and $\mathcal{X}_u(M)$ for the abelian group of unitary unramified characters of $M$.

Now fix $\sigma \in \mathcal{E}_2(M)$. Write $\mathcal{O} \subset \mathcal{E}_2(M)$ for the orbit of $\sigma$ under $\mathcal{X}_u(M)$: this is the set of equivalence classes of representations of $M$ of the form $\sigma \otimes \chi$, $\chi \in \mathcal{X}_u(M)$.
Let $\Theta$ denote the $G$-conjugacy class of the pair  $(M,\mathcal{O})$.  We will write $\Theta=(M,\mathcal{O})_G$. Then $\Theta$ determines a connected component $\widehat{G}_\Theta$ of $\widehat{G}_{\mathrm{temp}}$: if $P$ is a parabolic subgroup of $G$ with Levi factor $M$, then $\widehat{G}_\Theta$ is the set of irreducible representations of $G$ that occur in one of the induced representations $\mathrm{Ind}_P^G(\sigma \otimes \chi)$,  $\chi \in \mathcal{X}_u(M)$. 

{The unitary dual $\widehat{G}$ of $G$ admits a canonical topology, which is usually not Hausdorff. Its failure to be Hausdorff can be traced to the reducibility of some of  the induced representations mentioned above. }

{A traditional approach to the topology of $\widehat{G}$, and of the tempered dual $\widehat{G}_{\mathrm{temp}}$, accordingly uses noncommutative-geometric methods.  Let $C^\ast_r(G)$ be the reduced $C^\ast$-algebra of $G$. We equip $\widehat{G}_{\mathrm{temp}}$ with the topology obtained by transport of structure from the Jacobson topology on the spectrum of $C^\ast_r(G)$ (see \cite[\S 3.1, \S 18.1.1]{Dixmier}), and $\widehat{G}_\Theta$ with the induced topology. }

{Attached to $\Theta$ is the connected component $\widehat{G}_\Theta$ of $\widehat{G}_{\mathrm{temp}}$, and there is a natural subalgebra $C^\ast_r(G;\Theta)$ of $G$ with spectrum $\widehat{G}_\Theta$ (see \S \ref{c_theta} below, as well as \cite[Theorem 2.5]{Plymen1} and \cite[\S 2]{AATopo2}).}

When $F = \mathbb{R}$,  Wassermann gave in \cite{Wassermann} a complete and strikingly simple determination of $C^\ast_r(G;\Theta)$ up to strong Morita equivalence. It is our intention here to study the possibility that there could exist a related result when $F$ is $p$-adic, to prove such a result under a certain geometric assumption which we shall be describing shortly, and later to assess the validity of that assumption for quasi-split classical groups. 

We first introduce a bit of notation and recall Wassermann's theorem.

Let us attach to $\Theta=(M,\mathcal{O})_G$ a Weyl-like group $W_{\Theta}$, as in \cite[\S 2.4]{Plymen1}. We begin by setting 
\[ N_G(\Theta) = \left\{ \ \dot{w} \in N_G(M) \ | \ \exists \eta \in \mathcal{X}_u(M) \ : \ ^{\dot{w}}\!\sigma \simeq \sigma \otimes \eta \ \right\} \]
where $\sigma \in \mathcal{O}$ and, for $\dot{w} \in N_G(M)$, $^{\dot{w}}\!\sigma$  stands for the representation $m \mapsto \sigma(\dot{w}^{-1}m\dot{w})$ of $M$. We then define
\[ W_{\Theta} = N_G(\Theta)/M.\]
For $\tau \in \mathcal{O}$, we denote by $W_\tau$ the stabilizer of $\tau$ in $N_G(M)/M$:
\[W_\tau= \left\{ \ \dot{w} \in N_G(M) \ | \ \ ^{\dot{w}}\!\tau \simeq \tau \ \right\}/M.\]  
The group $W_{\Theta}$ acts on $\mathcal{O}$ and  $W_\tau$ is the stabilizer of $\tau$ in that action.

\subsection*{Wassermann's theorem} When $F = \mathbb{R}$, the abelian group $\mathcal{X}_u(M)$ is noncompact and contractible: it is isomorphic with the unitary dual $\widehat{A}$ of the split component $A$ of the center of $M$. (This is a consequence of the ``Langlands decomposition'' of $M$ as a direct product $M^0 A$, where $M^0$ is the subgroup of $M$ generated by all compact subgroups.) Moreover, and crucially, the action of $W_{\Theta}$ on $\mathcal{O}$ always admits a fixed point: by an appropriate twist we can find a point of $\mathcal{O}$ that corresponds to a representation $\sigma \in \mathcal{E}_2(M)$ whose restriction to $A$ is trivial, and  then $\sigma$ is a fixed point for the action of $W_\Theta$ on $\mathcal{O}$.  

The reducibility of the induced representation $\mathrm{Ind}_P^G(\sigma)$ can then be studied through the Knapp-Stein theory of intertwining operators, of which more later. Wassermann's result uses the semidirect product decomposition $W_{\Theta}=W_\sigma = W'_\sigma \rtimes R_\sigma$, where $R_\sigma$ is the Knapp-Stein $R$-group and $W'_\sigma$ is the normal subgroup of elements giving rise to scalar intertwining operators. His theorem can be stated as follows.

\begin{theo}[Wassermann] \label{th_wass} Assume $F = \mathbb{R}$. Let $\sigma$ be the above-described fixed point for the action of $W_\Theta$ on $\mathcal{O}$. There exists a strong Morita equivalence
\begin{equation} \label{wass_reel} C^\ast_r(G,\Theta) \underset{\text{Morita}}{\sim} \mathcal{C}_0(\mathcal{O}/W'_{\sigma}) \rtimes R_{\sigma},\end{equation}
 where $\mathcal{C}_0(\dots)$ denotes the space of continuous functions that vanish at infinity, and $\rtimes$ denotes the $C^\ast$-algebraic crossed-product. \end{theo}

An immediate corollary is that the connected component $\widehat G_\Theta$ is homeomorphic with the ``spectral extended quotient'' $\left[(\mathcal{O}/W'_\sigma)//R_\sigma\right]_{\mathrm{spec}}$ of \cite{ABPSTakagi}. (See \cite{ABPSTakagi} for spectral extended quotients, \cite[\S 4]{EchterhoffEmerson} for a simple description of the topology on the right-hand side{, and \cite[Chapter 2, Appendix A]{Connes} for the fact  that Morita-equivalent $C^\ast$-algebras have homeomorphic spectra).} We note that, in this formula, we have a rare instance of the co-existence of two different \emph{quotients}: we have the spectral extended quotient by $R_\sigma$ of the ordinary quotient  of $\mathcal{O}$ by $W'_\sigma$. This has significant consequences, and encodes at the same time:
\begin{enumerate}[(i)]
\item a parametrization of the irreducible factors of  $\mathrm{Ind}_{P}^G(\tau)$ for $\tau \in \mathcal{O}$: they are indexed by (irreducible representations of) the $R$-group\footnote{\label{footnote_tau} Note that $R_{\tau}$ arises, on the right-hand side of \eqref{wass_reel}, as the stabilizer in $R_\sigma$ of the coset~$\tau W'_{\sigma}$.} $R_\tau$,
\item and a determination of the Fell topology on $\widehat G_\Theta$, with the information that the subtle phenomena entailed by the reducibility of induced representations fit into a simple geometric structure.  The multiple points in the (possibly) non-Hausdorff topology correspond exactly to the number of irreducible constituents in the corresponding induced representations.
\end{enumerate}
Since the result is true for all components $\widehat G_\Theta$, Wassermann's theorem is a noncommutative-geometric way to encode both a parametrization of $\widehat{G}_{\mathrm{temp}}$ (conditional on the determination of the discrete series of Levi subgroups), and a simple description of its topology.

\subsection*{The $p$-adic case: previous work} When $F$ is $p$-adic,  the group $\mathcal{X}_u(M)$ is a compact torus, and no longer isomorphic with $\widehat{A}$ in general. The finite group $W_{\Theta}$ acts on the torus $\mathcal{O}$.

In Wassermann's proof of Theorem \ref{th_wass}, an important role is played by the fact that $\mathcal{X}_u(M) \simeq \widehat{A}$  is $W_{\Theta}$-equivariantly contractible. This is of course no longer true for $p$-adic groups, and makes it seem challenging to prove general results about the structure of $C^\ast_r(G)$ up to Morita equivalence. 

In spite of this difference, Roger Plymen and his students have identified, from the 1990s onwards, a number of examples in which $p$-adic analogues of Wassermann's theorem hold true. 

Plymen proved in 1990 \cite{Plymen1} that if the induced representations $\mathrm{Ind}_{P}^{G}(\tau)$, $\tau \in \mathcal{O}$, are all irreducible, then $C^\ast(G, \Theta)$ is Morita-equivalent with $\mathcal{C}(\mathcal{O}/W_{\Theta})$ $-$ ensuring, in particular, that the connected component $\widehat G_\Theta$ is Hausdorff in that case. That result gives, in particular, a complete description of the reduced $C^\ast$-algebra of $\mathrm{GL}(n,F)$ up to Morita equivalence. 

When $G$ is a $p$-adic Chevalley group, Plymen and his students gave several examples of Levi subgroups $M$ and representations $\sigma\in \mathcal{E}_2(M)$ for which ${W_{\Theta}} = R_\sigma$, for which  $W'_{\sigma}$ is accordingly trivial, and for which $C^\ast_r(G,\Theta)$ is Morita-equivalent with $\mathcal{C}(\mathcal{O})\rtimes R_\sigma$. See \cite{LeungPlymen} for certain principal series representations of $p$-adic Chevalley groups, \cite{JawdatPlymen} for certain elliptic representations of $\mathrm{SL}(n,F)$, and \cite{ChaoPlymen} for a remarkable example in $\mathrm{SL}(4,F)$. 

More recently, Opdam and Solleveld proved in \cite{OpdamSolleveld} that a \emph{localized version} of Wassermann's result holds true in a much more general setting. They proved that a Morita equivalence similar to that of Wassermann always holds when one restricts everything, using localized algebras and germs, to a neighborhood of $\sigma$ in $\mathcal{O}$. Their result necessitates, quite impressively in view of other studies, no hypothesis whatsoever on the $p$-adic group $G$ or on the Levi subgroup $M$ and discrete series representation $\sigma \in \mathcal{E}_2(M)$. Their work brings in, alongside other (and deeper) ingredients, the use of central extensions of the $R$-group to which we will come back later. 

Combining the very general (but local) work of Opdam and Solleveld with the very special (but non-local) examples of Plymen and his collaborators, it is tempting to speculate that Wassermann's theorem, and the attached geometric structure on the connected components of the tempered dual, may hold  for $p$-adic groups in a number of general situations. This hope can be connected with the conjectures of \cite{ABPSTakagi} on the structure of the admissible dual, which postulate the existence of a simple geometric structure governing the questions of reducibility for representations induced from supercuspidal representations of Levi subgroups, although a precise relationship is far from clear.

\subsection*{Our results} We will prove that a version of Wassermann's theorem holds under certain natural assumptions on the action of $W_{\Theta}$ on $\mathcal{O}$, and later study these assumptions in concrete examples. The first hypothesis is the existence of a fixed point for all of $W_{\Theta}$: 

\begin{enonce}{Assumption} \label{hyp_pt_fixe} The action of $W_{\Theta}$ on $\mathcal{O}$ has a fixed point. \end{enonce}

We shall of course discuss, in Part \ref{part2} of this paper and at the end of this Introduction, the plausibility of that assumption. For the moment, let us state our findings about the structure of $C^\ast_r(G; \Theta)$ when Assumption \ref{hyp_pt_fixe} is satisfied.

We may as well assume that $\sigma$ is itself  a fixed point. So we shall now assume that 
\begin{equation} \label{hyp} W_{\Theta} \cdot \sigma = \sigma. \end{equation} 

The Knapp-Stein theory of the $R$-group can be called in as before, and we will use the decomposition $W_{\sigma}=W'_{\sigma} \rtimes R_{\sigma}$ alluded to above (see \S \ref{r_groupe} for details). 

A moment's pause on Theorem \ref{th_wass} reveals that any analogous statement can only be true under a very strong compatibility condition on the $R$-groups $R_\tau$ attached to various points of the orbit $\mathcal{O}$. Indeed, an immediate consequence of point (i) in our discussion of real groups is that Wassermann's result can be true only if the $R$-group $R_\sigma$ attached to the fixed point $\sigma$ can ``see'' all the $R$-groups $R_\tau$, $\tau \in \mathcal{O}$, as subgroups (see the footnote on page \pageref{footnote_tau}). As we shall see in Part \ref{part2}, for $p$-adic groups, it is easy to construct examples where such a favorable situation cannot happen. We shall therefore need a strong additional assumption on the structure of stabilizers $W_{\tau}$, $\tau \in \mathcal{O}$:

\begin{enonce}{Assumption} \label{hyp_inclusions} If $\sigma$ is a fixed point as in \eqref{hyp}, then for every point $\tau \in \mathcal{O}$, the Knapp-Stein decompositions $W_{\tau}=W'_{\tau} \rtimes R_{\tau}$ and $W_{\sigma}=W'_{\sigma} \rtimes R_{\sigma}$ are compatible in the following sense: 
\begin{enumerate}[(a)]
\item we have $W'_\tau \subset W'_\sigma$,
\item and the $R$-group $R_\tau$ is isomorphic with a subgroup of $R_\sigma$.\qedhere
\end{enumerate}\end{enonce}

What we shall prove is that a version of Wassermann's theorem holds true if there exists a fixed point, as in Assumption \ref{hyp_pt_fixe}, that is ``good'' in the sense of Assumption \ref{hyp_inclusions}. We shall also provide examples (for $p$-adic groups) which indicate that one should not expect a reasonable analogue of Wassermann's result  to be true unless Assumptions \ref{hyp_pt_fixe} and \ref{hyp_inclusions} are both satisfied.

{We should mention two facts which can complicate the use of $R_\sigma$  in the $p$-adic case. The first complication is that $R_\sigma$ can be nonabelian when $F$ is $p$-adic  \cite{Keys}, whereas $R_\sigma$ is always an elementary abelian $2$-group when $F = \mathbb{R}$.}

{The second complication, which will play an important role in this paper, is a ``cocycle problem''. Let $\mathcal{H}$ be a carrier space for the induced representation $\mathrm{Ind}_P^G(\sigma)$. For every nontrivial $r$ in $R_\sigma$, the Knapp-Stein theory furnishes a non-scalar self-intertwining operator $R(r, \sigma)\colon \mathcal{H} \to \mathcal{H}$. The map $r \mapsto R(r, \sigma)$ defines a projective representation of $R_\sigma$ on $\mathcal{H}$. The second complication for using the $R$-group in the $p$-adic case is that, in contrast to the situation for real groups, the projective representation cannot always be linearized: attached to it is a $2$-cocycle of $R_\sigma$ which does not always split.}

Arthur has found a way out of this difficulty in  \cite[\S 2]{ArthurActa}: there is a certain central extension $\widetilde{R}_{\sigma}$ of $R_\sigma$ over which the $2$-cocycle splits, and it is often convenient to work with  $\widetilde{R}_{\sigma}$ instead of $R_\sigma$ $-$ at the cost of twisting some of the objects that appear in the theory.

Drawing on previous work of Leung and Plymen \cite{LeungPlymen}, Opdam and Solleveld \cite{OpdamSolleveld}, we prove that under Assumptions \ref{hyp_pt_fixe} and \ref{hyp_inclusions}, Wassermann's theorem holds true after a slight twisting. Recall that the central extension $\widetilde{R}_{\sigma} \to R_\sigma$ comes with a central idempotent $\widetilde{p}$ acting on the group algebra of $\widetilde{R}_{\sigma}$ (see for instance \cite[\S 2.1]{ABPSOrsay}). We can then form the twisted crossed-product algebra $\mathcal{C}(\mathcal{O}/W'_{\sigma}) \rtimes \widetilde{R}_{\sigma}$, and $\widetilde{p}$ defines a central idempotent in it (see \cite{OpdamSolleveld}).

\begin{theo}\label{principal} Assume $F$ is a local field of characteristic zero, and $\sigma$ is a fixed point for the action of $W_\Theta$ on $\mathcal{O}$ (Assumption \ref{hyp_pt_fixe}) satisfying Assumption \ref{hyp_inclusions}. Then we have the strong Morita equivalence
\[ C^\ast_r(G, \Theta)  \underset{\text{Morita}}{\sim} \widetilde{p} \left[\ \mathcal{C}_0(\mathcal{O}/W'_{\sigma}) \rtimes \widetilde{R}_{\sigma}\ \right].\qedhere\]
 \end{theo}

As before, the theorem encodes both a parametrization of $\widehat{G}_\Theta$ (see e.g. \cite[Lemma 2.3]{ABPSOrsay}) and a description of its topology: it identifies $\widehat{G}_\Theta$ with a connected component of the spectral extended quotient $\left[(\mathcal{O}/W'_\sigma)//\widetilde{R}_\sigma\right]_{\mathrm{spec}}$, with the topology induced from that described in \cite[\S 4]{EchterhoffEmerson}.

Our way to Theorem \ref{principal} is one that has been precisely paved by Leung and Plymen in \cite[\S 2-3]{LeungPlymen}, using Plymen's earlier work on the general structure of the reduced $C^\ast$-algebra of $G$ and the role of standard intertwining operators to describe the reducibility of induced representations.

When specialized to $F=\mathbb{R}$, the central extension step is trivial, and we shall see that $\sigma$ may always be chosen so as to satisfy the two assumptions. Then  Theorem \ref{principal} specializes to Wassermann's theorem. Our proof will in fact fill in the details of Wassermann's announcement \cite{Wassermann}.

When specialized to $p$-adic $F$, Theorem \ref{principal} admits as special cases the known results \cite{Plymen1, LeungPlymen, JawdatPlymen, ChaoPlymen} described in the previous paragraph. To the best of our knowledge, all existing results for $C^\ast$-blocks of $p$-adic groups have the property that the Morita equivalence of Theorem \ref{principal} has either $W'_\sigma$ or $R_\sigma$  trivial (and that, in addition, the central extension $\widetilde{R}_\sigma$ of $R_\sigma$ is trivial). For classical $p$-adic groups, we shall see in Part \ref{part2} how to construct infinite families of pairs $(M, \sigma)$ for which Theorem \ref{principal} holds, and which have both  $W'_\sigma$ and $R_\sigma$ nontrivial (see the end of this Introduction for a family of examples in $\mathrm{Sp}(2n)$).

Let us give a quick indication of our strategy for Theorem \ref{principal}. An important ingredient in  our application of the Leung--Plymen method is the fact that the intertwining operators can be normalized so as to satisfy a certain twisted cocycle relation. The normalization which we use here is due to Langlands \cite{Langlands} and Arthur \cite{ArthurActa}, and does not use the $L$-functions of \cite{Shahidi}. We shall devote \S \ref{norma} to recalling the necessary facts on intertwining operators and proving the twisted cocycle relation, which is a simple consequence of remarks by Arthur \cite{ArthurActa} and Goldberg-Herb \cite{GoldbergHerb}. To apply the methods of \cite{LeungPlymen}, the complications related to central extensions of the $R$-groups make it necessary to straighten the operators so that they satisfy a genuine cocycle relation: it is only in \S \ref{r_groupe}, after recalling well-known facts on $R$-groups and Arthur's central extensions, that we will introduce the operators to which we will eventually apply the Leung--Plymen method. 

Once this is done, we shall obtain Theorem \ref{principal} by a rather direct application of the results in \cite{LeungPlymen}, suitably adapted (in the $p$-adic case) by using ideas of Opdam and Solleveld \cite{OpdamSolleveld} concerning some of the algebras that appear in the discussion. We shall devote \S \ref{fin_preuve} to concluding the proof.

\subsection*{On the existence of fixed points and good fixed points} Of course the scope of Theorem \ref{principal} drastically depends on the possibility that Assumption \ref{hyp_pt_fixe} and \ref{hyp_inclusions} could be satisfied in a number of interesting cases. In Part \ref{part2} of this paper, we consider concrete examples of groups $G$, Levi subgroups $M$ and discrete series representations $\sigma \in \mathcal{E}_2(M)$, and discuss the validity of Assumptions \ref{hyp_pt_fixe} and \ref{hyp_inclusions}. 

We shall mainly be concerned with the quasi-split classical groups (meaning symplectic, orthogonal or unitary groups), for then the Levi subgroups and the general shape of their discrete series representations are easily described. In addition, much about the reducibility of induced representations is known from the work of David Goldberg \cite{GoldbergSPN, GoldbergSLN, GoldbergUN}. We shall heavily borrow from his papers to describe the Weyl groups attached to the various Levi subgroups and connected components of $\widehat{G}_{\mathrm{temp}}$, and to perform the  $R$-group calculations necessary to study Assumption \ref{hyp_inclusions}. The classical groups which we will discuss have the further advantage that all $R$-groups are {abelian. In fact, Goldberg proves in  \cite{GoldbergSPN, GoldbergSLN, GoldbergUN} that for the $p$-adic groups we will consider in Part \ref{part2}, all $R$-groups are elementary abelian $2$-groups, just as happens for all real groups. This conveniently excludes more subtle kinds of $R$-groups that one can encounter in the $p$-adic case (and, first and foremost, for $\mathrm{SL}(n,F)$).}

The first result of Part \ref{part2}, to be established in \S \ref{exist_pt_fixe}, is that Assumption \ref{hyp_pt_fixe} turns out to be vacuous for the classical groups under discussion:
\begin{theo} \label{th_pt_fixe} Let $G$ be a quasi-split symplectic, orthogonal or unitary group. Then for every Levi subgroup $M$ of $G$ and for every discrete series representation $\sigma \in \mathcal{E}_2(M)$, the action of $W_{\Theta}$ on $\mathcal{O}$ has a fixed point. \end{theo}

However, we shall see that Assumption \ref{hyp_inclusions} is quite strong, and far from being always satisfied. We will point out in \S \ref{iwahori} that it already fails for the Iwahori-spherical block of the groups under discussion (as it happens, condition (b) also fails for the Iwahori-spherical block of $\mathrm{SL}(n,F)$: this was pointed out to us by Maarten Solleveld). For those components $\Theta$ for which Assumption \ref{hyp_inclusions} fails, we will in fact point out phenomena which make it unlikely that the structure of $C^\ast_r(G,\Theta)$ cannot be given, up to Morita equivalence, by any simple crossed product of the kind exhibited in good cases by Theorem \ref{principal}.

We shall devote \S \ref{bons_pts_fixes} to determining the pairs $(M,\sigma)$ for which Assumption \ref{hyp_inclusions} is satisfied, concluding in Theorem \ref{conclusion_inclu} with a necessary and sufficient condition in terms of the reducibility criteria on the various constituents of $\sigma$. These criteria are due to Goldberg \cite{GoldbergSPN} and ultimately depend on an analysis of the poles of $L$-functions (see the work of Shahidi \cite{Shahidi, ShahidiDuke} and Goldberg-Shahidi \cite{GoldbergShahidi}). 

For the moment, let us simply give an idea of our results, and therefore of the scope of Theorem \ref{principal}, for a simple family of examples. Consider $G=\mathrm{Sp}(2n)$, fix a Levi subgroup  $M$ isomorphic with $\mathrm{GL}(n_1) \times \dots \times \mathrm{GL}(n_r)$, assume that all integers $n_i$ are odd, and assume that $\sigma = \sigma_1 \otimes \dots \otimes \sigma_r$ is supercuspidal. Then

\begin{itemize}
\item Conditions (a) and (b) of Assumption \ref{hyp_inclusions} hold when there is no equivalence among those constituents $\sigma_i$ that are self-dual. For instance, assume that $\sigma_1, \dots, \sigma_k$ are inequivalent self-dual representations, and that $\sigma_{k+1} \simeq \dots \simeq \sigma_r$ are equivalent non-self-dual representations. Then conditions (a) and (b) hold, and we have $R_\sigma = (\mathbb{Z}/2\mathbb{Z})^k$ and $W'_{\sigma} = \mathfrak{S}_{r-k}$ (the symmetric group on $r-k$ letters). This provides infinite families of components $\Theta$ where Theorem \ref{principal} applies, and where both $R$ and $W'$ are non-trivial (and explore all possible $R$-groups). 

\item However, conditions (a) and (b) fail as soon as two among the self-dual $\sigma_i$ are equivalent. As we shall see, a consequence which perhaps gives an indication of the strength of (a)-(b) is that Theorem \ref{principal} can hold only in situations where the semidirect product $W_\sigma = W'_\sigma \rtimes R_\sigma$ is in fact a direct product. 
\end{itemize}

Let us conclude this Introduction with a few words about those blocks $C^\ast_r(G; \Theta)$ for which Assumption \ref{hyp_inclusions} fails. We cannot hope, for those, that a  simple crossed-product statement can elucidate the structure of $C^\ast_r(G;\Theta)$ as in Theorem \ref{principal}. It should be noted, however, that a related crossed-product-like statement is expected, and known in certain cases, to exist at the level of $K$-theory. Each component $C^\ast_r(G;\Theta)$ lies in a ``Bernstein block'' $C^\ast_r(G,\mathfrak{s})$, which comes with its own compact torus $T_{\mathfrak{s}}$ and its own Weyl group $W_{\mathfrak{s}}$ (see \cite[\S 2.2]{ABPSOrsay}). The $K$-theory of $C^\ast_r(G,\mathfrak{s})$ is then conjectured in \cite[Conjecture 5]{ABPSOrsay} to be isomorphic with (a twisted version of) $K_{W_{\mathfrak{s}}}(T_\mathfrak{s})$, the equivariant $K$-theory of the torus $T_{\mathfrak{s}}$ with respect to the finite group $W_{\mathfrak{s}}$. This can be viewed as a cohomological counterpart of Theorem \ref{principal}, and has the feel of a crossed product statement, albeit for the Bernstein component $C^\ast_r(G,\mathfrak{s})$, which is a union of blocks $C^\ast_r(G;\Theta)$, rather than for the individual blocks $C^\ast_r(G;\Theta)$. The $K$-theory statement has been shown by Solleveld (by a combination of results in \cite{Solleveld1, Solleveld2, Solleveld3}) to be valid in many cases where Theorem \ref{principal} badly fails, including the Iwahori-spherical block for split groups. For the spherical principal series of $\mathrm{SL}(n,F)$, Kamran and Plymen studied in \cite{KamranPlymen} the relationship with $C^\ast$-blocks. The related extended quotients are computed quite explicitly in \cite{NPW} for $\mathrm{SL}(n,F)$.

It would probably be enlightening to know if traces of that structure can be detected on the components $C^\ast_r(G;\Theta)$ for which Theorem \ref{principal} fails, in spite of the absence of an obvious analogue of the geometric structure valid for all components in the case of real groups, and for the components here identified in the $p$-adic case.

\subsection*{Acknowledgements} {We are deeply grateful to Roger Plymen and Maarten Solleveld, whose comments, corrections and suggestions were extremely helpful in preparing the current manuscript. We are also much indebted to the referees for providing very constructive feedback, which led to significant improvements and corrections. }

\part{Crossed-product structure in the presence of a good fixed point}\label{part1}

We begin with a local field $F$ of characteristic zero and a reductive group $G$ over $F$, and take up the notation of the Introduction. Along our way towards a proof of Theorem \ref{principal}, we shall keep $F$ arbitrary (real or $p$-adic) for as long as possible. It is only at the very end, in \S \ref{preuve_wass} and \S \ref{preuve_padique}, that we shall specialize either to $F = \mathbb{R}$ (thereby filling in the details of Wassermann's Theorem \ref{th_wass}) or to $p$-adic $F$ (where our main Theorem \ref{principal} is really new). 

\section{Normalized intertwining operators and cocycle relations}\label{norma}

We here collect well-known material concerning the standard intertwining operators between induced representations, and a normalization of these intertwining operators due to Langlands \cite{Langlands} and Arthur \cite{ArthurActa}. We shall point out that the Langlands-Arthur normalization leads to a ``twisted'' cocycle relation (see Proposition \ref{cocycle_twiste}). Many of our remarks follow the detailed discussion of Goldberg and Herb \cite[pp. 114-123]{GoldbergHerb}.

\subsection{Intertwining operators and the Langlands-Arthur normalization} \label{induites}

\paragraph*{Induced representations}

{Consider a Levi subgroup $M$ of $G$, and a discrete series representation $\sigma$ of $M$: the representation $\sigma$ is irreducible, unitary, and all its matrix elements are square-integrable modulo the center of $M$. } Henceforth we shall fix a carrier space $V$ for $\sigma$. Whenever $P=MN$ is a parabolic subgroup of $G$ with Levi factor $M$, we write $\delta_P\colon P \to \mathbb{R}^+ $ for the modular function of $P$ (see \cite[\S II.3.7 and \S V.3.4]{Renard}). {We write $\mathcal{X}(M)$ for the set of unramified characters of $M$, and $\mathcal{X}_u(M)$ for the subset of unitary unramified characters.} For every $\chi \in \mathcal{X}_u(M)$, we introduce the space {$E_{P}(\sigma \otimes \chi)$ of {functions $f\colon G \to V$ that are right-invariant under some compact open subgroup of $G$ if $F$ is $p$-adic, that are smooth if $F = \mathbb{R}$, and which satisfy} 
\[ f(xmn)=\delta^{-\frac{1}{2}}_P(m)(\sigma\otimes \chi)(m)^{-1}f(x), \text{ \quad for $x \in G$, $m \in M$, $n \in N$}. \]
{We consider the representation $\mathcal{I}_P(\sigma \otimes \chi)$ of $G$ on $E_{P}(\sigma \otimes \chi)$ by left translation.}

{Let $K$ be a maximal compact  subgroup of $G$, equipped with the Haar measure of total mass one. For $p$-adic $F$, we may and will assume that $K$ is in fact a good special maximal compact subgroup and that $K$ and $M$ are in ``good relative position'' (see \cite[\S V.5.1]{Renard}). For every parabolic subgroup $P$ of $G$ with Levi factor $M$, we then have $G=KP$. We denote by $\mathcal{H}_P(\sigma \otimes \chi)$ for the Hilbert space completion of $E_P(\sigma \otimes \chi)$ with respect to the inner product $\langle f, f'\rangle = \int_K \langle f(k), f'(k) \rangle_{{V}}dk$. We still write $\mathcal{I}_P(\sigma \otimes \chi)$ for the representation of $G$ on $\mathcal{H}_{P}(\sigma \otimes \chi)$.}

We shall need the realization of induced representations in the ``compact picture''. {Let $E^K_P$ be the space of functions $f: K \to V$ that are right-invariant under some open subgroup of $K$ if $F$ is $p$-adic, that are  smooth if $F=\mathbb{R}$, and which satisfy }
\begin{equation} \label{h_k_p} f(kmn)=\sigma(m)^{-1}f(k), \quad \text{ for $k \in K$, $m \in M\cap K$, $n \in N\cap K$}.\end{equation}
Let us consider the Hilbert completion $\mathcal{H}^K_P$ of  $E^K_P$ for the inner product above. The representation of $G$ on $\mathcal{H}^K_P$ that we will consider is the one obtained by transferring the above representation on $ \mathcal{H}_{P}(\sigma \otimes \chi)$ through the linear isomorphism }
\[ F_{P}^K(\chi)\colon \mathcal{H}_{P}(\sigma \otimes \chi) \to \mathcal{H}^K_P\]
where, given $f \in \mathcal{H}_{P}(\sigma \otimes \chi)$, the map $F_P^K(\chi) f$ is just the restriction of $f$ to $K$. 

The inverse map $F_{P}^K(\chi)^{-1}\colon  \mathcal{H}^K_P \to  \mathcal{H}_{P}(\sigma \otimes \chi)$ is specified by the formula
\[ \forall \varphi \in \mathcal{H}^K_P, \quad \forall x \in G, \quad \left[F_{P}^K(\chi)^{-1}\varphi\right](x)=\delta_P^{-\frac{1}{2}}(\mu(x)) (\sigma \otimes \chi)(\mu(x))^{-1} \varphi(\kappa(x))\]
where, for $x \in G=KMN$, we isolate $\kappa(x) \in K$ and $\mu(x) \in M$ so that $x \in \kappa(x) \mu(x) N$. As a result, the representation $\mathcal{I}_{P}^K(\sigma \otimes \chi)$ of $G$ on $\mathcal{H}^K_P$ has $g \in G$ act on $\mathcal{H}^K_P$ through 
\[ \mathcal{I}_{P}^K(\sigma \otimes \chi)(g)= F^K_P(\chi) \mathcal{I}_P(\sigma \otimes \chi)(g) F^{K}_P(\chi)^{-1}.\]

\paragraph*{Intertwining operators}

When $P=MN$ and $P'=MN'$ are two parabolic subgroups of $G$ with Levi factor $M$, we can consider the standard intertwining operator
\[ J_{P' | P}(\sigma \otimes \chi)  \colon  \mathcal{H}_{P}(\sigma \otimes \chi)  \to   \mathcal{H}_{P'}(\sigma \otimes \chi).\]
{When $\chi$ is in a certain region of $\mathcal{X}(M)$, the definition is} 
\begin{equation} \label{formule_integrale}  J_{P' | P}(\sigma \otimes \chi)\colon f \mapsto \left[ x \mapsto \displaystyle \int_{\overline{N} \cap N'} f(x\overline{n})d\overline{n}\right] \end{equation}
where $M\overline{N}$ is the parabolic opposite to $MN$, and $d\overline{n}$ is the normalized Haar measure over $\overline{N} \cap N'$. {For more general $\chi$, the map $\chi \mapsto J_{P' | P}(\sigma \otimes \chi)$  is  obtained by analytic continuation of the above formula to a rational function $\chi \mapsto J_{P' | P}(\sigma \otimes \chi)$. }

We can use the isomorphism $F_{P}^K(\chi)$ to transfer $J_{P' | P}(\sigma \otimes \chi)$ to an operator acting on $\mathcal{H}^K_P$, setting
\[\begin{array}{ccccc}
J^K_{P'|P}(\sigma \otimes \chi) & : & \mathcal{H}_{P}^K& \to &  \mathcal{H}_{P'}^K \\ & & \varphi & \mapsto & \left[ k \mapsto \displaystyle \int_{\overline{N} \cap N'} \delta_P^{-1/2}(\mu(k\overline{n})) (\sigma \otimes \chi)^{-1}(\mu(k\overline{n})) \varphi(\kappa(k\overline{n}))d\overline{n}\right].
\end{array}
\]

\paragraph*{Normalization}

If $P=MN$ and $P'=MN'$ are two parabolic subgroups with Levi factor $M$ as above,  we recall that Langlands and Arthur introduced scalar normalizing factors 
\[ r_{P'|P}(\sigma \otimes \chi), \quad \chi \in \mathcal{X}_u(M)\]
(see \cite[Theorem 2.1]{ArthurJFA}). We shall use the normalized operators
\begin{align} \label{norm_std} 
R_{P'|P}(\sigma \otimes \chi) & : & \mathcal{H}_{P}(\sigma \otimes \chi) & \to   \mathcal{H}_{P'}(\sigma \otimes \chi)\\ 
& & f & \mapsto    r_{P'|P}(\sigma \otimes \chi)^{-1} J_{P'|P}(\sigma \otimes \chi) f \nonumber
\end{align}
and their ``compact picture'' counterparts
\begin{align} \label{norm_int}
R^K_{P'|P}(\sigma \otimes \chi) & : & \mathcal{H}_{P}^K &  \to   \mathcal{H}^K_{P'}\\
& & \varphi & \mapsto    r_{P'|P}(\sigma \otimes \chi)^{-1} J^K_{P'|P}(\sigma \otimes \chi) \varphi.\nonumber
\end{align}
The two are related through 
\begin{equation} \label{retour_compact} R^K_{P'|P}(\sigma \otimes \chi) = F_{P'}^K(\chi) \circ R_{P'|P}(\sigma \otimes \chi) \circ F_{P}^K(\chi)^{-1}.\end{equation}
{These normalized operators satisfy the conditions of \cite[\S 2]{ArthurJFA}, and we will use the results of  \cite[\S 2]{ArthurJFA} and  \cite[\S 2]{ArthurActa}. The map $\chi \mapsto R^K_{P'|P}$ is holomorphic on $\mathcal{X}_u(M)$ (see e.g. \cite[\S 1.5]{WaldspurgerGGP2}). }

\subsection{The operators $B_P(w)$ and the conjugation relation with the intertwining operators}\label{subsec:reverting}

Recall that we are working under the assumption that $\sigma$ is a fixed point for the action of $W_{\Theta}$ on the orbit $\mathcal{O}$. For each $w \in W_\Theta$ and every $\chi \in \mathcal{X}_u(M)$, we can use the intertwining operators in \eqref{norm_int} to obtain an intertwining operator between $\mathcal{I}_P^K(\sigma \otimes \chi)$ and $\mathcal{I}_{w^{-1}Pw}^K(\sigma \otimes (w\chi))$. In order to take up the Leung--Plymen approach to Wassermann's theorem in \S \ref{fin_preuve}, we will need to convert \eqref{norm_int} into self-intertwining operators, and we recall the classical way to do so (see \cite[\S 2]{ArthurActa} and \cite[\S 4]{GoldbergHerb}).

For each $w \in W_\Theta$, we fix a representative $\dot{w}$ of $w$ in $N_G(M) \cap K$. The fact that the representative can be chosen in $K$ comes from the fact that $K$ and $M$ are in good relative position. Since we have assumed $\sigma$ to be a fixed point, the representation $^{\dot{w}} \sigma: m \to \sigma(\dot{w}^{-1} m w)$ of $M$ is unitarily equivalent with $\sigma$. Therefore, there exists a unitary operator $T_{\dot{w}}: V \to V$ which satisfies
\[ T_{\dot{w}} \ \sigma(\dot{w}^{-1} m \dot{w}) = \sigma(m) \ T_{\dot{w}} \quad \text{ for all $m \in M$}.\]
We then introduce an operator $\sigma^N(\dot{w}): V \to V$, setting 
\begin{equation} \label{sigma_N} \sigma^N(\dot{w}) = T_{\dot{w}}.\end{equation}
We further define $\sigma^N(m)=\sigma(m)$ for all $m$ in $M$, thereby obtaining a map
\begin{equation} \label{rep_projective} \sigma^N: \quad N_G(\Theta) \to \mathrm{End}(V),\end{equation}
where $N_G(\Theta)$ is the subgroup of $N_G(M)$ defined in the first paragraph of the Introduction.
In general, this map does \emph{not} define a representation of $N_G(\Theta)$. But it always defines a projective representation of $W_\Theta$. Indeed, whenever $w_1$ and $w_2$ are elements of $W_\Theta$, the operator $\sigma^N(\dot{w}_1 \dot{w}_2) \sigma^N(\dot{w}_1) \sigma^N(\dot{w}_2)$ is a self-intertwiner for $\sigma$. Since $\sigma$ is assumed to be irreducible, we obtain 
 \begin{equation} \label{schur} \sigma^N(\dot{w}_1\dot{w}_2) = \eta_\sigma(w_1,w_2) \sigma^N(\dot{w}_1) \sigma^N(\dot{w}_2)\end{equation}
where $\eta_\sigma(w_1,w_2)$ is a complex scalar, which depends only on the elements $w_1$, $w_2$ and not on the choices of representatives $\dot{w}_1$, $\dot{w}_2$. As an immediate consequence of \eqref{schur}, we know that the map $\eta_\sigma\colon W_\Theta \times W_\Theta \to \mathbb{C}^\times$ is a $2$-cocycle.

\begin{rema} If the operators  $T_{\dot{w}}$, for $w \in W_\Theta$, can be chosen in such a way that $\sigma^N$ defines a genuine representation of $N_G(\Theta)$, then $\eta_\sigma$ is trivial. This is always the case when $G$ is a quasi-split classical group, as the reader will easily check from the explicit description of Levi subgroups and of groups $W_\Theta$ in Part II. See the remarks in \cite[\S 1.5]{WaldspurgerGGP2}. \end{rema}

Let us now see how the cocycle $\eta_\sigma$ enters the discussion of standard intertwining operators. For every parabolic subgroup $P$ with Levi factor $M$, and for every element $w$ in $W_{\Theta}$, we introduce 
 
%Cependant, avec la relation \eqref{sigma_N}, on constate que pour tout $w \in W_\Theta$, lorsqu'on restreint $\sigma^N$ au sous-groupe engendré par $M$ et $\dot{w}$, on obtient bien une représentation de ce sous-groupe.  C'est ce qui explique le lien avec les présentations d'Arthur et de Goldberg-Herb~: dans notre texte, nous notions $\sigma_w(n_w)$ pour ce que je note ici $\sigma^N(\dot{w})$. 

%
%  and extend  $\sigma$ to a representation $\sigma_{w}$ of the group generated by  $M$ and $n_w$. We then introduce the operator
%\begin{equation} \label{a_p} \begin{array}{ccccc}
%A_P(w) & : & \mathcal{H}_{w^{-1}Pw}(\sigma \otimes \chi) & \to &  \mathcal{H}_{P}\left[\sigma \otimes (w\chi)\right]  \\ & & f & \mapsto & \left[ x \mapsto \sigma_w(n_w) f(xn_w)\right].
%\end{array}
%\end{equation}
%It intertwines the representations $\mathcal{I}_{w^{-1}Pw}(\sigma \otimes \chi)$ and $\mathcal{I}_{P}\left[\sigma \otimes(w \chi)\right]$, for all $\chi \in \mathcal{X}_u(M)$. 
%
%Now, we may and will assume that the choice of representatives $n_w$, $w \in W_\Theta$, has been made so that they all lie in $K$.  We can then transfer the above operators to the compact realization, obtaining the operator
\begin{equation} \label{b_p}\begin{array}{ccccc}
B_P(w) & : & \mathcal{H}^K_{w^{-1}Pw} & \to &  \mathcal{H}^K_{P}  \\ & & \varphi & \mapsto & \left[ k \mapsto \sigma^N(\dot{w}) \varphi(\dot{w}^{-1} k)\right].
\end{array}
\end{equation}

This operator does not depend on the choice of representative $\dot{w}$ for $w$, because of the transformation law \eqref{h_k_p} for elements $\varphi \in \mathcal{H}^K_{P}$. Furthermore, for every $\chi  \in \mathcal{X}_u(M)$, the operator $B_P(w)$ intertwines  $\mathcal{I}^K_{w^{-1}Pw}(\sigma \otimes \chi)$ and $\mathcal{I}^K_{P}\left(\sigma \otimes(w \chi)\right)$.

We can summarize the above discussion in the following statement:

\begin{lemm} \label{retour} The relation $B_P(w_1w_2) = \eta_\sigma(w_1,w_2) B_P(w_1) B_{w_1^{-1}Pw_1}(w_2)$ holds for every $w_1$, $w_2$ in $W_\Theta$. The map $\eta_\sigma: W_\Theta \times W_\Theta \to \mathbb{C}^\times$, introduced in \eqref{schur}, is a $2$-cocycle of the finite group $W_{\Theta}$.\end{lemm} 

We now prove a conjugation relation between the operators $B_P(w)$ of \eqref{b_p} and the intertwining operators \eqref{norm_int}.

%
%
%Goldberg and Herb, in \cite[\S 4]{GoldbergHerb}, detailed the relationship between the operators $A_P(w)$ in \eqref{a_p} and the intertwining operators in \eqref{norm_std}. 
%
%(We mention that although Goldberg and Herb work with $p$-adic groups  in \cite{GoldbergHerb}, the results of \cite[\S 4]{GoldbergHerb} used below are based only on properties which apply verbatim to real groups). 
%
%
%
%We shall use two of their lemmas.  The first  is an immediate transcription of  \cite[Lemma 4.11]{GoldbergHerb} to the operators $B_P(w)$:
%
%
%\begin{lemm}[\cite{GoldbergHerb}] \label{retour} 
%\begin{itemize}
%\item[$\bullet$] For every $w \in W_\Theta$, the operators $A_P(w)$ and $B_P(w)$ are independent of the choice of representative $n_w$ in $N_G(M)\cap K$. 
%\item[$\bullet$] For each $w_1$ and every $w_2$ in $W_\Theta$,  there exists a complex number  $c_P(w_1,w_2)$ such that 
%\[ B_P(w_1w_2) = c_P(w_1,w_2) B_P(w_1) B_{w_1^{-1}Pw_1}(w_2).\qedhere\]
%\end{itemize}
%\end{lemm}
%
%The second is a twisted conjugation relation between the operators $B_P(w)$ of \eqref{b_p} and the intertwining operators \eqref{norm_int}, taken from the proof of \cite[Lemma 4.12]{GoldbergHerb}.
%
\begin{lemm} \label{conjretour} Let $w_1, w_2$ be elements of $W_\Theta$. Define three parabolic subgroups with Levi factor $M$ by setting $P_{1} = w_1^{-1}Pw_1$, $P_2 = w_2^{-1}Pw_2$ and $P_{12}=(w_1w_2)^{-1}P(w_1w_2)$. For every $\chi \in \mathcal{X}_u(M)$, we have 
\[ B_{P_1}(w_2) R^K_{P_{12}|P_2}(\sigma \otimes \chi)B_{P}(w_2)^{-1} = R^K_{P_1|P}(\sigma \otimes (w_2 \chi)). \qedhere\]
\end{lemm}

\begin{proof} We shall use the properties discussed by Arthur in Section 2 of \cite{ArthurJFA}. The proof is a simple, but somewhat tedious, calculation. It is very close to that in \cite[Lemma 4.12]{GoldbergHerb}.

We begin by remarking that the operator $B_P(w)$ is built from two ingredients: left translation of functions by $\dot{w}$, and pointwise application of the operator $\sigma(\dot{w})$. We introduce further notation to separate the two ingredients:
\begin{itemize}
\item In order to discuss left translation, we introduce 
\begin{equation} \label{b_p}\begin{array}{ccccc}
\ell_P(\dot{w}) & : & \mathcal{H}^K_{w^{-1}Pw} & \to &  \mathcal{H}^K_{P}  \\ & & \varphi & \mapsto & \left[ k \mapsto \varphi(\dot{w}^{-1} k)\right]
\end{array}\end{equation}
whenever  $P$ is a parabolic subgroup of $G$ with Levi factor $M$, $w$ is an element of $W_\Theta$, and  $\dot{w} \in K\cap N_G(M)$ is a representative for $w$. 

In  \cite[Theorem 2.1]{ArthurJFA}, the notation leaves $P$ implicit. But with the above convention, property $(R_5)$ in   \cite[Theorem 2.1]{ArthurJFA} is:
\begin{equation}  \label{r_5} \ell_{P'}(\dot{w}) R^K_{P'|P}(\sigma \otimes \chi)\ell_P(\dot{w})^{-1} = R^K_{wP'w^{-1}\ | \ wPw^{-1}}(\sigma \otimes (w\chi))\end{equation}
for all $P$, $P'$, all $w$ and all $\chi$. 
\item As for pointwise multiplication by $\sigma^N(\dot{w})$, we denote it as follows: for  $f\colon K \to V$ in $\mathcal{H}_P^K$, we set
\begin{equation} \label{pointwise}
\mathrm{mult}_{\sigma^N(\dot{w})} f  \quad :  \quad  \text{the map $g \mapsto \sigma^N(\dot{w}) f(g)$.} \end{equation}
This defines a map $\mathrm{mult}_{\sigma^N(\dot{w})}\colon \mathcal{H}_P^K \to \mathcal{H}_P^K$. Looking back to \S \ref{induites}, notice that formula \eqref{pointwise} also has a meaning when $f$ is a function on $G$ (rather than on $K$); therefore it defines an operator  $ \mathcal{H}_P(\sigma \otimes \chi) \to \mathcal{H}_P(\sigma \otimes \chi)$, which we will still denote by $\mathrm{mult}_{\sigma^N(\dot{w})}$. Of course, the operator $\mathrm{mult}_{\sigma^N(\dot{w})}$ commutes with restriction to $K$: using the notation from \S \ref{induites}, we have 
\begin{equation} \label{comm} \mathrm{mult}_{\sigma^N(\dot{w})} \circ F_P^K(\chi) = F^K_P(\chi)\circ \mathrm{mult}_{\sigma^N(\dot{w})} \end{equation}
for each $P$ and every $\chi$. 
\item With all this notation established, we of course have 
\begin{equation} B_P(w) = \mathrm{mult}_{\sigma^N(\dot{w})} \circ \ell_{w^{-1}Pw}(\dot{w}) = \ell_{w^{-1}Pw}(\dot{w})\circ \mathrm{mult}_{\sigma^N(\dot{w})}.\end{equation}
\end{itemize}

We can now prove Lemma \ref{conjretour}. by unpacking the definitions of the left-hand side of the desired equality. Bringing in the above ingredients, we have : 
\begin{align*} 
B_{P_1}(w_2) R^K_{P_{12}|P_2}(\sigma \otimes \chi)B_{P}(w_2)^{-1} &= \ell_{w_2^{-1}P_1w_2}(w_2)\circ\mathrm{mult}_{\sigma^N(\dot{w_2})} \circ R^K_{P_{12}|P_2}(\sigma \otimes \chi) \circ\mathrm{mult}_{\sigma^N(\dot{w_2})}^{-1} \circ \ell_{w_2^{-1}Pw_2}(w_2)^{-1}. \end{align*}
We now use the expression of the middle term $R^K_{P_{12}|P_2}(\sigma \otimes \chi)$ using the restriction operators, as in \eqref{retour_compact}. We find that the above operator is equal to 
\[ \ell_{P_{12}}(w_2) \circ \mathrm{mult}_{\sigma^N(\dot{w_2})} \circ F_{P_{12}}^K(\chi)\circ R_{P_{12}|P_2}(\sigma \otimes \chi) \circ F_{P_2}(\chi)^{-1} \circ \mathrm{mult}_{\sigma^N(\dot{w_2})}^{-1} \circ \ell_{P_2}(w_2)^{-1}.\]
The second term here commutes with the third, because of \eqref{comm}. So our operator is equal to
\[ \ell_{P_{12}}(w_2) \circ F_{P_{12}}^K(\chi)\circ  \mathrm{mult}_{\sigma^N(\dot{w_2})} \circ R_{P_{12}|P_2}(\sigma \otimes \chi) \circ F_{P_2}(\chi)^{-1} \circ \mathrm{mult}_{\sigma^N(\dot{w_2})}^{-1} \circ \ell_{P_2}(w_2)^{-1}. \]
Now, we remark that $\mathrm{mult}_{\sigma^N(\dot{w})}$ commutes with $R_{P_{12}|P_2}(\sigma \otimes \chi)$. This is obvious from the integral formula \eqref{formule_integrale} when $\chi\in \mathcal{X}(M)$ is in the region of convergence for the integrals; by analytic continuation, it is true for all $\chi \in \mathcal{X}_u(M)$.   Therefore, we can swap the third and fourth terms above, and obtain the expression 
\[ \ell_{P_{12}}(w_2) \circ F_{P_{12}}^K(\chi)\circ   R_{P_{12}|P_2}(\sigma \otimes \chi)\circ  \mathrm{mult}_{\sigma^N(\dot{w_2})} \circ F_{P_2}(\chi)^{-1} \circ \mathrm{mult}_{\sigma^N(\dot{w_2})}^{-1} \circ \ell_{P_2}(w_2)^{-1} \]
for our operator. Still using \eqref{comm}, we can swap the second-to-last term and the third-to-last, and simplify; we find that the left-hand-side for the equality of Lemma \ref{conjretour} is equal to 
\[  \ell_{P_{12}}(w_2) \circ F_{P_{12}}^K(\chi)\circ  R_{P_{12}|P_2}(\sigma \otimes \chi) \circ F_{P_2}(\chi)^{-1} \circ  \ell_{P_2}(w_2)^{-1}.\]
At this point, we revert back to the compact picture, and use \eqref{conjretour} to find that the latter operator is equal to 
\[ \ell_{P_{12}}(w_2) \circ R^K_{P_{12}|P_2}(\sigma \otimes \chi)  \ell_{P_2}(w_2)^{-1}.\]
We finally apply the key property  \eqref{r_5}, together with the definitions of $P_{12}$, $P_1$ and $P_2$, and conclude that
\begin{align*} 
B_{P_1}(w_2) R^K_{P_{12}|P_2}(\sigma \otimes \chi)B_{P}(w_2)^{-1} &=  R^K_{w_2 P_{12}w_2^{-1}|w_2P_2w_2^{-1}}(\sigma \otimes(w_2 \chi))  \\
&=   R^K_{P_1|P}(\sigma \otimes(w_2 \chi)). 
\end{align*}
This proves the Lemma. \end{proof}
%This is stated formally in \cite{GoldbergHerb} only for  $\chi = 1$ (and for the operators $A_P(w)$), but the proof of Lemma 4.12 in \cite{GoldbergHerb} yields the above formula, in which the scalar $c'_{P,\chi}(w_1,w_2)$ can be expressed as
%\[ c'_{P,\chi}(w_1,w_2) = \frac{r_{P_1|P}(\sigma \otimes (w_2\chi))}{r_{P_{12}|P_2}(\sigma \otimes \chi)}\] 
%for $\chi$ in the domain of convergence of the operators $J_{P'|P}(\sigma \otimes \chi)$, and then by analytic continuation for all~$\chi$.  \qed
%
%
%

\subsection{The twisted $1$-cocycle relation}
 For each $w \in W_\Theta$ and every $\chi \in \mathcal{X}_u(M)$, we now define
\[ \mathcal{A}(w,\sigma \otimes \chi) = B_P(w) R^K_{w^{-1}Pw | P}(\sigma \otimes \chi).\]
It acts on a space that depends neither on $w$ nor on $\chi$ : 
\[\begin{array}{ccccc}
\mathcal{A}(w,\sigma \otimes \chi) & : & \mathcal{H}^K_{P} & \to &  \mathcal{H}^K_{P}. \end{array}
\]
Furthermore, the operator $\mathcal{A}(w,\sigma \otimes \chi)$ intertwines $\mathcal{I}_P^K(\sigma \otimes \chi)$ and $\mathcal{I}_P^K(\sigma \otimes (w\chi))$.

We can combine the properties of normalized intertwining operators \cite[Theorem 2.1]{ArthurJFA} with Lemma \ref{conjretour} above, and obtain the following ``twisted'' cocycle relation. 

\begin{prop} \label{cocycle_twiste} For every $w_1$, $w_2$ in $W_\Theta$ and $\chi \in \mathcal{X}_u(M)$, we have 
\begin{equation} \label{cocycle_shahidi} \mathcal{A}(w_1w_2, \sigma \otimes \chi) = \eta_{\sigma}(w_1,w_2) \mathcal{A}(w_1, \sigma \otimes (w_2\chi)) \mathcal{A}(w_2, \sigma \otimes \chi)\end{equation}
where $\eta_{\sigma}$ is the cocycle of $W_\Theta$ introduced in \eqref{schur}.
 \end{prop}
 
\begin{proof} Define $P_1 = w_1^{-1}Pw_1$, $P_2=w_2^{-1}Pw_2$ and $P_{12} = w_2^{-1}w_1^{-1}Pw_1w_2$. By \cite[Theorem 2.1, Property $(R_2)$]{ArthurJFA}, we have 
\[ R^K_{P_{12}|P}(\sigma \otimes \chi) = R^K_{P_{12}|P_2}(\sigma \otimes \chi) R^K_{P_2|P}(\sigma \otimes \chi).\]
We now apply $B_P(w_1w_2)$ on the left, and obtain information on $\mathcal{A}(w_1w_2, \sigma \otimes \chi)$: we have
 \begin{align*}
 \mathcal{A}(w_1w_2, \sigma \otimes \chi)
&=  B_P(w_1w_2) R^K_{P_{12}|P_2}(\sigma \otimes \chi) R^K_{P_2|P}\sigma \otimes \chi)\\
&= \eta_\sigma(w_1,w_2) B_P(w_1) \left[ B_P(w_2) R^K_{P_{12}|P_2}(\sigma \otimes \chi) \right] R^K_{P_2|P}(\sigma \otimes \chi) \\ & \quad \text{(using Lemma \ref{retour}) }\\
&= \eta_\sigma(w_1,w_2) B_P(w_1) \left[  R^K_{P_{1}|P}( \sigma \otimes (w_2\chi)) B_P(w_2) \right] R^K_{P_2|P}(\sigma \otimes \chi) \\ & \quad \text{(using Lemma \ref{conjretour})}\\
&= \eta_\sigma(w_1,w_2)  \left[   B_P(w_1) R^K_{P_{1}|P}( \sigma \otimes (w_2\chi)) \right]  \left[ B_P(w_2) R^K_{P_2|P}(\sigma \otimes \chi) \right] \\  
&=  \eta_{\sigma}(w_1,w_2)  \mathcal{A}(w_1, \sigma \otimes (w_2\chi)) \mathcal{A}(w_2, \sigma \otimes \chi). \quad\quad\quad\quad\quad\quad\quad\quad\quad \qedhere
\end{align*}
%where we set
%\begin{equation} \label{def_eta} \eta_{\sigma \otimes \chi}(w_1,w_2) = c_P(w_1,w_2) c'_{P,\chi}(w_1,w_2).\end{equation}
%This is the announced twisted conjugation relation. The fact that $\eta$ is a $2$-cocycle of $W_\Theta$ can be seen by specializing the above to $\chi=1$ and expressing $\mathcal{A}(w_1w_2w_3,\sigma)$ in the obvious two ways. 
\end{proof}

\section{Analytic $R$-groups and central extensions} \label{extension_centrale} \label{r_groupe}

\subsection{The $R$-group \cite{SilbergerKnappStein}} \label{defs_rgroupe} Let $\chi$ be a unitary unramified character of $M$, and let $W_{\sigma \otimes \chi} \subset W_\Theta$ be the stabilizer of  $\sigma \otimes \chi$ in the action of $W_\Theta$ on $\mathcal{O}$. We can consider the subgroup 
\[  W'_{\sigma \otimes \chi} = \left\{ w \in W_{\sigma \otimes \chi} \ : \ \mathcal{A}(w, \sigma \otimes \chi) \text{ is scalar}\right\}.\]
We recall that $W'_{\sigma \otimes \chi}$ is the Weyl group of a root system, which can be defined through the zeroes of the Plancherel measure \cite{SilbergerRgroupe}. 

Let $\mu\colon \mathcal{O} \to \mathbb{R}^+$ be the Plancherel measure (see \cite{Waldspurger}), and let $\Delta$ be the set of roots for $(G,M)$. To each root $\alpha \in \Delta$ is attached a maximal Levi subgroup $M_\alpha \subset G$ that contains $M$ as a Levi subgroup of its own, and  we can call in the associated Plancherel measure $\mu_{M_\alpha} : \mathcal{O} \to \mathbb{R}^+$. A root $\alpha \in \Delta$ is called $(\sigma \otimes \chi)$-useful when $\mu_\alpha(\sigma \otimes \chi)=0$; the set $\Delta'$ of $(\sigma\otimes \chi)$-useful roots it then itself a root system, and the group $W'_{\sigma \otimes \chi}$ is its Weyl group. 
We now fix a positive system $\Delta'_+$ in $\Delta'$, and introduce the $R$-group
\[  R_{\sigma \otimes \chi} = \left\{ w \in W_{\sigma \otimes \chi} \ : \ w(\Delta'_+) = \Delta'_+\right\}. \]
We can then write the stabilizer  $W_{\sigma \otimes \chi} $ as a semidirect product
\begin{equation} W_{\sigma \otimes \chi} = W'_{\sigma \otimes \chi} \rtimes R_{\sigma \otimes \chi}, \end{equation}
where $W'_{\sigma \otimes \chi}$ is normal and $R_{\sigma \otimes \chi}$ acts on $W'_{\sigma \otimes \chi}$ by conjugation.

\subsection{Arthur $2$-cocycles and central extension} A consequence of Proposition \ref{cocycle_twiste} is that {for every point $\tau = \sigma \otimes \chi$ on the orbit $\mathcal{O}$,} the map
\[ r \mapsto \mathcal{A}(r, \tau)\]
defines a projective representation of {$R_\tau$} on $\mathcal{H}^K_P$. The multiplier of this projective representation is the restriction $\eta_\tau: R_\tau \times R_\tau \to \mathbb{C}^\times$ of the cocycle $\eta_{\sigma}$ of \eqref{schur}. 

The theory of the $R$-group can be brought to full fruition once we choose, as in  \cite[\S 2]{ArthurActa}, a central extension 
\[ 1 \to Z_\tau \to \widetilde{R}_\tau \overset{p}{\to} R_\tau \to 1\]
with the property that the $2$-cocycle $\eta_\tau$ of  $\widetilde{R}_\tau$ is a coboundary. Such a central extension exists, and we can then choose a map
\[ \xi_\tau\colon \widetilde{R}_\tau \to \mathbb{C}^\star\]
that splits $\eta_\tau$, in that:
\begin{equation} \label{cobord} \forall w_1, w_2 \in \widetilde{R}_\tau, \quad \eta_\tau(w_1,w_2) = \frac{\xi_\tau(w_1w_2)}{\xi_\tau(w_1)\xi_\tau(w_2)}.\end{equation}
 The representation $\tilde{r} \mapsto \xi_\tau(\tilde{r})^{-1} \mathcal{A}(p(\tilde{r}), \tau)$ is then a genuine representation of $\widetilde{R}_\tau$ on $\mathcal{H}_P^K$. Its central character on $Z_\tau$ is the map $\zeta_\tau\colon Z_\tau \to \mathbb{C}^\star$ given by $\zeta_\tau(z)= \xi_\tau(z)^{-1}$ for $z \in Z_\tau$.

\begin{theo}[\cite{ArthurActa}, \S 2] \label{quasi_arthur} \begin{enumerate}
\item The irreducible components of $\mathcal{I}^K_{P}(\tau)$ stand in natural bijection with the set of irreducible representations of $\widetilde{R}_\tau$ with $Z_\tau$-central character $\zeta_\tau$.
\item The representation $\tilde{r} \mapsto \xi_\tau(\tilde{r})^{-1} \mathcal{A}(p(\tilde{r}), \tau)$ of $\widetilde{R}_\tau$ on $\mathcal{H}_{P}^K$ is quasi-equivalent with the induced representation $\mathrm{Ind}_{Z_\tau}^{\widetilde{R}_\tau}(\zeta_\tau)$.\qedhere
\end{enumerate}
\end{theo}
 
(For (2), we recall that  ``quasi-equivalent'' means that the two representations have the same irreducible constituents, though these may occur with different multiplicities.)\\

Returning to the fixed point $\sigma$, we close this preliminary discussion with a remark on the behaviour of intertwining operators over the complementary subgroup $W'_{\sigma}$. As Arthur points out  on page 91 of \cite{ArthurActa}, it is possible to normalize the choice of operators $T_{\dot{w}}$ (from \S \ref{subsec:reverting}) in such a way that the scalar operators $\mathcal{A}(w,\sigma)$, for $w$ in the subgroup  $W'_{\sigma}$, are all equal to the identity. Upon introducing the extended group
\begin{equation} \label{wtilde} \widetilde{W}_{\Theta} = W'_{\sigma} \rtimes \widetilde{R}_{\sigma}\end{equation}
associated with the action of $\widetilde{R}_{\sigma}$ on $W'_{\sigma}$ {(where $Z_\sigma$ acts trivially),} %\footnote{Il s'agit de l'action obtenue en composant la projection $ \widetilde{R}_\sigma \to R_\sigma$ avec l'action de $R_\sigma$ sur $W'_\sigma$} de
 we then see that the maps  $\xi_\sigma$ and  $\eta_\sigma$ extend to all of  $ \widetilde{W}_{\Theta}$ and $ \widetilde{W}_{\Theta} \times \widetilde{W}_{\Theta}$, and that \eqref{cobord} remains true with $w_1$, $w_2$ in $\widetilde{W}_{\Theta}$.

\subsection{Straightening the intertwining operators}

Still working with our fixed point $\sigma$, we finally use the splitting map $\xi_{\sigma}: \widetilde{W}_\Theta \to \mathbb{C}^\star$ to define the straightened operators
\[ \mathfrak{a}(w, \sigma \otimes \chi) =\xi_{\sigma}(w)^{-1} \mathcal{A}(w,  \sigma \otimes \chi), \quad   w \in \widetilde{W}_\Theta, \quad \chi \in \mathcal{X}_u(M).\]

{These are unitary operators on $\mathcal{H}_P^K$ (combine the results of \cite[Theorem 2.1]{ArthurJFA}, especially properties $(R_4)$ and $(R_2$), with the fact that the operators  $\sigma^N(\dot{w})$ are unitary and the fact that the map $\eta_{\sigma}$  must have modulus one).} 

The results of \S \ref{norma}, combined with Proposition \ref{cocycle_twiste}, then lead to the following observation.
\begin{prop}[1-cocycle relation after straightening]~\label{cocycle_rel}
\begin{itemize}
\item[$\bullet$]  The cocycle relation
\[ \mathfrak{a}(w_1w_2, \sigma \otimes \chi) = \mathfrak{a}(w_1, w_2 (\sigma \otimes \chi)) \mathfrak{a}(w_2, \sigma \otimes \chi)\]
holds for every $w_1, w_2$ in $\widetilde{W}_\Theta$ and for every $\chi \in \mathcal{X}_u(M)$. 
\item[$\bullet$]  Given $w \in \widetilde{W_\Theta}$, {the map  $\chi \mapsto \mathfrak{a}(w, \sigma \otimes \chi)$ is continuous as a map from $\mathcal{X}_u(M)$ to the group of unitary operators of $\mathcal{H}_P^K$}.\qedhere
\end{itemize}
\end{prop}
%The second assertion combines Lemma \ref{cont_cocycle} with the classical fact that $\chi \mapsto \mathcal{A}(w, \sigma \otimes \chi)$ depends holomorphically on $\chi$ in $\mathcal{X}_u(M)$ (see \cite[\S 2]{ArthurJFA} and also \cite[\S 1.3]{WaldspurgerGGP2}).

\section{Variation of $R$-groups along the orbit and quasi-equivalence property}\label{point_universel}

\subsection{Extension of the groups $W_{\sigma \otimes \chi}$}\label{treillis_r_groupes} 

Fix a unitary unramified character $\chi$ of $M$, and consider the stabilizer $W_{\sigma \otimes \chi} \subset W_\Theta$. We recall that we are working under the assumption that $\sigma$ is a fixed point satistying the additional hypothesis \ref{hyp_inclusions}: upon considering the decompositions
\begin{align*} 
 {W}_{\Theta} = W_\sigma & = W'_{\sigma} \rtimes {R}_{\sigma}, \text{ and}\\
 W_{\sigma\otimes \chi} & = W'_{\sigma\otimes \chi} \rtimes {R}_{\sigma\otimes \chi},
\end{align*}
what we have assumed is that

\begin{enumerate}[(a)]
\item  $W'_{\sigma \otimes \chi} \subset W'_{\sigma}$,
\item and there exists an injective morphism $\iota_\chi\colon R_{\sigma \otimes \chi} \to R_{\sigma}$.
\end{enumerate}

In \S \ref{extension_centrale}, we introduced a central extension 
\[ 1 \to Z_\sigma \to \widetilde{R}_{\sigma} \to R_{\sigma} \to 1\]
of the $R$-group attached to the fixed point $\sigma$. We shall be more specific than in \S \ref{extension_centrale} concerning the central extensions $\widetilde{R}_{\sigma \otimes \chi}$ and $\widetilde{W}_{\sigma \otimes \chi}$ which we will use at unramified twists of $\sigma$ by nontrivial characters $\chi \in \mathcal{X}_u(M)$.  

Fix such a $\chi$, and using the embedding $\iota_\chi\colon R_{\sigma \otimes \chi} \to R_{\sigma}$, define
\[ \widetilde{R}_{\sigma \otimes \chi} = \left\{ \widetilde{r} \in \widetilde{R}_{\sigma} \ : \ \mathrm{proj}_{\widetilde{R}_{\sigma} \to R_{\sigma}}(\widetilde{r}) \text{ lies in the image } \iota_\chi(R_{\sigma \otimes \chi})\right\}.\]
This is a subgroup of $\widetilde{R}_{\sigma}$, it contains  $Z_\sigma$, and there is a short exact sequence
\[ 1 \to Z_\sigma \to \widetilde{R}_{\sigma \otimes \chi} \to R_{\sigma \otimes \chi} \to 1\]
where the surjective arrow takes an element in $ \widetilde{R}_{\sigma \otimes \chi} \subset \widetilde{R}_{\sigma}$, projects it in $R_{\sigma}$, and extracts the unique antecedent of the projection by the morphism $\iota_{\chi}$. {By construction, the cocycle $\eta_{\sigma \otimes \chi}$ of $R_{\sigma \otimes \chi}$ becomes a coboundary on $\widetilde{R}_{\sigma \otimes \chi}$, and we can split it using the map $\xi_{\sigma \otimes \chi}\colon  \widetilde{R}_{\sigma \otimes \chi} \to \mathbb{C}^\times$ that sends $\tilde{r}   \in \widetilde{R}_{\sigma}$ to the value of $\xi_\sigma$ at $\mathrm{proj}_{\widetilde{R}_{\sigma} \to R_{\sigma}}(\widetilde{r})$. Notice that the restriction of $\xi_{\sigma \otimes \chi}^{-1}$ to $Z_\sigma$ is given by the character $\zeta_\sigma=(\xi_{\sigma})_{|Z_{\sigma}}^{-1}$, and does not depend on $\chi$. }

Applied at $\sigma \otimes \chi$, Theorem  \ref{quasi_arthur} then gives us the following information. 
\begin{lemm} \label{artq} Given $\chi \in \mathcal{X}_u(M)$, the representation $\widetilde{r} \mapsto \mathfrak{a}(\widetilde{r}, \sigma \otimes \chi)$ of ${\widetilde{R}_{\sigma \otimes \chi}}$ on $\mathcal{H}^K_P$ is quasi-equivalent with the representation $\mathrm{Ind}_{Z_\sigma}^{\widetilde{R}_{\sigma\otimes \chi}}(\zeta_\sigma)$. \end{lemm}

We also define the semidirect product
\[ \widetilde{W}_{\sigma \otimes \chi} = W'_{\sigma \otimes \chi} \rtimes \widetilde{R}_{\sigma \otimes \chi}\]
associated with the action of $\widetilde{R}_{\sigma \otimes \chi}$ on $W'_{\sigma \otimes \chi}$ obtained by composing the projection  $\widetilde{R}_{\sigma \otimes \chi} \to {R}_{\sigma \otimes \chi}$ with the action of ${R}_{\sigma \otimes \chi}$ on $W'_{\sigma \otimes \chi}$.

%At this point, we need to check that using $Z_\sigma$ to split the cocycle $\eta_{\sigma \otimes \chi}$ does not create incoherence when $\chi$ is nontrivial. Recall from \S \ref{continuous_splitting} that there exists a character $\zeta_{\sigma \otimes \chi}$ of $Z_\sigma$ satisfying: 
%\[ \forall z \in Z_\sigma, \forall w \in \widetilde{W}_\Theta, \quad  \xi_{\sigma \otimes \chi}(zw) = \zeta_{\sigma \otimes \chi}(z) \xi_{\sigma \otimes \chi}(w).\]
%In fact, we have $\zeta_{\sigma \otimes \chi} = \zeta_\sigma$ for all $\chi \in \mathcal{X}_u(M)$.  Indeed, given $z \in Z_\sigma$, the map $\chi \mapsto \zeta_{\sigma \otimes \chi}(z) = \xi_{\sigma\otimes \chi}(z)$ is continuous with respect to $\chi$. Because $\zeta_\sigma$ is a character, it takes its values in the set of $k$-th roots of unity, where $k=\mathrm{Card}(Z_\sigma)$. As a result, the map $\chi \mapsto \zeta_{\sigma \otimes \chi}(z)$ is constant on $\mathcal{X}_u(M)$. 
%
%Let us therefore write $\zeta = \zeta_\sigma$ for that character of $Z_\sigma$. 

\subsection{Quasi-equivalence property} We now observe that under Assumption \ref{hyp_inclusions}, we can extend Lemma \ref{artq} to the following result (compare \cite[\S 3.9]{LeungPlymen}). 

\begin{prop} \label{prop_res} The representation $\mathfrak{a}(\cdot, \sigma \otimes \chi)$ of $\widetilde{W}_{\sigma \otimes \chi}$ is quasi-equivalent with the restriction  $\mathfrak{a}(\cdot, \sigma)_{|\widetilde{W}_{\sigma \otimes \chi}}$.\end{prop}

\begin{proof} Write $\lambda\colon \widetilde{W}_{\sigma \otimes \chi} \mapsto \mathrm{End}(\mathcal{H}_{P}^K)$ for the representation $w=w'\tilde{r} \mapsto \mathfrak{a}(\tilde{r}, \sigma \otimes \chi)$ of $\widetilde{W}_{\sigma \otimes \chi}$ (where, given $w \in \widetilde{W}_{\sigma \otimes \chi}$, we write $w=w'\widetilde{r}$ for its decomposition as an element of $W'_{\sigma \otimes \chi} \rtimes \widetilde{R}_{\sigma \otimes \chi}$). Let us also denote by $\mu:  \widetilde{W}_{\sigma \otimes \chi} \mapsto \mathrm{End}(\mathcal{H}_{P}^K)$ the representation $w=w'\tilde{r} \mapsto \mathfrak{a}(\tilde{r}, \sigma)$. What we need to show is that  $\lambda$ and $\mu$ are quasi-equivalent as representations of $\widetilde{W}_{\sigma \otimes \chi}$. 

Lemma \ref{artq} provides us with  a precise knowledge of the irreducible representations of  $\widetilde{W}_{\sigma \otimes \chi}$ that are contained in $\lambda$: they are those in which
\begin{enumerate}[(i)]
\item $W'_{\sigma\otimes \chi}$ acts trivially,
\item  and $\widetilde{R}_{\sigma \otimes \chi}$ acts through an irreducible representation $\beta$ whose central character on $Z_\sigma$ is $\zeta_\sigma$.
\end{enumerate}
As for the irreducible representations of $\widetilde{W}_{\sigma \otimes \chi}$ that occur in $\mu$, we can again use Lemma \ref{artq} to see that in every one of them,
\begin{enumerate}[(a)]
\item[(i)] $W'_{\sigma\otimes \chi}$ acts trivially (because of Assumption \ref{hyp_inclusions}.(a)),
\item[(ii')] and $\widetilde{R}_{\sigma \otimes \chi}$ acts through an irreducible representation $\gamma$ that occurs in the restriction $\mathrm{Res}_{\widetilde{R}_{\sigma \otimes \chi}}\left( \mathrm{Ind}_{Z_\sigma}^{\widetilde{R}_\sigma}\left( \zeta_\sigma\right)\right)$.
\end{enumerate}
Any representation that satisfies (i) and (ii') clearly satisfies  (i) and (ii), since $Z_\sigma$ is contained in $\widetilde{R}_{\sigma \otimes \chi}$. On the other hand, given a representation that satisfies (i) and (ii), the attached representation $\beta$ of $\widetilde{R}_{\sigma \otimes \chi}$ satisfies
\[ \mathrm{Hom}_{\widetilde{R}_{\sigma \otimes \chi}}\left[\beta, \, \mathrm{Res}_{\widetilde{R}_{\sigma \otimes \chi}}\left( \mathrm{Ind}_{Z_\sigma}^{\widetilde{R}_\sigma}\left( \zeta_\sigma\right)\right)\right] = \mathrm{Hom}_{\widetilde{R}_{\sigma}}\left[ \mathrm{Ind}_{\widetilde{R}_{\sigma \otimes \chi}}^{\widetilde{R}_\sigma}(\beta), \, \mathrm{Ind}_{Z_\sigma}^{\widetilde{R}_\sigma}(\zeta_\sigma)\right].\]
Now,  $\mathrm{Ind}_{Z_\sigma}^{\widetilde{R}_\sigma}(\zeta_\sigma)$ contains each of those irreducible representations of  $\widetilde{R}_\sigma$ whose central chracter on $Z_\sigma$ is $\beta$; whenever  $\beta$ satisfies (ii), the subgroup $Z_\sigma$ acts by  $\zeta_\sigma$ in $\mathrm{Ind}_{\widetilde{R}_{\sigma \otimes \chi}}^{\widetilde{R}_\sigma}(\beta)$ (this is seen by inspecting the definition of induced representations of finite groups, and inserting the fact that $Z_\sigma$ is central). We conclude that the representations that satisfy (i)-(ii) are the same as those that satisfy (i)-(ii'), and obtain the Proposition. \qedhere

\end{proof}

\section{Wassermann-Plymen theorem}\label{fin_preuve}

\subsection{The $C^\ast$-blocks and the Plancherel formula}\label{c_theta} We can now turn to $C^\ast$-algebras, and prove Theorem \ref{principal}. We shall first record the definition of the $C^\ast$-algebra $C^\ast_r(G,\Theta)$ on the left-hand side of the Morita-equivalence in Theorem \ref{principal}, and its relationship with the Plancherel formula. 

Let $M$ be a Levi subgroup of $G$, let $\sigma$ be a discrete series representation of $M$, and let  $\mathcal{O} \subset \mathcal{E}_2(M)$ be the orbit of $\sigma$ under $\mathcal{X}_u(M)$. Let $\Theta=(M,\mathcal{O})_G$ denote the $G$-conjugacy class of the pair $(M, \mathcal{O})$. 

Fix a parabolic subgroup $P$ of $G$ with Levi factor $M$. For every $\tau$ in $\mathcal{O}$, {we consider the representation $\mathcal{I}^K_{P}(\tau)$ on the Hilbert space $\mathcal{H}_P^K$}. Fixing a Haar measure on $G$, we can, for every smooth and compactly supported function $f \in \mathcal{C}^\infty_c(G)$, form the operator $\pi_{\tau}(f)=\int_{G} f(g)  \mathcal{I}_P^K(\tau)(g)dg$ (in the notation of \S \ref{induites}). For $f \in \mathcal{C}^\infty_c(G)$, the operator $\pi_{\tau}(f)$ on $\mathcal{H}_P^K$ is compact, and depends continuously on $\tau$. We obtain a linear map 
\[ C^\infty_{c}(G) \to \mathcal{C}_0(\mathcal{O},\mathfrak{K}(\mathcal{H}_{P}^K))\]
where $\mathfrak{K}(\mathcal{H}_P^K)$ stands for the algebra of compact operators on $\mathcal{H}_P^K$. Upon completing $C^\infty_{c}(G)$ to the reduced $C^\ast$-algebra $C^\ast_r(G)$, there arises (using the continuity arguments of \cite[\S 2.6]{Plymen1}) a $C^\ast$-morphism 
\[ C^\ast_r(G) \to \mathcal{C}_0(\mathcal{O},\mathfrak{K}(\mathcal{H}_{P}^K)).\]
We write $C^\ast_r(G;\Theta)$ for the image of $C^\ast_r(G)$ in the algebra on the right-hand side. 

Now form the {$C^\ast$-algebraic direct sum $\bigoplus C^\ast_r(G;\Theta)$, indexed by $G$-conjugacy classes $\Theta$ of discrete pairs $(M,\mathcal{O})$ (see \cite[\S 2]{CCH})}. From the above discussion, we obtain a morphism 
\begin{equation} \label{planch} C^\ast_r(G) \to \bigoplus \limits_{\Theta} C^\ast_r(G;\Theta). \end{equation}
One of the main results of \cite{Plymen1}, that can be viewed as a reformulation of Harish-Chandra's Plancherel formula, is that this map is a $C^\ast$-isomorphism. For enlightening treatments of this topic, see \cite{Plymen1} {on the reduced $C^\ast$-algebras of $p$-adic groups}, \cite[\S 5-6]{CCH} for a related $C^\ast$-algebraic discussion in the case of real groups, and of course \cite[\S VI-VIII]{Waldspurger} for the Plancherel formula in the $p$-adic case. The Plancherel formula  in \cite[Theorem VIII.1.1]{Waldspurger} uses the Schwartz algebra rather than the reduced $C^\ast$-algebra of \eqref{planch}, and spaces of smooth functions as in \cite[\S VI.3]{Waldspurger} serve as counterparts for the right-hand side of \eqref{planch}. {Concerning the $C^\ast$-algebraic aspect, we mention the role of Bernstein's uniform admissibility theorem in ensuring that the map in \eqref{planch} takes values in the $C^\ast$-algebraic direct sum; see the enlightening discussion in \cite[\S 5]{CCH}. }

\subsection{The Leung--Plymen method}\label{leung_plymen}

We are ready to apply the results of  \cite[sections 2 and 3]{LeungPlymen}  and go on to a proof of Theorem \ref{principal}. 

Consider the  {manifold $S = \mathcal{O}$}, the Hilbert space $\mathcal{H}=\mathcal{H}_P^K$, and the $C^\ast$-algebra
\[ \mathcal{A} = \mathcal{C}_0(S) \otimes \mathfrak{K}(\mathcal{H}),\]
where $\mathfrak{K}(\mathcal{H})$ still stands for the algebra of compact operators on $\mathcal{H}=\mathcal{H}_P^K$, and $\mathcal{C}_0(S)$ is the $C^\ast$-algebra of continuous functions on $S$ which vanish at infinity (of course, $S$ is noncompact only when $F=\mathbb{R}$). Bring in the finite group
\[ \Gamma = \widetilde{W}_\sigma\]
of \eqref{wtilde}. The straightened intertwining operators of Proposition \ref{cocycle_rel}  give rise to a map
\begin{align*} u\colon  \Gamma & \to \mathcal{C}_0(S, \mathcal{U}(\mathcal{H}))\\ w & \mapsto \left[ \chi \mapsto \mathfrak{a}(w, \sigma \otimes \chi)\right],\end{align*} 
where $ \mathcal{U}(\mathcal{H}_P^K)$ is the group of unitary operators of $\mathcal{H}$. The map $u$ is a $1$-cocycle of $\Gamma$ with values in $\mathcal{C}_0(S, \mathcal{U}(\mathcal{H}))$. To such a cocycle, Leung and Plymen associate an action
\[ \beta\colon \Gamma \to \mathrm{Aut}(\mathcal{A}),\]
where $\beta(w)$ sends a split element $\varphi \otimes M$ of $\mathcal{A}= \mathcal{C}_0(S) \otimes \mathfrak{K}(\mathcal{H})$ to the element {$\mathrm{Ad}(u(w))(\varphi(w^{-1} \cdot) \otimes M)$}. See \cite[\S 2.11]{LeungPlymen}.

The results proven above, especially Proposition \ref{prop_res}, indicate that we are in a position to apply the results of \cite{LeungPlymen}, Section 2. Consider the modification of $\beta$ built with a modification of the cocycle $u$, where we ``freeze'' $\chi$: this is the map 
\[ \gamma\colon \Gamma \to \mathrm{Aut}(\mathcal{A})\]
{associated with the cocycle  $v\colon \Gamma \to \mathcal{C}_0(S, \mathcal{U}(\mathcal{H}))$ given by $v(w) = \left(\chi \mapsto  u(w)(1_S)\right)$ for all $w \in \Gamma$. } We can apply Lemma 2.11 of \cite{LeungPlymen}, and conclude that  that the fixed-point algebras $\mathcal{A}^\beta$ and $\mathcal{A}^\gamma$ are Morita-equivalent:
\[ \mathcal{A}^\beta \underset{\text{Morita}}{\sim} \mathcal{A}^\gamma.  \]
{To see why \cite[Lemma 2.11]{LeungPlymen} can be applied here, we refer to  \cite[\S 3.5, 3.6 and 3.7]{LeungPlymen}: all the arguments there apply to our situation.  We mention that in \cite[\S 2]{LeungPlymen}, it is assumed that $S$ is a compact torus. But the proofs of all results of \cite[\S 2]{LeungPlymen} needed for the above application are valid also when $S$ is a real finite-dimensional vector space, using the fact that $\mathcal{C}_0(S)$ is then a $\sigma$-unital algebra.}

The $C^\ast$-algebra $\mathcal{A}^\beta$ of the left-hand side is isomorphic with the $C^\ast$-algebra of continuous maps $f\colon S \to \mathfrak{K}(\mathcal{H})$ such that $f(w \chi) = \mathfrak{a}(w, \sigma \otimes \chi)^{-1} f(\chi)\mathfrak{a}(w, \sigma \otimes \chi)$ for all $\chi \in T$ and $w \in \Gamma$. Plymen proved in \cite[Theorem 2.5]{Plymen1} that the latter algebra is isomorphic with $C^\ast_r(G; \Theta)$. We thus have the Morita equivalence
\begin{equation} \label{soll1} C^\ast_r(G;\Theta) \quad \underset{\text{Isom.}}{\simeq} \quad \mathcal{A}^\beta \quad \underset{\text{Morita}}{\sim} \quad \mathcal{A}^\gamma \underset{\text{Isom.}}{\simeq} \quad \left( \mathcal{C}_0(S) \otimes \mathfrak{K}(\mathcal{H})\right)^{\widetilde{W}_\sigma}  \end{equation}
{where, in the algebra on the right, the action of $\widetilde{W}_{\sigma}$ on $\mathcal{C}_0(S) \otimes \mathfrak{K}(\mathcal{H})$ is that where an element $w \in \widetilde{W}_{\sigma}$ sends a map $f\colon S \to \mathfrak{K}(\mathcal{H})$ to the map $\chi \mapsto \mathfrak{a}(w, \sigma) f(w\chi)\mathfrak{a}(w, \sigma)^{-1}$.}

Now, the elements of $W_\sigma'$ act trivially on $\mathfrak{K}(\mathcal{H}^K_P)$. It follows that

\begin{equation} \label{soll2} C^\ast_r(G;\Theta) \quad \underset{\text{Isom.}}{\simeq} \quad \mathcal{A}^\beta \underset{\text{Morita}}{\sim} \quad \left( \mathcal{C}_0(S/W_{\sigma}') \otimes \mathfrak{K}(\mathcal{H}^K_P)\right)^{\widetilde{R}_\sigma}  \end{equation}
as in \cite{LeungPlymen}, \S 2.13.

\subsection{Real groups: proof of Theorem \ref{th_wass}} \label{preuve_wass}

In this paragraph, we assume $F = \mathbb{R}$, and explain how Wassermann's theorem for real groups can be deduced from the isomorphism \eqref{soll2}. We emphasize the fact that the result is classical, and that our method is not new: this paragraph merely says how the above results  make it easy to fill in the details of Wassermann's proof in \cite{Wassermann}. Our remarks here can be viewed as an addendum to Leung and Plymen's paper \cite{LeungPlymen}, which itself unravels Wassermann's argument. 

For real groups, the 2-cocycle $\eta_\sigma$ of Proposition \ref{cocycle_twiste} always splits (see \cite[\S 8]{KnappStein} and \cite[\S 2]{ArthurActa}). Therefore we may take the central extension $\widetilde{R}_\sigma$ of \S \ref{extension_centrale} to be just $R_\sigma$. Besides, we know from Knapp and Stein \cite{KnappStein} that the $R$-group $R_\sigma$ is an abelian elementary $2$-group. Just as Leung and Plymen did in \cite[\S 2.13]{LeungPlymen}, we can therefore use Takai duality and Lemma \ref{artq} to see that the right-hand side of \eqref{soll2} is equivalent with $\mathcal{C}_0({S/W'_\sigma}) \rtimes R_\sigma = \mathcal{C}_0(\mathcal{O}/W'_\sigma) \rtimes {R_\sigma}$. This completes the proof of Theorem \ref{principal} when $F=\mathbb{R}$.

We now indicate how Wassermann's result, Theorem \ref{th_wass}, can be obtained from Theorem \ref{principal}. What remains to be seen is that for real groups, the basepoint $\sigma$ may always be chosen so as to satisfy Assumptions \ref{hyp_pt_fixe} and~\ref{hyp_inclusions}. 

Let us call in the ``Langlands decomposition'' of the Levi subgroup $M$: let $A$ be the split component of $M$, and let $M^0$ be the subgroup of $M$ generated by all compact subgroups of $M$. Then we have the direct product decomposition $M = M^0 A$ . Therefore, if $\sigma$ is a discrete series representation of $M$, the restriction of $\sigma$ to $A$ determines a unitary unramified character of $M$. Twisting $\sigma$ by the inverse of that character, we obtain a representation $\sigma_0$ in the orbit $\mathcal{O}$: the representation $\sigma_0$ is in fact the trivial extension to $M$ of a discrete series representation of $M^0$. 

Assumption \ref{hyp_pt_fixe} is then immediate for $\sigma_0$: since that representation is trivial on $A$, all representations $^w \sigma_0$, for $w \in W(G,M)$, are trivial on $A$ as well, and the action of $W_\Theta$ on $\mathcal{O}$ admits $\sigma_0$ as a fixed point. 

As for Assumption \ref{hyp_inclusions}, the corresponding results were established by Knapp and Stein in \cite[\S 13-14]{KnappStein}. For $\chi \in \mathcal{X}_u(M)$, they describe in the proof of \cite[Lemma 14.1]{KnappStein} an injection $R_{\sigma_0 \otimes \chi} \hookrightarrow R_{\sigma_0}$; a crucial part of argument is the fact that for $w \in W_\Theta$, when one considers the restriction $\mathfrak{a}(w, \sigma \otimes \chi)$ to a finite sum of $K$-types, the characteristic polynomial of the restriction must remain constant as $\chi$ varies, provided $\chi$ remains in the subspace of characters satisfying $w\chi = \chi$. The inclusion $W'_{\sigma_0 \otimes \chi} \subset W'_{\sigma_0}$ follows from the same argument with characteristic polynomials.  Therefore parts (a) and (b) of Assumption \ref{hyp_inclusions} are satisfied for the base-point $\sigma=\sigma_0$ in $\mathcal{O}$,  and we obtain Wassermann's Theorem \ref{th_wass} for real groups.

\subsection{$p$-adic groups: proof of Theorem \ref{principal}} \label{preuve_padique}
We now assume $F$ to be nonarchimedean, and complete the proof of our theorem for $p$-adic groups. Since the central extension in the isomorphism \eqref{soll2} can now be nontrivial, we must bring in a slight modification to the Leung--Plymen work to obtain Theorem \ref{principal}. We will use ideas that have been spelt out by Opdam and Solleveld~\cite{OpdamSolleveld}.

The $C^\ast$-algebra on the right-hand side of  \eqref{soll2} is of a kind studied in \cite[\S 3]{OpdamSolleveld}, especially Theorem 3.2 there. Consider 
\[ \mathcal{T} = S/W'_\sigma\]
and
\[ \mathcal{G} = \widetilde{R}_\sigma.\]
We have been studying two representations of $\mathcal{G}$: one on
\[ V = \mathcal{H}^K_P\]
through the normalized intertwining operators, and another one on the twisted group algebra
\[ V' = \mathbb{C}\left[{R}_\sigma, \eta_\sigma \right]\]
through the representation $\mathrm{Ind}_{Z_\sigma}^{\widetilde{R}_\sigma}(\zeta_\sigma)$. {To see where that representation comes from, recall that if $\mathbb{C}\left[ \tilde{R}_\sigma\right]$ is the group algebra for the extended group $\tilde{R}_\sigma$, and if $\tilde{p}$ is the idempotent in $\mathbb{C}\left[\tilde{R}_\sigma\right]$ given by projection on its isotypical component corresponding to $\mathrm{Ind}_{Z_\sigma}^{\widetilde{R}_\sigma}(\zeta_\sigma)$, then  $ \tilde{p}~\mathbb{C}\left[\tilde{R}_\sigma\right] \simeq \mathbb{C}\left[{R}_\sigma, \eta_\sigma \right]$. See for instance \cite[\S 2.1]{ABPSOrsay}}. 

What we need to prove is that the $C^\ast$-algebra
\begin{equation} \label{mor1}  \left( \mathcal{C}(\mathcal{T}) \otimes \mathfrak{K}(V)\right)^{\mathcal{G}}\end{equation}
on the right-hand side of  \eqref{soll2} is Morita-equivalent with the $C^\ast$-algebra 
\begin{equation} \label{mor2} \widetilde{p} \left[\ \mathcal{C}(\mathcal{O}/W'_{\sigma}) \rtimes \widetilde{R}_{\sigma} \right] \underset{\text{isomorphism}}{\simeq} \left( \mathcal{C}(\mathcal{T}) \otimes \mathrm{End}(V')\right)^{\mathcal{G}}\end{equation}
of Theorem \ref{principal}. (For the isomorphism between the two algebras in \eqref{mor2}, and for twisted crossed products and central idempotents attached to group extensions, see  \cite[pp. 700-701]{OpdamSolleveld} and \cite[\S 2.1]{ABPSOrsay}.)

{Proposition \ref{prop_res} shows that the representations of $\mathcal{G}$ on $V$ and $V'$ are quasi-equivalent. Theorem \ref{principal} would then follow from Theorem 3.2 in \cite{OpdamSolleveld} if we could apply it. The problem is that  $V$ is infinite-dimensional, whereas Theorem 3.2 in \cite{OpdamSolleveld} is formulated for finite-dimensional $V$ and $V'$ (and uses $\mathcal{C}^\infty(\mathcal{T})$ rather than $\mathcal{C}(\mathcal{T})$). We shall prove, however, that the same conclusion holds in our $C^\ast$-algebraic context:}

\begin{prop}\label{morita}
Let $\mathcal{G}$ be a finite group, let $M$ be a compact $\mathcal{G}$-space, and let $(V, \pi)$ and $(V', \pi')$ be two unitary representations of $\mathcal{G}$. If $(V, \pi)$ and $(V', \pi')$ are quasi-equivalent as representations of $\mathcal{G}$, then $A = (\mathcal{C}(M) \otimes \mathfrak{K}(V))^\mathcal{G}$ and $B = (\mathcal{C}(M) \otimes \mathfrak{K}(V'))^\mathcal{G}$ are Morita-equivalent as $C^\ast$-algebras. 
\end{prop}
\begin{proof}
{Consider  
\[ X = (\mathcal{C}(M) \otimes \mathfrak{K}(V; V'))^\mathcal{G}\] 
viewed as the space of continuous maps $M \to \mathfrak{K}(V; V')$ satisfying $f(\gamma^{-1} m) = \pi'(\gamma)^{-1}  f(m)  \pi(\gamma)$ for  $m \in M$, $\gamma \in \mathcal{G}$. }We will prove that $X$ is a Morita equivalence bimodule  in the sense of \cite{Rieffel}.

First note that $X$ is naturally a $(B,A)$-bimodule: if $f \colon M \to \mathfrak{K}(V; V')$ is an element of $X$, and if $a$ is an element of $A$, viewed as a $\mathcal{G}$-equivariant map $M \to \mathfrak{K}(V)$, then we define  $(fa) \in X$ as the map $m \mapsto f(m)  a(m)$.  

By equipping $X$ with the $A$-valued inner product that sends $f,g \in X$ to the element $\langle f, g\rangle_{A}$ of $A$ defined as 
\[ \langle f,g\rangle_{A}: m \mapsto f(m)^\ast g(m),\]
 we obtain on $X$ the structure of a \emph{Hilbert} right $A$-module (see \cite{Lance}). By similarly using pointwise composition to turn $X$ into a left $B$-module, and equipping $X$ with the $B$-valued inner product  $ _{B}\langle f, g\rangle: m \mapsto  f(m)g(m)^\ast$, we see that $X$ is a Hilbert $(B,A)$-bimodule that satisfies 
  \[ f~\langle g, h\rangle_{A} = ~ _{B}\langle f, g\rangle~h \quad \text{ for all $f$, $g$, $h$ in $X$.}\]
To show that $X$ implements a Morita-equivalence between $A$ and $B$, it is therefore enough to prove that 
\[ \langle X, X\rangle_{A} = \text{linear span of the elements $\langle f,g\rangle_{A}$, $f, g \in X$}\]
is dense in $A$ (and of course that $ _{B}\langle f, g\rangle$ is dense in $B$, but the same proof will work). 

We now prove that density. 

\begin{itemize}
\item[$\bullet$] Let us first fix $m$ in $M$ and collect some information on the value at $m$ of $\langle f,g\rangle_{A}$, for $f, g$ in $X$. Write $\mathcal{G}_m$ for the stabilizer of $m$ in $\mathcal{G}$. Because $f$ and $g$ are $\mathcal{G}$-equivariant as maps $M \mapsto \mathfrak{K}(V,V')$, the operators  $f(m)$ and $g(m)$ in $\mathfrak{K}(V,V')$ are  $\mathcal{G}_m$-equivariant.  
Write 
\begin{equation} \label{deco} V = \bigoplus \limits_{\rho \in \mathrm{Irr}(\mathcal{G}_m)} V_{\rho} \quad  \text{and}\quad V' = \bigoplus \limits_{\rho \in \mathrm{Irr}(\mathcal{G}_m)} V'_{\rho} \end{equation}
for the decomposition of $(V,\pi_{|\mathcal{G}_m})$ and $(V',(\pi')_{|\mathcal{G}_m})$ into isotypical subspaces for the irreducible representations of $\mathcal{G}_m$. Note that our quasi-equivalence hypothesis means that  $V_\rho = \{0\}$ if and only if $V'_{\rho} = \{0\}$. Now, the operators $f(m)$ and $g(m)$ are $\mathcal{G}_m$-equivariant, and we have $\mathrm{Hom}_{\mathcal{G}_m}(V_\rho,V_{\rho'})=\mathrm{Hom}_{\mathcal{G}_m}(V_\rho,V'_{\rho'})=\mathrm{Hom}_{\mathcal{G}_m}(V'_\rho,V'_{\rho'})=\{0\}$ whenever $\rho \neq \rho'$. Therefore the value $\varphi=\langle f,g\rangle_{A}(m)$ is an element of $\mathfrak{K}(V)^{\mathcal{G}_m}$ which can be written as
\begin{equation} \label{forme} \varphi=\sum \limits_{\rho \in \mathrm{Irr}(\mathcal{G}_m)} f_{m,\rho}^\ast g_{m, \rho},  \quad \text{where $f_{m, \rho}$ and $g_{m, \rho}$ are elements of $\mathfrak{K}(V_\rho,V'_\rho)^{\mathcal{G}_m}$. }\end{equation}

\item[$\bullet$]  Let us keep $m \in M$ fixed, and consider an element $\varphi$ of $\mathfrak{K}(V)^{\mathcal{G}_m}$. We now prove that $\varphi$ can be written as a limit of linear combinations of elements satisfying \eqref{forme}. Taking up the decomposition of $V$ in \eqref{deco}, write  $\varphi$ as $\sum \limits_{\rho  \in \mathrm{Irr}(\mathcal{G}_m)} \varphi_\rho$, where each $\varphi_{\rho}$ belongs to $\mathfrak{K}(V_\rho)^{\mathcal{G}_m}$. Since $\mathrm{Irr}(\mathcal{G}_m)$ is a finite set, it is enough to prove our claim separately for each $\varphi_\rho$. Furthermore, since the operator  $\varphi$ is compact, it is enough to prove our claim when $\varphi$ has finite rank. 

Let us therefore assume that  $\varphi$ is a nonzero element in $\mathfrak{K}(V_\rho)^{\mathcal{G}_m}$, and that it has finite rank. Decompose its range into irreducible representations of $\mathcal{G}_m$, writing   
\[ \varphi(V_\rho) = \sum \limits_{j \in J} V_{j}\]
where $J$ is a finite set and each  $V_j$ is a $\mathcal{G}_m$-stable subspace of $V_\rho$ with dimension $\dim(\rho)$. 

For each $j \in J$, the inverse image $A_j=\varphi^{-1}(V_j)$ is a closed subspace of $V_\rho$, and it is $\mathcal{G}_m$-stable. Furthermore, the kernel of  $\varphi_{|A_j}$ has codimension $\dim(V_j) = \dim(\rho)$. Let $E_j$ be the orthogonal of $\mathrm{Ker}( \varphi_{|A_j})$ in $A_j$; then $E_j$ carries an irreducible representation of $\mathcal{G}_m$ of class $\rho$. 

We are now in a position to use our quasi-equivalence assumption and bring  $V'$ back into the picture. Fix a $\mathcal{G}_m$-stable subspace $V'_{\rho, \star}$ of $V'_\rho$ with dimension $\dim(\rho)$: the existence of such a subspace comes from the quasi-equivalence assumption. Now, we have three irreducible representations  of $\mathcal{G}_m$, namely $E_j$, $V'_{\rho,\star}$ and $V'_j$; and all three are equivalent and of class $\rho$. Therefore, there are intertwining maps $g_{j}\colon E_j \to V'_{\rho, \star}$ and $f_j\colon V_{j} \to V'_{\rho, \star}$. By considering $f_{j}^\ast g_j$ and $\varphi_{|E_j}$, we then obtain two elements of  $\mathrm{Hom}_{\mathcal{G}_m}(E_j, V_j)$. 

But the space $\mathrm{Hom}_{\mathcal{G}_m}(E_j, V_j)$ is one-dimensional, by Schur's lemma. We deduce that  $f_{j}^\ast g_j$ and $\varphi_{|E_j}$ are proportional; since  $\varphi$ is zero on the orthogonal of $\bigoplus \limits_{j \in J} E_j$, we obtain the equality
\[ \varphi = \sum \limits_{j \in J} \lambda_j f_j^\ast g_j\]
where the $\lambda_j$, $j \in J$, are scalars, and where $g_j$ is extended to $V_\rho$ by setting it to zero on the orthogonal of $E_j$. This proves our claim that $\varphi$ can be written as a combination of elements satisfying \eqref{forme}. 

{As one referee pointed out, it is possible to organize the proof for our claim in such a way that algebra and analysis become well separated. Write $V_\rho$ as $E \otimes H_\rho$, where $E$ is a carrier space for $\rho$ and $H_\rho$ is a Hilbert space on which $\mathcal{G}_m$ acts trivially. Write $V_\rho'=E\otimes H'_\rho$ in the same way.  Then a rather straightforward application of Schur's lemma shows that the spaces $\mathfrak{K}(V_\rho,V'_\rho)^{\mathcal{G}_m}$ and $\mathfrak{K}(H_\rho,H'_\rho)$ are isomorphic. To check our claim that $\varphi$ can be written as a limit of linear combinations of elements satisfying (5.7), it is therefore enough to check the same property for elements of $\mathfrak{K}(H_\rho)$ and $\mathfrak{K}(H_\rho,H'_\rho)$, which is easy and elementary.     }

\item[$\bullet$] Up to this point, we have kept the point $m$ fixed in $M$, and proved that the value at $m$ of every element of $A$ can be appproximated by elements of the form \eqref{forme}. We must now deal with the fact that the elements of $A$ and $\langle X, X\rangle_A$ are maps on $M$.  

Notice that if $f_\star$ is an element of $\mathfrak{K}(V,V')$, there exists a continuous map $\tilde{f}\colon M \to \mathfrak{K}(V,V')$ whose value at $m$ is $f_\star$, and which vanishes outside a small neighborhood of $m$. Averaging over the $\mathcal{G}$-transforms of $\tilde{f}$, we obtain an element $f$ of $X = \mathcal{C}(M, \mathfrak{K}(V; V'))^\mathcal{G}$. If  the element $f_\star$ we started from is $\mathcal{G}_m$-equivariant, then the value of ${f}$ at $m$ is still equal to $f_\star$.
 
We can apply the above remark to the operators $f_j$ and $g_j$ above, viewed as elements in $\mathfrak{K}(V, V')$, and extend them to $\mathcal{G}$-equivariant continuous maps $M \to \mathfrak{K}(V, V')$. From the work we did at fixed $m$, we deduce the following statement: 
 \begin{equation} \label{ponctuel} \text{For every $\varepsilon >0$,  each $s \in S$ and $a \in A$, there exists $\alpha \in \langle X, X\rangle_{A}$ satisfying $\left\Vert{a(m)-\alpha(m)}\right\Vert<\varepsilon$.}\end{equation}
To conclude that  $\langle X, X\rangle_{A}$ is dense in $A$, we need to pass from the pointwise approximation property in \eqref{ponctuel} to a uniform approximation property for maps  $M \to \mathfrak{K}(V, V')$. This can be done by applying a suitable version of the Stone-Weierstrass theorem; see for instance \cite[Theorem 1]{Prolla}. \qedhere% est parfaitement adapté à notre situation et permet de conclure (dans les notations de \cite{Prolla}, il faut appliquer le théorème 1 à à $W = \langle X,X\rangle_{A}$ et $E = \mathfrak{K}(V)$).
\end{itemize}
\end{proof}
Applying Proposition \ref{morita} to $M = \mathcal{T}$ completes the proof of Theorem \ref{principal}. 

\newpage

\part{Fixed points and good fixed points for quasi-split classical groups}\label{part2}

Throughout Part \ref{part2}, we assume $F$ to be a $p$-adic field. 

\section{Existence of fixed points for symplectic, orthogonal and unitary groups} \label{exist_pt_fixe}

\subsection{A preliminary remark on $\mathrm{GL}(n,F)$}

\begin{lemm} \label{contrag} Let $\sigma$ be a discrete series representation of $ \mathrm{GL}(n,F)$, $ n \geq 1$. 

\begin{itemize}
\item[$\bullet$] Write $\widecheck{\sigma}$ for the representation $g \mapsto \sigma(^t g^{-1})$ of $\mathrm{GL}(n,F)$.  If there exists a unitary unramified character $\chi$ of $\mathrm{GL}(n,F)$ such that $\widecheck{\sigma} \simeq \sigma \otimes \chi$, then there exists another character $\nu \in \mathcal{X}_u(\mathrm{GL}(n,F))$ such that $\tau = \sigma \otimes \nu$ satisfies $\widecheck{\tau}\simeq\tau$. 

\item[$\bullet$] The above remains true if $\widecheck{\sigma}$ is replaced by the representation $g \mapsto \sigma(^\tau g^{-1})$ of $\mathrm{GL}(n,F)$, in which $^\tau(\dots)$ denotes transposition with respect to the off-diagonal. 
\end{itemize}
\end{lemm}

\begin{proof}
\begin{itemize}
\item[$\bullet$] Let us first note that if $\nu$ is a unitary unramified character of $G=\mathrm{GL}(n,F)$, then $\widecheck{\nu} = \nu^{-1}$. Indeed, for every $g \in G$, we have $\widecheck{\nu}(g) = \nu(^t\!g)^{-1}$; {but since every character of $\mathrm{GL}(n,F)$ factors through the determinant, we have $\nu(^t\!g) = \nu(g)$ for all $g$.} %The character $\nu$ has trivial restriction to $G^0 = \left\{ g \in G : \det(g) \in \mathcal{O}^\star\right\}$, where $\mathcal{O}$ is the integer ring of $F$, and if we denote by $D$ the subgroup of diagonal matrices in $\mathrm{GL}(n,F)$, then $G=G^0 D$. Starting from $g \in G$ and writing $g = g^0 d$ with $g^0 \in G^0$ and $d \in D$, we then have $\nu(g) = \nu(d)$, but also $\nu(^t g)=\nu(^t d \cdot ^t g^0)= \nu(^t\!d) \nu (^t g\!^0) = \nu(d)$: the equality $\nu(^t\!g) = \nu(g)$ follows.

\item[$\bullet$] Now, assume $\widecheck{\sigma} \simeq \sigma \otimes \chi$, where $\chi \in \mathcal{X}_u(\mathrm{GL}(n,F))$. Let $\nu = \chi^{1/2}$ be a square-root of $\chi$. We then have
\[ \widecheck{(\sigma \otimes \nu)} = \widecheck{\sigma} \otimes \widecheck{\nu} = \sigma \otimes (\chi \widecheck{\nu}) = \sigma \otimes (\chi \nu^{-1}) = \sigma \otimes (\chi \chi^{-1/2}) = \sigma \otimes \chi^{1/2} = \sigma \otimes \nu\]
as announced. 
\item[$\bullet$] Remark that for every $n\times n$ matrix $A$, we have $^\tau(A) = J_n\, ^t\!(A) \, J_n^{-1}$, where $J_n$ is the matrix whose entries are $1$ on the off-diagonal and $0$ elsewhere. It is then clear that the above argument goes through if we replace $^t\!g$ by $^\tau\!g$ everywhere.\qedhere
\end{itemize}
\end{proof}
\subsection{Symplectic groups} We now fix a positive integer $n$ and consider the symplectic group $G=\mathrm{Sp}(2n,F)$.

 \subsubsection{Levi subgroups and their Weyl groups} \label{levi}
 
Let us recall a description of the Levi subgroups of $G$, and their attached Weyl groups (see \cite{GoldbergSPN}).

\begin{itemize}
\item[$\bullet$] If $M$ if a Levi subgroup of $G$, there exists an isomorphism
\[ M \simeq \mathrm{GL}(n_1,F) \times \dots \times \mathrm{GL}(n_r,F) \times M'\]
where $M'$ is either trivial or a symplectic group $\mathrm{Sp}(2q,F)$, and the integers $n_1, \dots, n_r$ satisfy  $n_1 + \dots + n_r + q = n$ (we use  $q=0$ when $M'$ is trivial). Furthermore, $M$ is conjugate in $G$ with the block-diagonal subgroup 
\begin{equation} \label{desc_levi} M_{\mathrm{std}} = \left\{ \mathrm{Diag}\left(g_1, \dots, g_r, \ \gamma , \ ^\tau\!(g_r)^{-1}, \dots, ^\tau\!(g_1)^{-1}\right), \quad g_1 \in \mathrm{GL}(n_1,F), \dots,  g_r \in \mathrm{GL}(n_r,F), \gamma \in M'\right\},\end{equation}
where $^\tau\!(\dots)$ still denotes transposition with respect to the off-diagonal. 

Henceforth we shall assume that $M = M_{\mathrm{std}}$, and write
\begin{equation} \label{desc_levi_spn} \left[ g_1, \dots, g_r, \gamma \right] \end{equation}
for the general element of $M$ that appears in \eqref{desc_levi}.

\item[$\bullet$] The Weyl group $W(G,M)$ is generated by two kinds of elements:
\begin{enumerate}[(a)]
\item permutations of ``blocks of the same size'' : if $i$ and $j$ are two integers in $\{1, \dots, r\}$ such that $n_i=n_j$, the transposition $(ij)$ acts on $M$ by exchanging $g_i$ and $g_j$ in the general element \eqref{desc_levi_spn}.
\item  sign changes  $\mathbf{c}_i$: for $i \in \{1, \dots, r\}$, the sign change $\mathbf{c}_i$ acts on the general element $ \left[ g_1, \dots, g_r, \gamma \right]$  of $M$ (in \eqref{desc_levi_spn})  by replacing $g_i$ with $^\tau\!g_{i}^{-1}$.
\end{enumerate}
 
The permutations in (a) determine a subgroup of $W(G,M)$ that is isomorphic with a subgroup  $\mathfrak{S}$ of the symmetric group $\mathfrak{S}_r$; the subgroup $\mathfrak{S}$ is generated by  those transpositions $(ij)$ such that $n_i=n_j$. The sign changes in (b) determine a normal subgroup $\mathfrak{C}$  of $W(G,M)$ that is isomorphic with  $(\mathbb{Z}/2\mathbb{Z})^r$. We have $W(G,M) = \mathfrak{S} \mathfrak{C}$, and in fact $W(G,M)$ is the semidirect product $\mathfrak{S} \ltimes \mathfrak{C}$ associated with the action of $\mathfrak{S}$ on $\mathfrak{C}$ by conjugation.

\end{itemize}

\subsubsection{Orbits and orbit types} \label{orbites} Let us now fix a discrete series representation $\sigma$ of $M$, and decompose 
\[ \sigma = \sigma_1 \otimes \dots \otimes \sigma_r \otimes \tau\] 
 where $\sigma_i \in \mathcal{E}_2(\mathrm{GL}(n_i,F))$ for $i \in \{1, \dots, r\}$, and where $\tau \in \mathcal{E}_2(M')$. If $\pi$ is a representation of one of the blocks $\mathrm{GL}(n_i,F)$, recall that we write $\widecheck{\tau}$ for the representation $g \mapsto \pi(^\tau\!g^{-1})$. 
 
We will determine the group $W_{\Theta}$ attached to $\Theta=(M, \mathcal{X}_u(M)\cdot \sigma)_{G}$. Our first step is to introduce an equivalence relation on $\{1, \dots, r\}$ by declaring that $i \sim j$ when
\begin{itemize}
\item[$\bullet$] $n_i = n_j$,
\item[$\bullet$] and $\sigma_j$ is either a twist\footnote{When $\pi$ and $\pi'$ are two representations of a group $L$, we shall use the phrase ``$\pi'$ is a twist of $\pi$''~as an shortened way of saying  that there exists $\chi \in \mathcal{X}_u(L)$ such that $\pi'$ is equivalent with $\pi \otimes \chi$.} of $\sigma_i$ or a twist of $\widecheck{\sigma_i}$.
\end{itemize}
The relation $\sim$ is indeed an equivalence relation (that is because  $\widecheck{\chi}=\chi^{-1}$, for every {$\chi \in \mathcal{X}_u(\mathrm{GL}(n_i,F))$}: see the proof of Lemma \ref{contrag}). If we consider an orbit $\Omega \subset \{1, \dots, r\}$ of that equivalence relation, and if we choose an element $i_\Omega$ in $\Omega$, then two situations can present themselves: 
\begin{itemize}
\item[$\bullet$] \emph{Case 1 : $\widecheck{\sigma_{i_\Omega}}$ is a twist of $\sigma_{i_\Omega}$}. Then it is also true that for all $j \in \Omega$, the representation $\sigma_j$ is a twist of $\sigma_{i_\Omega}$, and therefore $\widecheck{\sigma_j}$ is also a twist of $\sigma_j$. As a consequence, the property that defines case $1$ depends only on $\Omega$, not on the choice of $i_\Omega$.

If we are in that situation, we shall say that $\Omega$ is a \emph{pure orbit of type I}. 

\item[$\bullet$] \emph{Case 2 : $\widecheck{\sigma_{i_\Omega}}$ is not a twist of $\sigma_{i_\Omega}$}. Taking up the above discussion, we see that for $j \in \Omega$, the representation $\widecheck{\sigma}_j$ cannot be a twist of  $\sigma_j$. We then introduce the following subsets of~$\Omega$:
\[ \Omega^{\mathrm{per}} = \left\{ j \in \Omega \ : \ \sigma_j \text{ is a twist of } \sigma_{i_\Omega} \right\},\]
\[ \Omega^{\mathrm{flip}} = \left\{ j \in \Omega \ : \ \sigma_j \text{ is a twist of } \widecheck{\sigma_{i_\Omega}} \right\}.\]
These are clearly disjoint subsets, and we have  $\Omega = \Omega^{\mathrm{per}} \cup \Omega^{\mathrm{flip}}$. 

When $\Omega^{\mathrm{flip}}$ is empty, we shall say that  $\Omega$ is a \emph{pure orbit of type II}. When it is nonempty, we shall say that $\Omega$ is a \emph{mixed orbit}. As before, these notions depend only on $\Omega$, not on the choice of $i_\Omega$. When $\Omega$ is a mixed orbit, it is true that the decomposition $\Omega = \Omega^{\mathrm{per}} \cup \Omega^{\mathrm{flip}}$ can depend on the choice of $i_{\Omega}$: but if $i_\Omega$ changes, the only way in which the decomposition can be affected is that one may have to exchange the roles of  $\Omega^{\mathrm{per}}$ and $\Omega^{\mathrm{flip}}$. 
\end{itemize}

\subsubsection{Description of the group $W_\Theta$} \label{w_theta} At this point, we can exhibit some special elements in $W_\Theta$: 
\begin{enumerate}[(a)]
\item When $\Omega$ is a pure orbit of type I, the group $W_\Theta$ contains all transpositions  $(ij)$, $(i,j) \in \Omega^2$, and contains all sign changes $\mathbf{c}_i$, $i \in \Omega$. 
\item When $\Omega$ is a pure orbit of type II, the group  $W_\Theta$ contains all transpositions $(ij)$, $(i,j) \in \Omega^2$, but none of the sign changes $\mathbf{c}_i$, $i \in \Omega$. 
\item When $\Omega$ is a mixed orbit, the group $W_\Theta$ contains none of the sign changes $\mathbf{c}_i$, $i \in \Omega$, but: 
\begin{enumerate}[(i)]
\item if $i$ and $j$ are integers that lie either both in  $\Omega^{\mathrm{per}}$, or both in $\Omega^{\mathrm{flip}}$, we have $(ij) \in W_\Theta$;
\item if $i$ and $j$ have the property that one of them lies in  $\Omega^{\mathrm{per}}$ and the other lies in $\Omega^{\mathrm{flip}}$, then we have $ (ij)\mathbf{c}_i \mathbf{c}_j \in W_\Theta$, but $(ij) \notin W_\Theta$ and $(ij)\mathbf{c}_i \notin W_\Theta$.
\end{enumerate}
Given two integers $i$ and $j$ in a mixed orbit $\Omega$, the fact that we are in case (i) or case (ii) above depends only on the (nonordered) pair $\{i,j\}$, and not on the choice of $i_\Omega$.
\end{enumerate}

\begin{lemm} \label{det_wt} The group $W_\Theta$ is generated by
\begin{itemize}
\item[$\bullet$] the transpositions $(ij)$ for which $n_i=n_j$ and $\sigma_i$ is a twist of $\sigma_j$,
\item[$\bullet$] the sign changes $\mathbf{c}_i$ for which $\widecheck{\sigma}_i$ is a twist of $\sigma_i$,
\item[$\bullet$] and the products $(ij)\mathbf{c}_i\mathbf{c}_j$ such that  $\widecheck{\sigma}_i$ is a twist of $\sigma_j$, but not of  $\sigma_i$, and not of $\widecheck{\sigma}_j$ either.
\end{itemize}
 \end{lemm}

 \begin{proof} 
 Let us start with an element $w$ in $W_\Theta$ and write $w = \mathbf{s}\mathbf{c}$, where $\mathbf{s} \in \mathfrak{S}$ is a permutation of $\{1, \dots, r\}$ and  $\mathbf{c} \in (\mathbb{Z}/2\mathbb{Z})^r$ belongs to the sign-change subgroup  $\mathfrak{C}$. The element $\mathbf{c}$ reads $\mathbf{c} = \gamma_1 \dots \gamma_r$, where for $i \in \{1, \dots, r\}$, the element $\gamma_i$ is either the identity or the sign change $\mathbf{c}_i$. 
 
Note that if $w$ lies in $W_\Theta$, then for every orbit $\Omega$ in $\{1, \dots, r\}$, we must have $\mathbf{s}(\Omega) = \Omega$. Set $\mathbf{s}_\Omega= \mathbf{s}_{|\Omega}$ and $\mathbf{c}_\Omega = \prod \limits_{i \in \Omega} \gamma_i$ for each orbit $\Omega$, and note that when $\Omega$, $\Omega'$ are two distinct orbits, the elements $\mathbf{s}_\Omega$ and $\mathbf{c}_{\Omega'}$ commute. We can thus decompose  $w$ as $\prod \limits_{\Omega} w_\Omega := \prod\limits_{\Omega} \mathbf{s}_\Omega \mathbf{c}_\Omega$, where the elements $w_{\Omega}=\mathbf{s}_\Omega \mathbf{c}_\Omega$ commute with one another and  all do lie in  $W_\Theta$. We shall prove that each of the elements $w_\Omega$ satisfies the conclusion of the Lemma. 
 
 Let us therefore fix an orbit $\Omega$ and inspect the situation, depending on the type of $\Omega$ : 
 \begin{enumerate}
\item If $\Omega$ is a pure orbit of type I, then each of the  $\mathbf{c}_i$, $i \in \Omega$, belongs to $W_{\Theta}$, and $w_\Omega$ is a product of elements of the two first kinds mentioned in the statement of the Lemma. So it satisfies the desired conclusion.
\item If $\Omega$ is a pure orbit of type II, then we must have $\mathbf{c}_{\Omega} = 1$, so  $w_\Omega$ must be a product of permutations of the first kind, and satisfies the desired conclusion.
\item Let us now consider the case where $\Omega$ is a mixed orbit, and write it as a union $\Omega = \Omega^{\mathrm{per}} \cup \Omega^{\mathrm{flip}}$. Let us decompose $\mathbf{s}_{\Omega}$ as a product of cycles with disjoint supports, writing $s_{\Omega} = \alpha_1 \dots \alpha_s$ where the $\alpha_i$ are cycles in $\mathfrak{S}$ whose supports are all disjoint and all contained in  $\Omega$. We can then rewrite $ w_\Omega$ as $\prod\limits_{k=1}^{s} \left( \alpha_k \mathbf{c}_{\alpha_k}\right)$, where $\mathbf{c}_{\alpha_k} = \prod \limits_{\ell \in \mathrm{Supp}(\alpha_k)} \gamma_\ell$. Here again, the various $\alpha_k \mathbf{c}_{\alpha_k}$ commute with one another, and we must have  $\alpha_k \mathbf{c}_{\alpha_k} \in W_\Theta$ for all $k$. Our last task is to show that for all $k$, the element $\alpha_k \mathbf{c}_{\alpha_k}$ satisfies the conclusion of the Lemma. 

\noindent Let us then inspect one cycle $\alpha$ among $\alpha_1, \dots, \alpha_k$.
\begin{enumerate}
\item[3.1.] If the support of $\alpha$ is either entirely contained in  $\Omega^{\mathrm{per}}$ or entirely contained in  $\Omega^{\mathrm{flip}}$, then 
\begin{itemize}
\item[$\bullet$] every transposition that exchanges two elements of  $\mathrm{Supp}(\alpha)$ must lie in  $W_{\Theta}$, so $\alpha$ must be a product of those transpositions mentioned in the Lemma,
\item[$\bullet$] and we must have $\gamma_\ell = 1$ for all $\ell \in \mathrm{Supp}(\alpha)$, because the orbit $\Omega$ is mixed. 
\end{itemize}
We see that in that case, we have $\alpha \mathbf{c}_{\alpha} = \alpha$, and that this element is a product of transpositions which all lie in $W_\Theta$. So it satisfies the conclusion of the Lemma.
\item[3.2.] If the support of $\alpha$ has nonempty intersection with both  $\Omega^{\mathrm{per}}$ and $\Omega^{\mathrm{flip}}$, then because $\alpha$ is a cycle, there must exist an element {$i \in \mathrm{Supp}(\alpha)$} with the property that  $(i {\alpha}(i))\mathbf{c}_i\mathbf{c}_{ {\alpha}(i)}$ lies in $W_{\Theta}$. 

We see, then that $W_{\Theta}$ contains the element  
\[ x =\left[(i\alpha(i))\mathbf{c}_i\mathbf{c}_{\alpha(i)} \right] \cdot\left[ \alpha \mathbf{c}_{\alpha}\right].\]
If we prove that $x$ must itself be a product of elements of the kind mentioned in the Lemma, then our task will be completed. Now, the projection of $x$ in the subgroup $\mathfrak{S}$ is a permutation with support is contained in that of $ {\alpha}$, but that admits $i$ as a fixed point. That permutation is, therefore, a product of disjoint cycles with length smaller than that of the cycle $\alpha$, and those cycles are all supported in $\Omega$. 

The desired conclusion then follows by induction on the maximal length of cycles in the projection of $x$. 
 \end{enumerate} 
 \end{enumerate}
\end{proof} 
 \subsubsection{Existence of a fixed point} \label{verif_pt_fixe}
 
We now come back to the representation
\[ \sigma = \sigma_1 \otimes \dots \otimes \sigma_r \otimes \tau\] 
and take up the above partition of  $\{1, \dots, r\}$ into orbits $\Omega$. We shall attach to each orbit $\Omega$ a unitary unramified character of the group $\prod \limits_{i \in \Omega} \mathrm{GL}(n_i,F)$, then twist $\sigma$ by those characters to obtain a fixed point for the action of $W_\Theta$, {thereby proving Theorem \ref{th_pt_fixe} for symplectic groups}. 

Fix an orbit $\Omega$, choose  $i_{\Omega}$ in $\Omega$, and observe the type of $\Omega$:

\begin{enumerate}
\item If $\Omega$ is a pure orbit of type I, then for all $i \in \Omega$, there exists $\chi_i \in \mathcal{X}_u(\mathrm{GL}(n_i,F))$ such that $\sigma_i \simeq \sigma_{i_{\Omega}} \otimes \chi_i$. 

Furthermore, the fact that $\Omega$ is type-I means that $\mathbf{c}_{i_\Omega} \in W_{\Theta}$, so there exists a character $\lambda_{i_\Omega} \in \mathcal{X}_u(\mathrm{GL}(n_i,F))$ such that
 \begin{equation} \label{un} \widecheck{\sigma_{i_\Omega}} \simeq \sigma_{i\Omega} \otimes \lambda_{i_\Omega}.\end{equation}
Calling in Lemma \ref{contrag}, we know that there is a unitary unramified character  $\nu_{{\Omega}}$ such that 
\begin{equation} \label{deux} \widecheck{(\sigma_{i_\Omega} \otimes \nu_{\Omega})}\simeq \sigma_{i_\Omega} \otimes \nu_{\Omega}.\end{equation}
 Combining \eqref{un} and \eqref{deux}, a simple calculation shows that 
 \begin{equation} \label{trois} \forall i \in \Omega, \quad  \widecheck{\left[\sigma_{i} \otimes (\nu_{\Omega} \chi_{i}^{-1})\right]}\simeq \sigma_{i} \otimes (\nu_{\Omega} \chi_i^{-1}).\end{equation}
We then define 
\[ \chi_{\Omega} = \bigotimes \limits_{i \in \Omega} (\nu_\Omega \chi_i^{-1}).\]
\item If $\Omega$ is a pure orbit of type II, then for every $i \in \Omega$,  there exists $\chi_i \in \mathcal{X}_u(\mathrm{GL}(n_i))$ such that  $\sigma_i \simeq \sigma_{i_{\Omega}} \otimes \chi_i$. We define $\chi_\Omega = \bigotimes \limits_{i \in \Omega} \chi_i^{-1}$.

\item Finally assume that  $\Omega$ is a mixed orbit. Decompose $\Omega$ as $\Omega^{\mathrm{per}} \cup \Omega^{\mathrm{flip}}$, as in \S \ref{w_theta}. For every $i \in \Omega^{\mathrm{per}}$, we have $\sigma_i \simeq \sigma_{i_\Omega} \otimes \chi_i$, where $\chi_i \in \mathcal{X}_u(\mathrm{GL}(n_i,F))$; on the other hand, for every $i\in \Omega^{\mathrm{flip}}$, we have $\sigma_i \simeq \widecheck{\sigma_{i_\Omega}} \otimes \chi_i$, where $\chi_i \in \mathcal{X}_u(\mathrm{GL}(n_i))$. In that case, we set $\chi_\Omega = \bigotimes \limits_{i \in \Omega} \chi_i^{-1}.$

\end{enumerate}

Now that a character $\chi_\Omega$ has been defined for every orbit $\Omega \subset \{1, \dots, r\}$, we set
\[ \chi_{\mathrm{tot}} = \bigotimes \limits_{\Omega} \chi_\Omega\]
and define
\[ \sigma_{\mathrm{end}} = \sigma \otimes \chi_{\mathrm{tot}}.\]

  \begin{prop} \label{fixe_ortho} The action of $W_{\Theta}$ on $\mathcal{O}$ admits $\sigma_{\mathrm{end}}$ as a fixed point. \end{prop}

\begin{proof} We start with the decomposition
\[ \sigma_{\mathrm{end}} = \sigma'_1  \otimes \dots \sigma'_r \otimes \tau,\]
with notation that should be obvious at this stage. Given the construction of the character $\chi_{\mathrm{tot}}$, if we consider an orbit $\Omega \in \{1, \dots, r\}$ and fix an element  $i_\Omega$ on $\Omega$, then we can make the following remarks:
 \begin{itemize}
\item[$\bullet$] If $\Omega$  is pure of type I, then for every $i \in \Omega$, we have $\sigma'_i \simeq \sigma'_{i_\Omega}$, and also  $\widecheck{\sigma'_i} \simeq \sigma'_i$.
\item[$\bullet$] If $\Omega$ is pure of type II, then for every $i \in \Omega$, we have $\sigma'_i \simeq \sigma'_{i_\Omega}$, while $\widecheck{\sigma'_i}$ is not a twist of~$\sigma'_i$.
\item[$\bullet$] If $\Omega$ is mixed and if $i \in \Omega$, then $\sigma'_i \simeq \sigma'_{i_\Omega}$ when  $i \in \Omega^{\mathrm{per}}$, and $\sigma'_i \simeq  \widecheck{\sigma'_{i_\Omega}}$ when $i \in \Omega^{\mathrm{flip}}$.
\end{itemize}
Turning to the generators of  $W_{\Theta}$ described in \ref{det_wt}, we see that each of them acts trivially on  $\sigma_{\mathrm{end}}$.  The Proposition follows.\end{proof}

  \subsection{Odd orthogonal groups $\mathrm{SO}(2n+1,F)$} \label{so_2n+1}Consider the group $G=\mathrm{SO}(2n+1,F)$, where $n$ is a nonnegative integer. The Levi subgroups of $G$ and their Weyl groups have the same form as their counterparts in $\mathrm{Sp}(2n,F)$: if $M$ is a Levi subgroup of $G$, then
\[ M \simeq \mathrm{GL}(n_1,F) \times \dots \times \mathrm{GL}(n_r,F) \times M'\]
where $M'$ is either trivial or an orthogonal group $\mathrm{SO}(2q+1,F)$, and $n_1 + \dots + n_r + q = n$ (we still use the convention  $q=0$ when $M'$ is trivial). Furthermore, $M$ is conjugate with the block-diagonal subgroup
\begin{equation*}M_{\mathrm{std}} = \left\{ \mathrm{Diag}\left(g_1, \dots, g_r, \ \gamma , \ ^t\!(g_r)^{-1}, \dots, ^t\!(g_1)^{-1}\right), \quad  g_1 \in \mathrm{GL}(n_1,F), \dots,  g_r \in \mathrm{GL}(n_r,F), \gamma \in M'\right\}.\end{equation*}
where  $^t\!(\dots)$ denotes ordinary transposition.

Everything we said of symplectic groups can be said for orthogonal groups: we need only replace the off-diagonal transpose $\tau$ by the ordinary transpose in the argument. Therefore Theorem \ref{th_pt_fixe} also holds for odd orthogonal groups.

  \subsection{Split orthogonal groups $\mathrm{SO}(2n,F)$} \label{so_2n} Let $n$ be a positive integer. We now consider the split group $G=\mathrm{SO}(2n,F)$. The structure of Weyl groups is slightly subtler than that we met for odd orthogonal groups, and will make some slight changes necessary in our arguments. As far as Levi subgroups are concerned, the description perfectly matches the odd case: every Levi subgroup $M$ of $G$ satisfies
\[ M \simeq \mathrm{GL}(n_1,F) \times \dots \times \mathrm{GL}(n_r,F) \times M'\]
where $M'$ is either trivial or an orthogonal group $\mathrm{SO}(2q,F)$, with $n_1 + \dots + n_r + q = n$ (we still use the convention $q=0$ for trivial $M'$). As before, $M$ is conjugate with 
\begin{equation*} M_{\mathrm{std}} = \left\{ \mathrm{Diag}\left(g_1, \dots, g_r, \ \gamma , \ ^t\!(g_r)^{-1}, \dots, ^t\!(g_1)^{-1}\right), \quad  g_1 \in \mathrm{GL}(n_1,F), \dots,  g_r \in \mathrm{GL}(n_r,F), \gamma \in M'\right\},\end{equation*}
and we shall assume that $M=M_{\mathrm{std}}$. 

It is in the structure of the Weyl group $W(G,M)$ that a subtlety manifests itself. We need to distinguish two cases:
\begin{itemize}
\item[\emph{Case $1$}.]   If $M'$ is trivial, the Weyl group $W(G,M)$ is generated by 
\begin{enumerate}[(a)]
\item the  transpositions $(ij)$ such that $n_{i}=n_j$,
\item the sign changes $\mathbf{c}_i$ for which $n_i$ is even,
\item and the products $\mathbf{c}_i \mathbf{c}_j$, $i \neq j$, such that both $n_i$ and $n_j$ are odd.
\end{enumerate}
\item[\emph{Case $2.$}] When $M'$ reads $\mathrm{SO}(2q,F)$ with $q \geq 1$, we must bring in a new sign change $\mathbf{c}_{\mathrm{spec}}$. It operates on the component $\mathrm{SO}(2q,F)$, by conjugation with the matrix 
\begin{equation} \label{matrice_c} \mathbf{C} = {\begin{pmatrix} I_{q-1} & & & \\ & 0 & 1 & \\ & 1 & 0 & \\ & & & I_{q-1}\end{pmatrix}}.\end{equation}
(which lies in $\mathrm{O}(2q,F)$ but not in $\mathrm{SO}(2q,F)$). The group $W(G,M)$ is then generated by 
\begin{enumerate}[(a)]
\item the transpositions  $(ij)$ such that $n_{i}=n_j$,
\item the sign changes $\mathbf{c}_i$ for which $n_i$ is even, 
\item[(c')] and the products $\mathbf{c}_i \mathbf{c}_{\mathrm{spec}}$, for which $n_i$ is odd. 
\end{enumerate}
Note that the product of two elements of type (c') is an element of type (c) from case 1.
\end{itemize}
We can take up from the case of symplectic groups the decomposition  $W(G,M) = \mathfrak{S} \ltimes \mathfrak{C}$, where $\mathfrak{S}$ is isomorphic with the subgroup of $\mathfrak{S}_r$ generated by those transpositions $(ij)$ such that  $n_i = n_j$, and where $\mathfrak{C}$  is a $2$-group, which in case $1$ is generated  by the elements of type (b)-(c), and in case $2$ is generated by the elements of type (b)-(c'). 

Now let 
\[ \sigma = \sigma_1 \otimes \dots \otimes \sigma_r \otimes \tau\] 
be a discrete series representation of $M$, where  $\sigma_i \in \mathcal{E}_2(\mathrm{GL}(n_i,F))$ for all $i \in \{1, \dots, r\}$, and where $\tau \in \mathcal{E}_2(M')$. In order to determine the group $W_{\Theta}$ for the pair $\Theta=(M, \mathcal{X}_u(M)\cdot \sigma)_{G}$, only very slight adaptations are needed from the case of $\mathrm{SO}(2n+1,F)$. Recall that for all  $i \in \{1, \dots, r\}$, we write $\widecheck{\sigma_i}$ for the representation $g \mapsto \sigma(^t\!g^{-1})$ of $\mathrm{GL}(n_i,F)$. When $M'$ is nontrivial, we shall also denote by   $\widetilde{\tau}$ the representation  $g \mapsto \tau(\mathbf{C}g\mathbf{C})$, where $\mathbf{C}$ is the matrix \eqref{matrice_c}. 

By going through the arguments of \S \ref{w_theta}, we obtain the following description of $W_\Theta$.

\begin{lemm} \begin{enumerate}
\item When $M'$ is trivial, the group $W_\Theta$ is generated by
\begin{itemize}
\item[$\bullet$] the transpositions  $(ij)$ for which  $n_i=n_j$ and $\sigma_i$ is a twist of $\sigma_j$,
\item[$\bullet$]  the sign changes  $\mathbf{c}_i$ that have $n_i$ even and for which  $\widecheck{\sigma}_i$ is a twist of $\sigma_i$,
\item[$\bullet$]  the sign changes $\mathbf{c}_i\mathbf{c}_j$ for which $n_i$ and $n_j$ are odd, $\widecheck{\sigma}_i$ is a twist of $\sigma_i$, and $\widecheck{\sigma}_j$ is also a  twist of $\sigma_j$,
\item[$\bullet$] and the products $(ij)\mathbf{c}_i\mathbf{c}_j$ such that $\widecheck{\sigma}_i$ is a twist of $\sigma_j$, but neither of $\sigma_i$ nor of $\widecheck{\sigma}_j$.
\end{itemize}
\item When $M'=\mathrm{SO}(2q,F)$ ($q \geq 1$) and  $\widetilde{\tau}$ is not a twist of  $\tau$, the group $W_\Theta$ is generated by the same elements as in case 1 above. 

\item When $M'=\mathrm{SO}(2q,F)$ ($q \geq 1$) and $\widetilde{\tau}$ is a twist of $\tau$, the group $W_\Theta$ is generated by
\begin{itemize}
\item[$\bullet$] the transpositions  $(ij)$ for which  $n_i=n_j$ and $\sigma_i$ is a twist of $\sigma_j$,
\item[$\bullet$]  the sign changes $\mathbf{c}_i $ for which $n_i$ is even and $\widecheck{\sigma}_i$ is a twist of $\sigma_i$,
\item[$\bullet$]  the products $\mathbf{c}_i \mathbf{c}_{\mathrm{spec}}$ for which $n_i$ is odd and $\widecheck{\sigma}_i$ is a twist of $\sigma_i$,
\item[$\bullet$] and the products $(ij)\mathbf{c}_i\mathbf{c}_j$ such that $\widecheck{\sigma}_i$ is a twist of $\sigma_j$, but neither of $\sigma_i$ nor of $\widecheck{\sigma}_j$.
\end{itemize}

\end{enumerate}
 \end{lemm} 
In cases $1$ and $2$, the existence of a fixed point can be obtained by going through the same discussion we went through for symplectic groups, using the various kinds of  $\Omega \subset \{1, \dots, r\}$ for the equivalence relation $\sim$ of \S \ref{w_theta}. The only necessary adaptation, when discussing an orbit $\Omega$, is that one must distinguish between the parity of the common value for the $n_i$,  $i\in\Omega$. But the discussion is identical, so we omit the details.

In case $3$, we must insert the following observation.

\begin{lemm} \label{torsion} Let $\tau$ be a discrete series representation of $G=\mathrm{SO}(2q,F)$, $q \geq 1$. Let us write $\widetilde{\tau}$ for the representation $g \mapsto \tau(\mathbf{C}g\mathbf{C})$, where $\mathbf{C}$ is the matrix \eqref{matrice_c}. 

If there exists a unitary unramified character $\chi \in \mathcal{X}_u(G)$ such that $\widetilde{\tau} \simeq \tau \otimes \chi$, then there exists a character  $\nu \in \mathcal{X}_u(G)$ such that $\widecheck{\tau\otimes \nu} \simeq \tau \otimes \nu$.  \end{lemm}

\begin{proof} If $G = \mathrm{SO}(2q,F)$ and if the integer $q$ satisfies $q\geq 2$, then the group $ \mathcal{X}_u(G)$ is trivial, showing that the Lemma is also trivial. Indeed, recall that if  $G^0$ is the common kernel for all unramified characters of $G$, then we have $G^0 = G$ whenever the center of $G$ is compact (see \cite[\S V.2.6, p. 163]{Renard}). For $G = \mathrm{SO}(2q,F)$ with $q \geq 2$, we have $Z(G) = \{-I_{2q}, I_{2q}\}$ (see \cite[\S II.6]{Dieudonne}), so $G^0 = G$ and $\mathcal{X}_u(G) = \{1\}$. 

We need only check, then, that the Lemma holds for $G=\mathrm{SO}(2,F)$. But that group is abelian, so $\tau$ is just a unitary character of $G$. Under the hypothesis of the Lemma, we have  $\tau = \widetilde{\widetilde{\tau}} = \widetilde{\tau \chi} = \widetilde{\chi} \chi \tau$, and since $\tau$ is a character, that can happen only if $\widetilde{\chi} \chi = 1$, that is, only if $\widetilde{\chi} = \chi^{-1}$. Setting  $\nu = \chi^{1/2}$, we then see $\widetilde{\tau \nu} = \widetilde{\tau} \chi^{-1/2} = \tau \chi \chi^{-1/2} = \tau \nu$, as announced.  \end{proof}

With Lemma \ref{torsion} in hand, we can complete our discussion of case $3$, where $\widetilde{\tau} \simeq \tau \otimes \chi$, $\chi \in \mathrm{SO}(2q,F)$. Consider the unitary unramified character $\tilde{\nu} = 1 \otimes \dots \otimes 1 \otimes \nu$ of $M = \mathrm{GL}(n_1,F) \times \dots \times \mathrm{GL}(n_r,F) \times \mathrm{SO}(2q,F)$, and twist  $\sigma$ by $\tilde{\nu}$, setting ${\sigma}' = \sigma \otimes \tilde{\nu}$ . If we apply to the element ${\sigma}'$ of the orbit $\mathcal{O}$ the very argument of \S \ref{verif_pt_fixe} , we obtain the existence in $\mathcal{O}$ of a fixed point for $W_{\Theta}$ (we emphasize that no adaptation of \S \ref{verif_pt_fixe}, in particular no discussion of parity, is  needed here).  This completes our proof that Proposition \ref{fixe_ortho} also holds for $G=\mathrm{SO}(2n,F)$.

  \subsection{Quasi-split orthogonal groups $\mathrm{SO}^\star(2n,F)$} For the quasi-split but non-split classical group $\mathrm{SO}^\star(2n,F)$, $n \geq 1$, the structure theory detailed in \S \ref{so_2n} extends verbatim, replacing split forms $\mathrm{SO}(2q,F)$ by quasi-split forms $\mathrm{SO}^\star(2q,F)$ everywhere: see \cite[\S 2]{GoldbergShahidi} and \cite[Appendix]{ChoiyGoldberg}. Therefore the results of \S \ref{so_2n} are also valid for that case.

  \subsection{Unitary groups $\mathrm{U}(n,n)$ and $U(n+1,n)$} \label{u_n}
 
 Let $E/F$ be a quadratic extension, and let $x \mapsto \bar{x}$ denote the nontrivial element in the Galois group $\mathrm{Gal}(E/F)$; we fix an element $\beta \in E$ such that $\overline{\beta}=-\beta$. We define two matrices in  $\mathrm{GL}(2n, E)$ by setting $J_n = \begin{pmatrix}  & \beta I_n \\ \beta I_n &  \end{pmatrix}$ and $J'_n = \begin{pmatrix} &  & \beta I_n \\ & 1 & \\  \beta I_n & & \end{pmatrix}$, and write $G$ for one of the two groups
  \[ U(n,n) = \left\{ g \in \mathrm{GL}(2n,E) \ : \ ^t\bar{g} J_n g = J_n \right\}, \quad \quad \quad  U(n+1,n) = \left\{ g \in \mathrm{GL}(2n,E) \ : \ ^t\bar{g} J_n' g = J_n' \right\}. \]
For every Levi subgroup $M$ of $G$, we have (see   \cite{GoldbergUN})
\[ M \simeq \mathrm{GL}(n_1,E) \times \dots \times \mathrm{GL}(n_r,E) \times M',\]
where $M'$ is either trivial or one of the two groups $U(q,q)$, $U(q+1,q)$ (here $M'$ must be of the same type as $G$), and $n_1 + \dots + n_r + q = n$. The group $M$ is conjugate with 
\begin{equation*}  M_{\mathrm{std}} = \left\{ \mathrm{Diag}\left(g_1, \dots, g_r, \ \gamma , \ \varepsilon(g_r) \dots, \varepsilon(g_1)\right),  \quad g_1 \in \mathrm{GL}(n_1,E), \dots,  g_r \in \mathrm{GL}(n_r,E), \gamma \in M'\right\},\end{equation*}
where $\varepsilon(\dots)$ denotes the involution $g \mapsto ^t{\bar{g}}^{-1}$ of $\mathrm{GL}(n,E)$.

Everything we said of symplectic groups then applies to that case: we need only replace the off-diagonal transpose $\tau$ with the involution $\varepsilon$ in the argument. {We have now completed our proof that Theorem \ref{th_pt_fixe} holds for all quasi-split classical groups. }

\section{Inclusions of $R$-groups, good and bad fixed points, and examples}\label{bons_pts_fixes}~

Let $G$ be one of the groups considered in \S \ref{exist_pt_fixe}; fix a Levi subgroup $M$ of $G$ and a representation $\sigma \in \mathcal{E}_2(M)$. By the results of \S \ref{exist_pt_fixe}, we may assume that $\sigma$ is a fixed point for the action of $W_\Theta$. The aim of this section is to determine whether the conditions in Assumption \ref{hyp_inclusions} can be satisfied for $\sigma$. Properties (a) and (b) in Assumption \ref{hyp_inclusions} turn out to be quite strong; however, we shall eventually give, in Theorem \ref{conclusion_inclu}, a necessary and sufficient condition on $\sigma$ for the conditions to be satisfied. 

\subsection{Example of a bad fixed point}\label{iwahori} We start by pointing out a simple example in which Assumption \ref{hyp_inclusions} cannot be satisfied, and indicate why there is little hope for a statement analogous to Theorem \ref{principal} in that case. 

Let $G$ be the group $\mathrm{Sp}(2n,F)$, $n \geq 2$, and let us consider the Iwahori-spherical block of $\widehat{G}_{\mathrm{temp}}$. This is the component built from the minimal Levi subgroup 
\[ M \simeq \mathrm{GL}(1,F)^n \]
where $M$ is the torus $\left\{ \mathrm{Diag}\left(a_1, \dots, a_n, a_n^{-1}, \dots, a_{1}^{-1}\right) \ : \ (a_1, \dots, a_n) \in (F^\times)^n\right\}$, and from the discrete series representation
\[ \sigma = \text{ the trivial character of $M$. }\]
The representation $\sigma$ is a fixed point for the action of the whole group $W(G,M) = \mathfrak{S}_n \ltimes (\mathbb{Z}/2\mathbb{Z})^n$. Keys proved \cite[Theorem $C_n$]{Keys} that the $R$-group  $R_\sigma$ is trivial (see also the discussion by Goldberg in \cite[p. 1140]{GoldbergSPN}). 

Now, let $\nu_\star$ be the unique unramified character of $\mathrm{GL}(1,F)=F^\times$ such that $(\nu_\star)^2 = 1$ and $\nu_\star \neq 1$. Define an unramified character $\chi \in \mathcal{X}_{u}(M)$ as $(\nu_\star) \otimes 1 \otimes \dots \otimes 1$. Then at the point $\sigma \otimes \chi$ of the unramified orbit of $\sigma$, we have $R_{\sigma \otimes \chi} = \mathbb{Z} / 2\mathbb{Z}$ (by \cite[Theorem $C_n$]{Keys}). So Assumption \ref{hyp_inclusions} definitely fails in that case.

We next point out that for every $\chi \in \mathcal{X}_u(M)$, the unramified twist $\sigma \otimes \chi$ cannot  at the same time be a fixed point for $W(G,M)$ and satisfy conditions (a) and (b) in Assumption \ref{hyp_inclusions}. Indeed, assume $\chi = \chi_1 \otimes \dots \otimes \chi_n$ has the property that $\sigma \otimes \chi$ is a fixed point for $W(G,M)$. The invariance under the permutation subgroup $\mathfrak{S}_n$ means that we must have $\chi_i = \chi_j$ for all $i$, $j$. The invariance under the sign changes means that the common value $\chi$ of all $\chi_i$ , must satisfy $\chi^2 = 1$. So we either have $\chi=1$ or $\chi  = \nu_\star$. In either case, by  \cite[Theorem $C_n$]{Keys}, the $R$-group $R_{\sigma \otimes \chi}$ is trivial. But the argument used above shows that by twisting with $(\nu_\star) \otimes 1 \otimes \dots \otimes 1$, we obtain a point with nontrivial $R$-group. 

{There is, therefore, no hope for a fixed point satisfying Assumption \ref{hyp_inclusions}. As we already discussed in the Introduction (see the discussion just before Assumption \ref{hyp_inclusions}), this leaves little hope that the  Iwahori-spherical block $C^\ast_r(G;\Theta)$ could satisfy a Morita-equivalence of the kind mentioned in \eqref{wass_reel}. Since the $2$-cocycle of $R_\sigma$ is always trivial for the component under discussion, we feel that there is little hope that an equivalence like that of  Theorem \ref{principal} can hold for the block $C^\ast_r(G;\Theta)$.}

\subsection{Description of good fixed points: symplectic groups} \label{bons} We return to the case where $G$ is one of the classical groups studied in \S \ref{exist_pt_fixe}, and $\widehat{G}_\Theta$ is an arbitrary connected component of $\widehat{G}_{\mathrm{temp}}$. Let $(M,\sigma)$ be an attached discrete pair in which $\sigma$ is a fixed point for $W_\Theta$. The aim of this section is to  describe necessary and sufficient conditions on $\sigma$ to satisfy the properties in Assumption \ref{hyp_inclusions}. 

\subsubsection{}  For the moment, let us assume that $G=\mathrm{Sp}(2n,F)$, and start from the general Levi subgroup 
\[ M \simeq \mathrm{GL}(n_1,F) \times \dots \times \mathrm{GL}(n_r,F) \times \mathrm{Sp}(2q,F)\]
of \eqref{desc_levi}, where $n_1, \dots, n_r,q$ satisfy  $n_1 + \dots + n_r + q = n$ (we do not exclude  $q=0$).
We fix a discrete series representation $\sigma$ of $M$, and decompose it as
\begin{equation} \label{sigma} \sigma = \sigma_1 \otimes \dots \otimes \sigma_r \otimes \tau\end{equation}
 where $\sigma_i \in \mathcal{E}_2(\mathrm{GL}(n_i,F))$ for $i \in \{1, \dots, r\}$, and where $\tau \in \mathcal{E}_2(\mathrm{Sp}(2q,F))$. 

Goldberg proved in \cite[\S 6]{GoldbergSPN} that in calculating the $R$-group for $\sigma$, a key role is played by special reducibility conditions attached to the various constituents $\sigma_i$. Let us establish some notation:

\begin{defi} Fix $i \in \{1, \dots, r\}$, and let $\rho$ be a discrete series representation of $\mathrm{GL}(n_i,F)$. We say that $\rho$ \emph{satisfies condition $\mathcal{C}_i(\tau)$} when:
\begin{equation} \label{ci} \tag{$\mathcal{C}_i(\tau)$} {\text{$\mathrm{Ind}_{M_iN_i}^{G_i}(\rho \otimes \tau)$ is reducible }}\end{equation}
where $G_i$ is the symplectic group $\mathrm{Sp}(2n_i + 2q,F)$,  {$M_i$ is the standard Levi subgroup isomorphic with $\mathrm{GL}(n_i,F) \times \mathrm{Sp}(2q,F)$,  $M_iN_i$ is any parabolic subgroup of $G_i$ with Levi factor $M_i$, and the induced representation $\mathrm{Ind}_{M_iN_i}^{G_i}(\rho \otimes \tau)$ is constructed in one of the equivalent ways described in \S \ref{induites}. }
\end{defi} 
\noindent(For the reader's convenience, we mention that $\mathcal{C}_i(\tau)$ for $\rho$ is denoted $\mathcal{X}_{n_i, q}(\rho \otimes \tau)$ in \cite{GoldbergSPN}).

\begin{rema} \label{dualite} ~
\begin{itemize}
\item[$\bullet$]  If a representation $\rho$ satisfies condition \ref{ci}, then we must have $\widecheck{\rho}=\rho$ (see \cite[p. 1140]{GoldbergSPN}). 
\item[$\bullet$]  In some special cases, it is also sufficient that $\rho$ satisfy $\widecheck{\rho}=\rho$. This happens when $q=0$ and $n_i=1$, as recalled in \S \ref{iwahori}. Shahidi \cite{ShahidiDuke} also proved that $\widecheck{\rho}=\rho$ implies \ref{ci} for $\rho$ when $q=0$, $n_i$ is odd and $\rho$ is supercuspidal unitary (see also \cite[\S 7]{GoldbergSPN}). 
\item[$\bullet$]  In general, however, condition \ref{ci} appears to be of an arithmetical nature: when $q$ and $n_i$ are arbitrary, but $\rho$ is supercuspidal,  Goldberg and Shahidi give in \cite{GoldbergShahidi} a precise criterion (using  the residue of the standard intertwining operators) for \ref{ci} to be satisfied.
\end{itemize} \end{rema}

\subsubsection{} We now start from the discrete series representation $\sigma$ of \eqref{sigma}, assume that $\sigma$ is a fixed point for the action of $W_\Theta$, and embark on our search for conditions so that $\sigma$ satisfies Assumption \ref{hyp_inclusions}. We will start by spelling out what Goldberg's determination of $R$-groups tells us for $R_\sigma$ under the present fixed-point hypothesis.

We take up, from \S \ref{orbites}, the partition of $\{1, \dots, r\}$ into orbits $\Omega$. When $\sigma$ is a fixed point, we can extract from \S \ref{w_theta} the following observations:

\begin{enumerate}
\item If $\Omega$ is a pure orbit of type I, then 
\begin{itemize} 
\item for all $i$, $j$ in $\Omega$, we have $\sigma_i \simeq \sigma_j$, 
\item and we also have $\widecheck{\sigma_i} \simeq \sigma_i$ for all $i \in \Omega$.
\end{itemize}
\item If $\Omega$ is a pure orbit of type II, then
\begin{itemize} 
\item for all $i$, $j$ in $\Omega$, we have $\sigma_i \simeq \sigma_j$, 
\item but we have  $\widecheck{\sigma_i} \not\simeq \sigma_i$ for all $i \in \Omega$.
\end{itemize} 
\item Finally, if $\Omega$ is a mixed orbit, 
\begin{itemize} 
\item we have  $\widecheck{\sigma_i} \not\simeq \sigma_i$ for all $i \in \Omega$, 
\item and for all $i,j$ in $\Omega$, we   have either $\sigma_i \simeq \sigma_j$ or $\sigma_i \simeq \widecheck{\sigma_j}$.
\end{itemize}
\end{enumerate} 

We now call in  the fact, due to Goldberg, that the $R$-group $R_\sigma$ is generated by pure sign changes. Recall from \S \ref{levi} that $W_\sigma$ is a subgroup of the semi-direct product $W(G,M) = \mathfrak{S} \ltimes \mathfrak{C}$, where $\mathfrak{S}$ consists of permutations and $\mathfrak{C}$ is the group of sign changes.

\begin{lemm}[Goldberg \cite{GoldbergSPN}] \label{signes} Let $w=\mathbf{s}\mathbf{c}$ be an element of $W_\sigma$. If $w$ lies in the $R$-group $R_\sigma$, then $\mathbf{s}=1$. \end{lemm}
\begin{proof} Combine Lemma 6.3 and Theorem 4.9 in \cite{GoldbergSPN}. \end{proof}

Now, suppose $\mathbf{c} \in \mathfrak{C}$ is an element of $R_\sigma$. We can decompose $\mathbf{c}$ as $\prod \limits_{\Omega} \mathbf{c}_{\Omega}$, where an element $\mathbf{c}_\Omega$ is attached to each orbit $\Omega$ and is a product of sign changes with support on $\Omega$. We observe that if $\Omega$ is either a pure orbit of type II or a mixed orbit, then we must have $\mathbf{c}_\Omega=1$.

We then insert the following reformulation of one of Goldberg's main results, Theorem 6.4 in \cite{GoldbergSPN}, adding an observation taken from its proof. 
\begin{theo}[Goldberg \cite{GoldbergSPN}] \label{gold_r} 
\begin{enumerate}[(i)]
\item Let $\sigma$ be a fixed point for the action of $W_\Theta$ on $\mathcal{O}$. The $R$-group $R_\sigma$ reads $R_\sigma \simeq (\mathbb{Z}/2\mathbb{Z})^d$, where $d$ is the number of orbits $\Omega \in \{1, \dots, r\}$ with the property that  one (equivalently all) of the representations $\sigma_i$, $i \in \Omega$, satisfies condition~\ref{ci}. Furthermore, for each of those orbits $\Omega$, the group $R_\sigma$ contains the element $\mathbf{c}^{\mathrm{full}}_\Omega = \prod \limits_{i \in \Omega} \mathbf{c}_i$. 
\item  {Let $\chi=\chi_1 \otimes \dots \otimes \chi_r$ be an unramified unitary character of $M$. The $R$-group $R_{\sigma \otimes \chi}$ is generated by the elements of the form $\mathbf{c}_A^{\mathrm{full}}=\prod \limits_{i \in A} \mathbf{c}_i$, where $A$ is a subset of $\{1, \dots, r\}$ satisfying $\sigma_i \otimes \chi_i \simeq \sigma_i \otimes \chi_j$ for $i, j \in A$ and such that \ref{ci} is valid for one (equivalently all) of the representations $\sigma_i \otimes \chi_i$, $i \in A$. }
\end{enumerate}\end{theo}

We shall give a name to the particular orbits that have the property in {Theorem~\ref{gold_r}(i)}, calling them \emph{$R_\sigma$-relevant}. We note, from Remark~\ref{dualite} and Lemma~\ref{signes}, that an orbit $\Omega$ cannot be $R_\sigma$-relevant unless it is pure of type I.

\subsubsection{} \label{singlet} We can now point out a simple situation in which part of Assumption \ref{hyp_inclusions} is satisfied. Assume that $\sigma$ has the property that all $R_\sigma$-relevant orbits are singletons, and observe what happens to the $R$-group $R_\sigma$ when we twist by a unitary unramified character $\chi$ of $M$, \label{centre_compact} with\footnote{We note that a unramified unitary character of $M$ must be trivial on the block $M'=\mathrm{Sp}(2q,F)$, because that group has compact center, and therefore no nontrivial unramified character.} $\chi = \chi_1 \otimes \dots \otimes \chi_r \otimes 1$.  Attached to one of the orbits $\Omega \in \{i_\Omega\}$ of Theorem \ref{gold_r} is a character $\chi_{i_\Omega} \in \mathcal{X}_u(\mathrm{GL}(n_i,F))$. Then two things can happen: either $\sigma_{i_\Omega} \chi_{i\Omega}$ satisfies condition \ref{ci}, or it does not. If it does the contribution of the orbit $\Omega$ to the $R$-group does not change  {by Theorem \ref{gold_r}(ii)}, and is still a factor $\mathbb{Z}/2\mathbb{Z}$. If it does not, the contribution of $\Omega$ to the $R$-group becomes trivial,  {again by Theorem \ref{gold_r}(ii)}. 

In the special case where all $R_\sigma$-relevant orbits are singletons, we see that the $R$-group at $\sigma \otimes \chi$ is necessarily contained in the $R$-group at $\sigma$, thereby ensuring that condition (b) in Assumption \ref{hyp_inclusions} is satisfied. 

\subsubsection{} \label{superrel} However, we must note that an even stronger condition is necessary if $\sigma$ is also to satisfy condition (a) in Assumption \ref{hyp_inclusions}. Indeed, still assuming that the relevant orbits are singletons, assume that one of the $R_\sigma$-relevant orbits $\{i_\Omega\}$ has the property that there exists a nontrivial character $\chi_{i_\Omega} \in \mathcal{X}_u(\mathrm{GL}(n_{i_\Omega},F))$ such that $\sigma_{i_\Omega} \otimes \chi_{i_{\Omega}}$ does not satisfy condition \ref{ci}, but does satisfy $\widecheck{\sigma_{i_{\Omega}}\chi_{i_\Omega}} = {\sigma_{i_{\Omega}}\chi_{i_\Omega}}$.

Now let us twist $\sigma$ by the character $\chi$ of $M$ which acts by $\chi_{i_{\Omega}}$ on the block $\mathrm{GL}(n_{i_{\Omega}},F)$ of $M$, and trivially on the other blocks. The result is a representation that is again a fixed point for $W_\Theta$. As we noted above, the contribution of $\Omega$ to the $R$-group becomes trivial. But if we turn to the determination of the group $W'$, calling in the notion of ``relevant root of $(G,M)$'' from \S \ref{defs_rgroupe}, we see that in our situation, Goldberg's work \cite[Lemma 6.2.(3)]{GoldbergSPN} exhibits a root of $(G,M)$ which is $(\sigma \otimes \chi)$-relevant but \emph{not} $\sigma$-relevant. This proves that  $W'_{\sigma \otimes \chi}$ cannot be contained in $W'_{\sigma}$, thereby showing that condition (a) in Assumption \ref{hyp_inclusions} cannot be satisfied.

\subsubsection{} The discussion in \S \ref{superrel} leads us to introduce stronger notions of relevance, which we now spell out. 

Whenever $i \in \{1, \dots, r\}$, we write $\mathcal{T}(\sigma_i)$ for the finite group of unramified characters $\nu \in \mathcal{X}_u(\mathrm{GL}(n_i,F))$ such that $\widecheck{\sigma_i \otimes \nu} \simeq \sigma_i \otimes \nu$. Because $\sigma$ is a fixed point for $W_\Theta$, this group can be nontrivial only if we already have $\widecheck{\sigma_i} = \sigma_i$, and in that case $\mathcal{T}(\sigma_i)$ can equivalently be described as $\left\{ \nu \in \mathcal{X}_u(\mathrm{GL}(n_i,F)) \ : \ \sigma_i \otimes \nu^2 \simeq \sigma_i \right\}$.

\begin{defi} \label{relevant} Let $\sigma = \sigma_1 \otimes \dots \otimes \sigma_r \otimes \tau$ be a fixed point for the action of $W_{\Theta}$, and let $\Omega \subset \{1, \dots, r\}$ be an orbit for the relation $\sim$ of \S \ref{orbites}. Recall that  $\sigma_i \simeq \sigma_j$ for all $i,j \in \Omega$.
\begin{itemize}
\item[$\bullet$] We say that $\Omega$ is \emph{weakly $R_\sigma$-relevant} when, for  $i \in \Omega$, there is, among  the representations $\sigma_i \otimes \nu$, $\nu \in \mathcal{T}(\sigma_i)$, at least one that satisfies condition \ref{ci}.
\item[$\bullet$] We say that $\Omega$ is \emph{$R_\sigma$-super-relevant} when, for $i \in \Omega$, the representations $\sigma_i \otimes \nu$, $\nu \in \mathcal{T}(\sigma_i)$, all do satisfy condition \ref{ci}.
\item[$\bullet$] We say that $\Omega$ is \emph{special} when $\Omega$ is $R_\sigma$-relevant and for $i \in \Omega$, we have $\sigma_i \otimes \nu \simeq \sigma_i$ for all $\nu \in \mathcal{T}(\sigma_i)$.
\end{itemize}
\end{defi}

As before, the fixed-point hypothesis means that an orbit $\Omega$ cannot be weakly $R_\sigma$-relevant unless it is pure of type I.

\begin{rema}[Weak relevance and super-relevance] The super-relevance condition may appear quite strong. Note, however, from Remark \ref{dualite}, that we have some amount of information on it:

\begin{itemize}
\item[$\bullet$]  If the factor $M'=\mathrm{Sp}(2q,F)$ of $M$ is trivial, and if $\Omega$ is such that for one (equivalently all) $i \in \Omega$, the integer $n_i$ is odd and the representation $\sigma_i$ is supercuspidal, then weak relevance implies super-relevance.
\item[$\bullet$] If $M', \Omega$ and $\sigma_i$ are as above, but  $n_i$ is even, Shahidi's work shows that the super-relevance condition depends on the theory of twisted endoscopy. See the discussion in \cite[\S 3]{ShahidiDuke}.
\item[$\bullet$] If $M'$ is nontrivial,  super-relevance should depend on the behavior of residues as calculated in  Goldberg and Shahidi \cite{GoldbergShahidi}.
\end{itemize}

The above remarks certainly show that there can indeed exist super-relevant orbits. We should also mention that weak relevance does \emph{not} in general imply super-relevance. Consider for instance the Levi $M \simeq \mathrm{GL}(n,F)$ of the Siegel parabolic subgroup in $\mathrm{Sp}(2n,F)$, and assume that $n$ is even. Then  Kutzko and Morris prove in Proposition 7.4 of \cite{MorrisKutzko}  that for any depth-zero unitary supercuspidal representation $\sigma$ of $\mathrm{GL}(n,F)$ such that  $\check\sigma\simeq\sigma$ and $\sigma$ has trivial central character, the representation $\mathrm{Ind}_P^G(\sigma\otimes\chi)$ is reducible if and only if $\chi$ is the nontrivial quadratic unramified character of $\mathrm{GL}(n,F)$.

In that case, there is only one orbit $\Omega=\{1\}$, and the above shows that it is weakly $\sigma$-relevant without being  itself $\sigma$-relevant. \end{rema}

Let us turn now to special orbits. The notion will be short-lived, as a result of our next observation.

\begin{lemm} \label{speciales} Let $\sigma = \sigma_1 \otimes \dots \otimes \sigma_r \otimes \tau$ be a fixed point for the action of $W_{\Theta}$. If $\Omega \subset \{1, \dots, r\}$ is an orbit for the relation $\sim$ of \S \ref{orbites}, then $\Omega$ cannot be special. \end{lemm}

\begin{proof} Let $\Omega$ be an orbit. Since $\Omega$ can be special only if we have $\widecheck{\sigma_i} = \sigma_i$ for all $i \in \Omega$, we may as well assume this to be the case. Recall that $\sigma$ is a fixed point for $W_\Theta$, and therefore that all $\sigma_i$, $i \in \Omega$, are equivalent. Consider the finite group 
\[ \mathcal{T}(\sigma_i) = \left\{ \nu \in \mathcal{X}_u(\mathrm{GL}(n_i,F)) \ : \ \sigma_i \otimes \nu^2 \simeq \sigma_i \right\}\]
mentioned just before Definition \ref{relevant}. What we need to do is compare $\mathcal{T}(\sigma_i)$ with
\[ \mathcal{Q}(\sigma_i) = \left\{ \nu \in \mathcal{X}_u(\mathrm{GL}(n_i,F)) \ : \ \sigma_i \otimes \nu \simeq \sigma_i \right\}.\]
Of course $\mathcal{Q}(\sigma_i)$ is contained in $\mathcal{T}(\sigma_i)$. What we must show is that the two groups cannot be equal. To describe both groups, we now call in an observation by Bushnell and Kutzko \cite{BKLivre}. To the representation $\sigma_i$ of $\mathrm{GL}(n_i,F)$, Bushnell and Kutzko attach a simple type $(J,\lambda)$ (where $J$ is a compact open subgroup of $\mathrm{GL}(n_i,F)$ and $\lambda$ is an irreducible representation of $J$; see \cite[\S 6]{BKLivre}). In turn, $J$ and $\lambda$ determine an extension $E$ of $F$; we let $e(E|F)$ denote the ramification index of the extension $E$. This is a divisor of $n_i$. Bushnell and Kutzko proceed to show that a character $\nu \in \mathcal{X}_u(\mathrm{GL}(n_i,F))$ satisfies $\sigma_i \otimes \nu \simeq \sigma_i$ if and only if its order is finite and divides $m_{\sigma_i}=n_i/e(E|F)$ (see \cite[Lemma 6.2.5]{BKLivre}).

We see, thus, that $\mathcal{Q}(\sigma_i)$ is the group of characters $\nu$ whose order divides $m_{\sigma_i}$. But the Bushnell-Kutzko description above, together with the fact that $\nu$ lies in $\mathcal{T}(\sigma_i)$ if and only if $\nu^2$ lies in $\mathcal{Q}(\sigma_i)$, exhibits $\mathcal{T}(\sigma_i)$ as the group of  characters $\nu$ whose order divides $2m_{\sigma_i}$. 

Now, the group $G=\mathrm{GL}(n_i,F)$ admits unitary unramified  characters of any given order. In fact, given an integer $k>0$ and a uniformizer $\varpi_F$ for $F$, there is exactly one unitary unramified character $\chi$ of $G$ which sends the matrix $\mathrm{Diag}(\varpi_F, 1, \dots, 1)$ to $e^{\frac{2i\pi}{k}}$, and that character has order~$k$. Therefore, there does exist a character of order $k=2m_{\sigma_i}$. That one will lie in $\mathcal{T}(\sigma_i)$, but not in $\mathcal{Q}(\sigma_i)$. The finite groups $\mathcal{Q}(\sigma_i)$ and $\mathcal{T}(\sigma_i)$ cannot, then, be equal, showing that the orbit $\Omega$ of $i$ cannot be special. 
 \end{proof}

We are now ready to state a necessary and sufficient condition for a fixed point  $\sigma \in \mathcal{E}_2(M)$  to satisfy Assumption \ref{hyp_inclusions}.

\begin{defi} \label{good} Let $\sigma$ be a fixed point for the action of  $W_\Theta$ on $\mathcal{O}$. We say that $\sigma$ is a \emph{good fixed point} when an orbit $\Omega \in \{1, \dots, r\}$ can be weakly $R_\sigma$-relevant only if it is a singleton and is in addition $R_\sigma$-super-relevant. {When $\sigma$ is a fixed point and is not good, we say that $\sigma$ is a \emph{bad fixed point}.}
\end{defi}

Note that if $\sigma$ is a good fixed point, the only orbits that can contribute to the $R$-group are super-relevant singletons. Our result is then:

\begin{theo} \label{conclusion_inclu}  Let $G$ be the group $\mathrm{Sp}(2n,F)$, $n \in \mathbb{N}^\star$, let $\sigma$ be a discrete series representation of a Levi subgroup $M$ of $G$. Assume that $\sigma$ is a fixed point for the action of  $W_\Theta$ on $\mathcal{O}$. 
\begin{itemize}
\item[$\bullet$] If $\sigma$ is a good fixed point, then conditions (a) and (b) in Assumption \ref{hyp_inclusions} are satisfied. In fact, we have the following stronger assertion: for each $\chi \in \mathcal{X}_u(M)$, 
\begin{enumerate}[(a)]
\item we have $W'_{\sigma \otimes \chi} \subset W'_\sigma$,
\item[(b')] and the group $R_{\sigma \otimes \chi}$ is contained in $R_{\sigma}$.
\end{enumerate}
\item[$\bullet$] If $\sigma$ is a bad fixed point, then there exists a character $\chi \in \mathcal{X}_u(M)$ that satisfies either $\mathrm{Card}(R_{\sigma \otimes \chi}) > \mathrm{Card}(R_{\sigma})$ or $\mathrm{Card}(W'_{\sigma \otimes \chi}) > \mathrm{Card}(W'_{\sigma})$. Furthermore, if $\sigma$ is a bad fixed point, no unramified twist of $\sigma$ can be a good fixed point.
\end{itemize}
\end{theo}

\begin{proof} 
Let $\sigma = \sigma_1 \otimes \dots \otimes \sigma_r \otimes \tau$ be a good fixed point. Fix $\chi = \chi_1 \otimes \dots \otimes \chi_r$ in $\mathcal{X}_u(M)$. 

\begin{itemize}
\item[$\bullet$] We check (a) by comparing the groups $W'_{\sigma \otimes \chi}$ and $W'_\sigma$. Recall from \S \ref{defs_rgroupe} that these groups are generated by reflections attached to the roots of $(G,M)$ that are respectively $(\sigma \otimes \chi)$-relevant and $\sigma$-relevant. 

  Let $\Delta'_{\sigma \otimes \chi}$ and $\Delta'_{\sigma}$ be the corresponding sets of relevant roots. Then Goldberg has determined $\Delta'_{\sigma}$ {and $\Delta'_{\sigma\otimes \chi}$} explicitly, in terms of repetitions and repetitions-up-to-sign-change among the $\sigma_i$ {and the $\sigma_i \chi_i$}. We can rephrase the combination of \cite[Lemma 6.2]{GoldbergSPN} and \cite[Lemmas 4.8 and 4.17]{GoldbergSPN} as follows: 
{  \begin{itemize}
\item[$\star$] each pair $(i,j)$ such that $\sigma_{i}  \chi_i \simeq \sigma_j \chi_j$ contributes one root to $\Delta'_{\sigma \otimes \chi}$,
\item[$\star$] each pair $(i,j)$ such that $\sigma_{i} \chi_i \simeq \widecheck{\sigma_j \chi_j}$ contributes one root to $\Delta'_{\sigma \otimes \chi}$,
\item[$\star$] each integer $i$ such that $\widecheck{\sigma_i \chi_i} \simeq \sigma_i \chi_i $, but such that condition \ref{ci} is \emph{not} satisfied for $\sigma_i\chi_i$, contributes one root to $\Delta'_{\sigma \otimes \chi}$,
\item[$\star$] and all roots in $\Delta'_{\sigma \otimes \chi}$ come from one of the above contributions.
\end{itemize}
This determines $W'_{\sigma \otimes \chi}$; and by specializing to $\chi=1$, also determines $W'_{\sigma}$. }

We now point out that $\sigma$ is already a fixed point for the action of $W_{\Theta}$. Therefore, we cannot have $\sigma_i \chi_i \simeq \sigma_j \chi_j$ if we do not already have $\sigma_i \simeq \sigma_j$, and we cannot have $\sigma_i \chi_i \simeq \widecheck{\sigma_j \chi_j}$ if we do not already have $\sigma_i \simeq \widecheck{\sigma_j}$. 

We also know that if $i$ is an integer such that $\widecheck{\sigma_i \chi_i} \simeq \sigma_i \chi_i$, then by the fixed-point property we already have $\widecheck{\sigma_i} = \sigma_i$. Since we are assuming that $\sigma$ is a good fixed point, condition \ref{ci} cannot fail for $\sigma_i \chi_i$ unless it also fails for $\sigma_i$: if $i$ lies on a weakly $R_\sigma$-relevant orbit, by the good-fixed-point assumption it must lie on a $R_\sigma$-super-relevant orbit, so that condition  \ref{ci} for $\sigma_i \chi_i$ is equivalent with condition \ref{ci} for $\sigma_i$. 

We deduce that for every $i$ such that $\widecheck{\sigma_i \chi_i} \simeq \sigma_i \chi_i$, the integer $i$ cannot contribute a root to $\Delta'_{\sigma \otimes \chi}$ if it does not already contribute a root to $\Delta'_{\sigma}$. 

All this indicates that $\Delta'_{\sigma \otimes \chi}$ must be contained in $\Delta'_{\sigma}$, proving (a). 

  \item[$\bullet$] \label{verif_b} We now check (b), by determining $R_{\sigma \otimes \chi}$ and comparing it with $R_{\sigma}$. Recall from Theorem \ref{gold_r} that $R_{\sigma} \simeq (\mathbb{Z}/2\mathbb{Z})^d$, where $d$ is the number of $R_\sigma$-relevant orbits. Since we are assuming that $\sigma$ is a good fixed point, we know that every $\sigma$-relevant orbit $\Omega$ is a super-relevant singleton.

Let $\Omega = \{i_{\Omega}\}$ be such a singleton; then upon considering condition \ref{ci} for $\sigma_{i_{\Omega}} \chi_{i_{\Omega}}$, there are two possibilities. The first possibility is  $\chi_{i_{\Omega}} \notin \mathcal{T}(\sigma_{i_\Omega})$, and then $\sigma_{i_{\Omega}} \chi_{i_{\Omega}}$ cannot satisfy condition \ref{ci}, so that $\Omega$ does not contribute to the $R$-group at $\sigma \otimes \chi$. The second possibility is  $\chi_{i_{\Omega}} \in \mathcal{T}(\sigma_{i_\Omega})$, in which case $\sigma_{i_{\Omega}} \chi_{i_{\Omega}}$ must satisfy \ref{ci} because we requested that $\Omega$ be super-relevant; then $\Omega$ contributes a factor $\mathbb{Z}/2\mathbb{Z}$ to $R_{\sigma \otimes \chi}$.

Now, we note that only the $R_{\sigma}$-relevant orbits can contribute to $R_{\sigma \otimes \chi}$. Indeed, if $i$ is an element that does not lie on a $R_{\sigma}$-relevant orbit, and if the sign change $\mathbf{c}_i$ lies in the group $W_{\Theta}$,  then we have $\widecheck{\sigma_i} = \sigma_i$, but the irrelevance means that  condition \ref{ci} cannot be satisfied by $\sigma_i$; by the definition of good fixed points, we then see that it can be satisfied by none of the unramified twists of $\sigma$. Given the structure of $W_{\Theta}$ detailed in Lemma \ref{det_wt}, we deduce that no sign change in $W_{\Theta}$ can have a support with nonempty intersection with the union of $R_\sigma$-irrelevant orbits and at the same time lie in $R_{\sigma \otimes \chi}$. By Lemma \ref{signes}, this means that every element of $R_{\sigma \otimes \chi}$ must be a product of sign changes  with support in the union of $R_\sigma$-relevant orbits.

 We conclude that $R_{\sigma \otimes \chi} \simeq (\mathbb{Z}/2\mathbb{Z})^{d'}$, where $d'$ is the number of orbits $\Omega$ that are  a super-relevant singleton $\{i_\Omega\}$ satisfying $\chi_{i_{\Omega}} \in \mathcal{T}(\sigma_{i_\Omega})$.
  
 A consequence, of course, is that $d' \leq d$, proving part (b) of Assumption \ref{hyp_inclusions} and concluding our study of the case where $\sigma$ is a good fixed point. Another consequence is that all generators $\mathbf{c}_{i_{\Omega}}$ of $R_{\sigma \otimes \chi}$ already do lie in $R_\sigma$, whence the stronger assertion (b'). 
 
  \end{itemize}
Let us now turn to the case where $\sigma$ is a bad fixed point. What we must do is find a twist $\sigma \otimes \chi$ for which either the $R$-group gets larger, or the group $W'$ gets larger. 

Definition \ref{good} says that $\sigma$ is good if and only if all orbits are either (i) non-weakly-relevant, or (ii) at the same time super-relevant and a singleton. Since $\sigma$ is bad, there must exist an orbit $\Omega$ which is weakly $R_\sigma$-relevant,  and which is either not a singleton or a non-super-relevant singleton. 

Let $\Omega$ be such an orbit. We discuss cases separately:

\begin{itemize}
\item[\emph{Case 1:}] $\Omega$ is not a singleton. There are then three possibilities: 
\begin{itemize}
\item[\emph{Case 1(a):}] $\sigma_i$ satisfies \ref{ci}, but there exists $\nu \in \mathcal{T}(\sigma_i)$ such that $\sigma_i \otimes \nu$ does not satisfy \ref{ci}. 

We then twist $\sigma$ by the character $\chi \in \mathcal{X}_u(M)$ that operates by $\nu$ on each of the blocks $\mathrm{GL}(n_i,F)$, $i \in \Omega$, and trivially on the other blocks.  Under our hypothesis, the element  $\mathbf{c}_{\Omega}^{\mathrm{full}}$ of Theorem \ref{gold_r} lies in $R_{\sigma}$, but not in $R_{\sigma \otimes \chi}$. In fact, by \cite[Lemma 6.2]{GoldbergSPN}, we have $\mathbf{c}_{\Omega}^{\mathrm{full}} \in W'_{\sigma \otimes \chi}$. In addition, by an inspection of the $(\sigma \otimes \chi)$-relevant roots discussed in the starred items of our proof for the ``good fixed point'' case, we find out that the twist $\chi$ is special enough that we actually have $\Delta'_{\sigma} \subsetneq \Delta'_{\sigma \otimes \chi}$. This proves that $W'_{\sigma}$ is \emph{strictly contained} in $W'_{\sigma \otimes \chi}$, proving that condition (a) of Assumption \ref{hyp_inclusions} cannot hold.

\item[\emph{Case 1(b):}]  \label{cas1b} for all $\nu \in \mathcal{T}(\sigma_i)$, the representation $\sigma_i \otimes \nu$ satisfies \ref{ci}. 

We then use the fact that $\Omega$ cannot be special (Lemma \ref{speciales}), but must be relevant: so there must exist $\nu \neq 1$ in $\mathcal{T}(\sigma_i)$ such that $\sigma_i \nu \not\simeq \sigma_i$, and then both must satisfy \ref{ci}.

Now, $\Omega$ is not a singleton. We can therefore fix $i \neq j$ on $\Omega$, and twist $\sigma$ by the character $\chi$ that acts through $\nu$ on the component $\mathrm{GL}(n_j,F)$ of $M$, and trivially on the other components. Since the representations $\sigma_i$  and $\sigma_i \otimes \nu$ are not equivalent, but both satisfy \ref{ci} under our current hypotheses, we see from  \cite[Theorem 6.4]{GoldbergSPN} that the orbit $\Omega$  contributes a factor $(\mathbb{Z}/2\mathbb{Z})^2$ to the group $R_{\sigma \otimes \chi}$, but only a factor $(\mathbb{Z}/2\mathbb{Z})$ to the group $R_{\sigma}$. The contributions coming from other orbits do not change, and we conclude that  $\mathrm{Card}(R_{\sigma \otimes \chi}) > \mathrm{Card}(R_\sigma)$, forestalling the second part of Assumption \ref{hyp_inclusions}  in that case.

\item[\emph{Case 1(c):}] $\sigma_i$ does not satisfy \ref{ci}, but there exists $\nu \in \mathcal{T}(\sigma_i)$ such that $\sigma_i \otimes \nu$ satisfies~\ref{ci}.

Then we twist by the same character as in Case 1(a). Under our hypothesis, a discussion analogous to Case 1(a) reveals that the element $\mathbf{c}_{\Omega}^{\mathrm{full}}$ lies in $R_{\sigma \otimes \chi}$, but also lies in $W'_{\sigma}$. Furthermore, the twist $\chi$ does not affect the orbits distinct from $\Omega$, so that the generators of $R_{\sigma}$ attached to orbits distinct from $\Omega$ also lie in $R_{\sigma \otimes \chi}$. We see then that $\mathrm{Card}(R_{\sigma \otimes \chi}) > \mathrm{Card}(R_\sigma)$, proving that the second part of Assumption \ref{hyp_inclusions} cannot hold.

\end{itemize}
Since $\Omega$ must be weakly relevant, we have exhausted the possibilities where $\Omega$ is not a singleton. This concludes our inspection of Case 1. 
\item[\emph{Case 2:}] \label{cas2} $\Omega$ is a singleton, say $\Omega = \{i_\Omega\}$, and is weakly relevant without being super-relevant. We are left with two possibilities:

\begin{itemize}

\item[\emph{Case 2(a):}] $\sigma_{i_{\Omega}}$ satisfies \ref{ci}, but there exists $\nu \in \mathcal{T}(\sigma_{i_\Omega})$ such that $\sigma_{i_\Omega} \nu$ does not. Then a discussion analogous to Case 1(a) furnishes a character $\chi$ such that $\mathrm{Card}(W'_{\sigma \otimes \chi})>\mathrm{Card}(W'_{\sigma})$, and we have $\mathbf{c}_{i_{\Omega}} \in W'_{\sigma \otimes \chi}$ while $\mathbf{c}_{i_{\Omega}} \in R_{\sigma}$. 

\item[\emph{Case 2(b):}] $\sigma_{i_{\Omega}}$ does not satisfy \ref{ci}, but  there exists $\nu \in \mathcal{T}(\sigma_{i_\Omega})$ such that $\sigma_{i_\Omega} \nu$ does satisfy \ref{ci}. Then by taking up Case 1(b), we obtain a character $\chi$ such that $\mathrm{Card}(R_{\sigma \otimes \chi})>\mathrm{Card}(R_{\sigma})$, and    we have $\mathbf{c}_{i_{\Omega}} \in R_{\sigma \otimes \chi}$ while $\mathbf{c}_{i_{\Omega}} \in W'_{\sigma}$. 
\end{itemize}
\end{itemize}
The only part of Theorem \ref{conclusion_inclu} that remains to be proven is that no unramified twist of $\sigma$ can be a good fixed point if $\sigma$ is bad. Let $\chi = \chi_1 \otimes \dots \otimes \chi_r$ be an unramified character of $M$; if $\sigma \otimes \chi$ is a  fixed point for $W_{\Theta}$, then we can make the following observations: 
\begin{enumerate}[(i)]
\item if $\Omega$ is a pure orbit in $\{1, \dots, r\}$, we must have $\sigma_i \otimes \chi_i = \sigma_j \otimes \chi_j$ whenever $i$ and $j$ lie on $\Omega$,
\item if $\Omega$ is in addition type-1, we must have $\widecheck{\sigma_i  \chi_i} = \sigma_i  \chi_i$ for all $i \in \Omega$, so we must have $\chi_i \in \mathcal{T}(\sigma_i)$ for $i \in \Omega$. 

\end{enumerate}
 
Now, assume that $\sigma \otimes \chi$ is good. Any orbit that is weakly relevant for $\sigma$ must be weakly relevant for $\sigma \otimes \chi$ by (ii),  and therefore must be a strongly relevant singleton for $\sigma \otimes \chi$; but then condition  (ii) again implies that it is  also a strongly relevant singleton for $\sigma$, proving that $\sigma$ must have been a good fixed point in the first place.
\end{proof}

\subsection{Other groups, and the cases of $\mathrm{SO}(2n,F)$ and $\mathrm{SO}^\star(2n,F)$} The discussion in \S \ref{bons}, especially Definition \ref{good} and the analogue of Theorem \ref{conclusion_inclu}, is applicable to odd-orthogonal and unitary groups, the only difference being that the involutions used in sign changes must be switched to the ones in \S \ref{so_2n+1} and \S \ref{u_n}. (Note that odd-orthogonal and unitary groups do have compact center, securing the simplification recorded in the footnote on page \pageref{centre_compact}.) Our last task, therefore, is to indicate the way in which the discussion of \S \ref{bons} must be adapted when $G = \mathrm{SO}(2n,F)$ or $G =  \mathrm{SO}^\star(2n,F)$, $n \geq 1$.

It will turn out that Theorem \ref{conclusion_inclu} applies to $G$, exactly as stated, if we do use Definitions \ref{relevant} and \ref{good}. Let us, for the record, isolate that statement:

\begin{prop} \label{cas_so2N} When $G=\mathrm{SO}(2n,F)$ or $G=\mathrm{SO}^\star(2n,F)$, the analogue of Theorem \ref{conclusion_inclu} holds for $G$. \end{prop}

The difference between these two groups and the others already discussed is that, for certain pairs $(M, \sigma)$, the interpretation of Definition \ref{good} and the proof of Theorem \ref{conclusion_inclu} must be adjusted, following the ideas and results in \cite[\S 5 and \S 6 after p. 1143]{GoldbergSPN}.

\subsubsection{Split  even-orthogonal groups} \label{so_split} We first consider $G= \mathrm{SO}(2n,F)$, and discuss the case of an arbitrary pair $(M, \sigma)$, where 
\[ M \simeq \mathrm{GL}(n_1,F) \times \dots \times \mathrm{GL}(n_r,F) \times \mathrm{SO}(2q,F)\]
(we recall that the case $q=0$, with the last factor trivial, is quite admissible), and
 \[ \sigma = \sigma_1 \otimes \dots \otimes \sigma_r \otimes \tau\] 
is a discrete series representation of $M$, where  $\sigma_i \in \mathcal{E}_2(\mathrm{GL}(n_i,F))$ for $i \in \{1, \dots, r\}$, and $\tau \in \mathcal{E}_2(\mathrm{SO}(2q,F))$. Recall that $\widetilde{\tau}$ denotes the representation $g \mapsto \tau(\mathbf{C}g\mathbf{C})$ of $\mathrm{SO}(2q,F)$, where $\mathbf{C}$ is the matrix in \eqref{matrice_c}. 

As before, we assume that $\sigma$ is a fixed point for the action of $W_\Theta$ on $\mathcal{O}$, and embark on a discussion of the structure of $R_\sigma$ and $W'_{\sigma}$. We must distinguish three cases: 
\begin{enumerate}
\item[\emph{Case 1 :}]  \emph{$q>1$, and the representation $\tau$ of $\mathrm{SO}(2q,F)$ satisfies $\widetilde{\tau} \simeq \tau$.}

In that case, Goldberg  proved in \cite[Theorems 5.8, 5.9, 5.19, 5.20, 6.5]{GoldbergSPN} that the computation of $R$-groups can be conducted exactly as in the symplectic case. All our arguments also apply to this case, and there is no new phenomenon to be reported. 

\item[\emph{Case 2 :}]    \emph{$q=0$, or $\widetilde{\tau} \not\simeq \tau$.}

\item[\emph{Case 3 :}]    \emph{$q=1$.}

\end{enumerate} 

For the third case, we note that the split group $\mathrm{SO}(2,F)$ is the group of $2 \times 2$ matrices of the form $\mathrm{Diag}(a,a^{-1})$, $a \in F^\times$. Therefore the standard subgroup isomorphic with $\mathrm{GL}(n_1,F) \times \dots \times \mathrm{GL}(n_r,F) \times \mathrm{SO}(2,F)$ is none other than the standard subgroup isomorphic with $\mathrm{GL}(n_1,F) \times \dots \times \mathrm{GL}(n_r,F) \times \mathrm{GL}(1,F)$. As a result, Case 3 can be viewed as an instance of Case 2. 

Only Case 2, therefore, remains. As we will see, the discussion that has to be conducted for Case 2 is significantly simpler than that we already conducted for Case 1. 

As a first step, let us recall that Goldberg proved that we can separately consider the blocks of odd size and the blocks of even size. Define
\begin{itemize}
\item[$\bullet$]  $N_{\textrm{odd}}$: the sum of all odd $n_i$,
\item[$\bullet$] $G_{\textrm{odd}}$: the group $\mathrm{SO}(2N_{\textrm{odd}} +2q,F)$, 
\item[$\bullet$]  $M_{\textrm{odd}} $: the standard Levi subgroup of $G_{\textrm{odd}} $ isomorphic with the product of all $\mathrm{GL}(n_i,F)$, $n_i$ odd, and $\mathrm{SO}(2q,F)$,
\item[$\bullet$]  $\sigma_{\textrm{odd}} \in \mathcal{E}_2(M_{\textrm{odd}} )$: the tensor product of all $\sigma_i$, $n_i$ odd, and of $\tau$. 
\end{itemize}

Attached to the discrete pair $(M_{\textrm{odd}}, \sigma_{\textrm{odd}})$ for $G_{\textrm{odd}}$ are a Weyl group $W(\sigma_ {\textrm{odd}})$, a system $\Delta'_{\sigma_{\mathrm{odd}}}$ of relevant roots, and groups $W'_{\mathrm{odd}}$ and $R_{\mathrm{odd}}$. 

Similarly, we can define $N_{\text{even}}$, $G_{\textrm{even}}$, $M_{\textrm{even}}$, $\sigma_{\textrm{even}}$, and then they come with associated data $W(\sigma_ {\textrm{even}})$, $\Delta'_{\sigma_{\mathrm{even}}}$, $W'_{\mathrm{even}}$ and $R_{\mathrm{even}}$.  Goldberg then proves in \cite[\S 5]{GoldbergSPN} that: 
 \begin{itemize}
 \item[$\bullet$]  we have the product decompositions 
 \begin{itemize} 
 \item $W_{\sigma} = W(\sigma_ {\textrm{odd}}) \times W(\sigma_{\textrm{even}})$, 
 \item $\Delta'_{\sigma} = \Delta'_{\mathrm{odd}} \cup \Delta'_{\mathrm{even}}$, and so $W'_{\sigma}= W'_{\mathrm{odd}}\times W'_{\mathrm{odd}}$,
 \item  $R_\sigma= R_{\mathrm{odd}}\times R_{\mathrm{even}}$; 
 \end{itemize}
 \item[$\bullet$]  furthermore, the groups {$W'_{\mathrm{even}}$ and $R_{\mathrm{even}}$} can be calculated exactly as in the symplectic case. 
 \end{itemize}
 
{Because of the direct products decompositions, we see that $\sigma$ is a good fixed point if and only if $\sigma_{\textrm{odd}}$ and $\sigma_{\textrm{even}}$ are both good fixed points when these notions are considered for $G_{\mathrm{odd}}$ and $G_{\mathrm{even}}$. Furthermore, the last bullet point above implies that we have already conducted the necessary discussion for $\sigma_{\textrm{even}}$, and we have already assumed that we are considering Case $2$, where $q=0$ or  $\widetilde{\tau} \not\simeq \tau$. Therefore, if we can prove Proposition \ref{cas_so2N}  when all $n_i$ are odd and $q=0$, and when all $n_i$ are odd and $\widetilde{\tau} \not\simeq \tau$, then we will have completed the task we set ourselves. }

As a last preliminary remark, let us recall three facts from Goldberg's work which will considerably simplify matters in the situation under discussion (see \cite[Theorems 6.8 and 6.11 and Lemmas 6.6 and 6.9]{GoldbergSPN}).

\begin{lemm}[\cite{GoldbergSPN}, \S 6] \label{cas_biz} Assume that all $n_i$ are odd, and that we have either $q=0$ or $\widetilde{\tau} \not\simeq \tau$. 
\begin{enumerate}[(1)]
\item For all $i \in \{1, \dots, r\}$, condition \ref{ci} is equivalent with $\widecheck{\sigma_i}=\sigma_i$.
\item Let $m(\sigma)$ be the number of inequivalent $\sigma_i$ such that $\widecheck{\sigma_i} \simeq \sigma_i$. Then 
\[ R_{\sigma} \simeq \begin{cases} 1 & \text{ if $m(\sigma)=0$, } \\  (\mathbb{Z}/2\mathbb{Z})^{m(\sigma)-1} & \text{ if $m(\sigma) \geq 1$.} \end{cases}\] In the second case, $R_{\sigma}$ is generated by a collection of double-sign-changes $\mathbf{c}_i\mathbf{c}_j$, indexed by $(m(\sigma)-1)$ pairs $\{i,j\}$ contained in the union of pure type-I orbits.
\item The group $W'_{\sigma}$ is generated by 

\begin{itemize}
\item[$\bullet$] the transpositions  $(ij)$ for which $n_i=n_j$ and $\sigma_i \simeq \sigma_j$,
\item[$\bullet$] and the products $(ij)\mathbf{c}_i\mathbf{c}_j$ such that $i$ and $j$ lie on a mixed orbit $\Omega = \Omega^{\mathrm{per}}\cup\Omega^{\mathrm{flip}}$ and one of $i,j$ lies in $\Omega^{\mathrm{flip}}$ and the other lies in $\Omega^{\mathrm{per}}$. \end{itemize} 

\end{enumerate}
(In the statement above, we need not assume that $\sigma$ is a fixed point for $W_{\Theta}$.)

 \end{lemm} 

\noindent  We can now discuss Definition \ref{good} and the adaptations to be made to the proof of Theorem~\ref{conclusion_inclu}. 

Concerning Definition \ref{good}, recall that $\sigma$ is a {good fixed point} when all weakly relevant orbits $\Omega$ are super-relevant singletons. Now, Lemma \ref{cas_biz}(1) means that weak relevance implies super-relevance, and is automatically satisfied by all pure type-I orbits (recall that it is \emph{not} satisfied by the other orbits). Therefore, the fact that $\sigma$ is  a good fixed point simply means here that if $\Omega$ is an orbit such that $\widecheck{\sigma_i} = \sigma_i$ for $i \in \Omega$, then $\Omega$ must be a singleton. 

We turn, now, to the proof of Theorem \ref{conclusion_inclu} in the present case. 

Assume that $\sigma = \sigma_1 \otimes \dots \otimes \sigma_r \otimes \tau$ satisfies the hypothesis of Lemma \ref{cas_biz}.  

Let us inspect the case where $\sigma$ is a good fixed point, \label{rq_comp} fixing\footnote{Note that the center of $\mathrm{SO}(2q,F)$ is compact for $q\neq1$: this ensures that any unramified character $\mathcal{X}_u(M)$ has trivial restriction to the block $\mathrm{SO}(2q,F)$. } $\chi = \chi_1 \otimes \dots \otimes \chi_r$ in $\mathcal{X}_u(M)$ and discussing the structure of $W'_{\sigma \otimes \chi}$ and $R_{\sigma \otimes \chi}$. 

Part (a) in Assumption \ref{hyp_inclusions} can then be checked quite simply from the description of $W'_{\sigma \otimes \chi}$ in Lemma \ref{cas_biz}(3). Indeed, if $i$ and $j$ are such that $n_i=n_j$ and $\sigma_i \chi_i \simeq \sigma_j \chi_j$, then we already know that $\sigma_i \simeq \sigma_j$ by the fixed-point property. Therefore, all transpositions $(ij)$ that lie in $W'_{\sigma \otimes \chi}$ must also lie in $W'_{\sigma}$. The same argument shows that all products $(ij)\mathbf{c}_i\mathbf{c}_j$ that lie in $W'_{\sigma \otimes \chi}$ must already lie in $W'_{\sigma}$. Now, using Lemma \ref{cas_biz}(3)  at $\sigma \otimes \chi$, we deduce that $W'_{\sigma \otimes \chi}$ is contained in $W'_{\sigma}$, proving part (a) in Assumption \ref{hyp_inclusions}. 

We now point out that for part (b), Lemma \ref{cas_biz}(2) means that it is enough for us to remark that we necessarily have $m(\sigma \otimes \chi) \leq m(\sigma)$. 

Consider, then, an integer $i \in \{1, \dots, r\}$ such that $\widecheck{\sigma_{i}} \simeq \sigma_i$. As before, we remark that the hypothesis that $\sigma$ is a fixed point for $W_{\Theta}$ means that we cannot have $\widecheck{\sigma_i \chi_i} \simeq \sigma_i \chi_i$ if we do not already have $\widecheck{\sigma_i} \simeq \sigma_i$. Furthermore, we notice that for $i,j \in \{1, \dots, r\}$ such that $\widecheck{\sigma_i} = \sigma_i$ and $\widecheck{\sigma_j} = \sigma_j$, we cannot have $\sigma_i \chi_i \simeq \sigma_j \chi_j$ unless we already have $\sigma_i \simeq \sigma_j$. We can therefore conclude that  $m(\sigma \otimes \chi) \leq m(\sigma)$, proving (b) and closing our discussion of the case of a good fixed point.

Finally, when $\sigma$ is a bad fixed point, the proof of the second part of Theorem \ref{conclusion_inclu} goes through, with additional simplifications: the case of a singleton (Case 2 on page \pageref{cas2}) can be excluded, and for the kind of non-singleton orbits that have to be considered in the proof, only case 1(b) on page \pageref{cas1b} can present itself. \qed

\subsubsection{Quasi-split but non-split even-orthogonal groups} We now consider the quasi-split but non-split group $G= \mathrm{SO}^\star(2n,F)$, and discuss the case of an arbitrary pair $(M, \sigma)$, where $M \simeq \mathrm{GL}(n_1,F) \times \dots \times \mathrm{GL}(n_r,F) \times \mathrm{SO}^\star(2q,F),$ and the discrete series representation $\sigma$ reads $\sigma = \sigma_1 \otimes \dots \otimes \sigma_r \otimes \tau$ as above. 
 
We take up the three cases of \S \ref{so_split}: 
\begin{enumerate}
\item[\emph{Case 1 :}]  \emph{$q>1$, and the representation $\tau$ of $\mathrm{SO}^\star(2q,F)$ satisfies $\widetilde{\tau} \simeq \tau$.}

\item[\emph{Case 2 :}]    \emph{$q=0$, or $\widetilde{\tau} \not\simeq \tau$.}

\item[\emph{Case 3 :}]    \emph{$q=1$.}

\end{enumerate} 

Using the results of \cite[Appendix]{ChoiyGoldberg}, the analysis we conducted for the split group  $\mathrm{SO}(2n,F)$ extends verbatim to cover Cases $1$ and $2$ for $\mathrm{SO}^\star(2n,F)$. The difference between $\mathrm{SO}^\star(2n,F)$ and $\mathrm{SO}^\star(2n,F)$ is now that in Case $3$,  the group $\mathrm{SO}^\star(2q,F)=\mathrm{SO}^\star(2,F)$ is compact (it is an anisotropic torus, see \cite[pp. 7-8]{Borel} and \cite{Prasad}). Therefore, it has no nontrivial unramified character.  This removes the only obstacle that prevented us, in the previous subsection, from treating Case $3$ together with Case $1$: so the proof of Theorem \ref{conclusion_inclu} given for symplectic groups extends to cover Cases $1$ and $3$. The analysis of \S \ref{so_split} covers Case $2$, and we conclude that Theorem \ref{conclusion_inclu} does hold for $G=\mathrm{SO}^\star(2n,F)$.

%%%-------------------------------------------------------------------------------%%%
%%%-------------------------------------------------------------------------------%%%
%%%-------------------------------------------------------------------------------%%%
%%%-------------------------------------------------------------------------------%%%
%%%-------------------------------------------------------------------------------%%%
%
%

\providecommand{\bysame}{\leavevmode ---\ }
\providecommand{\og}{``}
\providecommand{\fg}{''}
\providecommand{\smfandname}{\&}
\providecommand{\smfedsname}{\'eds.}
\providecommand{\smfedname}{\'ed.}
\providecommand{\smfmastersthesisname}{M\'emoire}
\providecommand{\smfphdthesisname}{Th\`ese}


\begin{thebibliography}{10}

\bibitem{AATopo2}
{\scshape A.~{Afgoustidis} {\normalfont \smfandname} A.-M. {Aubert}} -- {\og
  {Continuity of the Mackey-Higson bijection}\fg},
  \emph{Pacific J. Math.}, \textbf{310} (2021), no.~2, p.~257--273. 

\bibitem{ArthurJFA}
{\scshape J.~Arthur} -- {\og Intertwining operators and residues. {I}.
  {W}eighted characters\fg}, \emph{J. Funct. Anal.} \textbf{84} (1989), no.~1,
  p.~19--84.

\bibitem{ArthurActa}
\bysame , {\og On elliptic tempered characters\fg}, \emph{Acta Math.}
  \textbf{171} (1993), no.~1, p.~73--138.

\bibitem{ABPSTakagi}
{\scshape A.-M. Aubert, P.~Baum, R.~Plymen {\normalfont \smfandname}
  M.~Solleveld} -- {\og Geometric structure in smooth dual and local
  {L}anglands conjecture\fg}, \emph{Jpn. J. Math.} \textbf{9} (2014), no.~2,
  p.~99--136.

\bibitem{ABPSOrsay}
\bysame , {\og Conjectures about {$p$}-adic groups and their noncommutative
  geometry\fg}, in \emph{Around {L}anglands correspondences}, Contemp. Math.,
  vol. 691, Amer. Math. Soc., Providence, RI, 2017, p.~15--51.

\bibitem{Borel}
{\scshape A.~Borel} -- {\og Linear algebraic groups\fg}, in \emph{Algebraic
  {G}roups and {D}iscontinuous {S}ubgroups ({P}roc. {S}ympos. {P}ure {M}ath.,
  {B}oulder, {C}olo., 1965)}, Amer. Math. Soc., Providence, R.I., 1966,
  p.~3--19.

\bibitem{BKLivre}
{\scshape C.~J. Bushnell {\normalfont \smfandname} P.~C. Kutzko} -- \emph{The
  admissible dual of {${\rm GL}(N)$} via compact open subgroups}, Annals of
  Mathematics Studies, vol. 129, Princeton University Press, Princeton, NJ,
  1993.

\bibitem{ChaoPlymen}
{\scshape K.~F. Chao {\normalfont \smfandname} R.~Plymen} -- {\og Geometric
  structure in the tempered dual of {$\rm SL(4)$}\fg}, \emph{Bull. Lond. Math.
  Soc.} \textbf{44} (2012), no.~3, p.~460--468.


\bibitem{ChoiyGoldberg}
{\scshape K.~Choiy {\normalfont \smfandname} D.~Goldberg} -- {\og Invariance of
  {$R$}-groups between {$p$}-adic inner forms of quasi-split classical
  groups\fg}, \emph{Trans. Amer. Math. Soc.} \textbf{368} (2016), no.~2,
  p.~1387--1410.

\bibitem{CCH}
{\scshape P.~Clare, T.~Crisp {\normalfont \smfandname} N.~Higson} -- {\og
  Parabolic induction and restriction via {$C^*$}-algebras and {H}ilbert
  {$C^*$}-modules\fg}, \emph{Compos. Math.} \textbf{152} (2016), no.~6,
  p.~1286--1318.
  

\bibitem{Connes}
{\scshape A. Connes} -- \emph{Noncommutative Geometry}, Academic Press, Inc., San Diego, CA, 
  1994.

  
\bibitem{Dieudonne}
{\scshape J.~A. Dieudonn\'{e}} -- \emph{La g\'{e}om\'{e}trie des groupes
  classiques}, Springer-Verlag, Berlin-New York, 1971, Troisi\`eme \'{e}dition,
  Ergebnisse der Mathematik und ihrer Grenzgebiete, Band 5.

\bibitem{Dixmier}
{\scshape J. Dixmier } -- \emph{{$C\sp*$}-algebras}, Translated from the French by Francis Jellett,
              North-Holland Mathematical Library, Vol. 15, North-Holland Publishing Co., Amsterdam-New York-Oxford, 1977.
 

\bibitem{EchterhoffEmerson}
{\scshape S.~Echterhoff {\normalfont \smfandname} H.~Emerson } -- {\og {Structure and {$K$}-theory of crossed products by proper
              actions}\fg}, \emph{Expo. Math.} \textbf{29} (2011), no.~3,
  p.~300--344.
  
\bibitem{GoldbergSLN}
{\scshape D.~Goldberg} -- {\og {$R$}-groups and elliptic representations for
  {${\rm SL}_n$}\fg}, \emph{Pacific J. Math.} \textbf{165} (1994), no.~1,
  p.~77--92.

\bibitem{GoldbergSPN}
\bysame , {\og Reducibility of induced representations for {${\rm Sp}(2n)$} and
  {${\rm SO}(n)$}\fg}, \emph{Amer. J. Math.} \textbf{116} (1994), no.~5,
  p.~1101--1151.

\bibitem{GoldbergUN}
\bysame , {\og {$R$}-groups and elliptic representations for unitary
  groups\fg}, \emph{Proc. Amer. Math. Soc.} \textbf{123} (1995), no.~4,
  p.~1267--1276.

\bibitem{GoldbergHerb}
{\scshape D.~Goldberg {\normalfont \smfandname} R.~Herb} -- {\og Some results
  on the admissible representations of non-connected reductive {$p$}-adic
  groups\fg}, \emph{Ann. Sci. \'{E}cole Norm. Sup. (4)} \textbf{30} (1997),
  no.~1, p.~97--146.

\bibitem{GoldbergShahidi}
{\scshape D.~Goldberg {\normalfont \smfandname} F.~Shahidi} -- {\og On the
  tempered spectrum of quasi-split classical groups\fg}, \emph{Duke Math. J.}
  \textbf{92} (1998), no.~2, p.~255--294.

\bibitem{JawdatPlymen}
{\scshape J.~Jawdat {\normalfont \smfandname} R.~Plymen} -- {\og {$R$}-groups
  and geometric structure in the representation theory of {${\rm SL}(N)$}\fg},
  \emph{J. Noncommut. Geom.} \textbf{4} (2010), no.~2, p.~265--279.

\bibitem{KamranPlymen}
{\scshape T.~Kamran {\normalfont \smfandname} R.~Plymen} -- {\og {$K$}-theory
  and the connection index\fg}, \emph{Bull. Lond. Math. Soc.} \textbf{45}
  (2013), no.~1, p.~111--119.

\bibitem{Keys}
{\scshape C.~D. Keys} -- {\og On the decomposition of reducible principal
  series representations of {$p$}-adic {C}hevalley groups\fg}, \emph{Pacific J.
  Math.} \textbf{101} (1982), no.~2, p.~351--388.

\bibitem{KnappStein}
{\scshape A.~W. Knapp {\normalfont \smfandname} E.~W. Stein} -- {\og
  Intertwining Operators for Semisimple Groups, \textsc{II}\fg}, \emph{Invent. Math.}
  \textbf{60} (1980), p.~9--84.
  
\bibitem{MorrisKutzko}
{\scshape P.~Kutzko {\normalfont \smfandname} L.~Morris} -- {\og Level zero
  {H}ecke algebras and parabolic induction: the {S}iegel case for split
  classical groups\fg}, \emph{Int. Math. Res. Not.} (2006), p.~Art. ID 97957,
  40.

\bibitem{Lance}
{\scshape E.~C. Lance} -- \emph{Hilbert {$C^*$}-modules}, London Mathematical
  Society Lecture Note Series, vol. 210, Cambridge University Press, Cambridge,
  1995, A toolkit for operator algebraists.


\bibitem{Langlands}
{\scshape R.~{Langlands}} -- {\og {Friday Morning {S}eminar on the {T}race
  {F}ormula}\fg}, \emph{unpublished} (1984).

\bibitem{NPW}
{\scshape G.~A. Niblo, R.~Plymen {\normalfont \smfandname} N.~Wright} -- {\og
  Stratified {L}anglands duality in the {$A_n$} tower\fg}, \emph{J. Noncommut.
  Geom.} \textbf{13} (2019), no.~1, p.~193--225.

\bibitem{OpdamSolleveld}
{\scshape E.~Opdam {\normalfont \smfandname} M.~Solleveld} -- {\og Extensions
  of tempered representations\fg}, \emph{Geom. Funct. Anal.} \textbf{23}
  (2013), no.~2, p.~664--714.

\bibitem{Plymen1}
{\scshape R.~J. Plymen} -- {\og Reduced {$C^*$}-algebra for reductive
  {$p$}-adic groups\fg}, \emph{J. Funct. Anal.} \textbf{88} (1990), no.~2,
  p.~251--266.

\bibitem{LeungPlymen}
{\scshape R.~J. Plymen {\normalfont \smfandname} C.~W. Leung} -- {\og
  Arithmetic aspect of operator algebras\fg}, \emph{Compositio Math.}
  \textbf{77} (1991), no.~3, p.~293--311.
  
  
\bibitem{Prasad}
{\scshape G.~Prasad} -- {\og Elementary proof of a theorem of
  {B}ruhat-{T}its-{R}ousseau and of a theorem of {T}its\fg}, \emph{Bull. Soc.
  Math. France} \textbf{110} (1982), no.~2, p.~197--202.

\bibitem{Prolla}
{\scshape J.~B. Prolla} -- {\og On the {W}eierstrass-{S}tone theorem\fg},
  \emph{J. Approx. Theory} \textbf{78} (1994), no.~3, p.~299--313.
  
\bibitem{Renard}
{\scshape D.~Renard} -- \emph{Repr\'{e}sentations des groupes r\'{e}ductifs
  {$p$}-adiques}, Cours Sp\'{e}cialis\'{e}s [Specialized Courses], vol.~17,
  Soci\'{e}t\'{e} Math\'{e}matique de France, Paris, 2010.

\bibitem{Rieffel}
{\scshape M.~A. Rieffel} -- {\og Morita equivalence for operator algebras\fg},
  in \emph{Operator algebras and applications, {P}art {I} ({K}ingston, {O}nt.,
  1980)}, Proc. Sympos. Pure Math., vol.~38, Amer. Math. Soc., Providence,
  R.I., 1982, p.~285--298.
  
\bibitem{Shahidi}
{\scshape F.~Shahidi} -- {\og A proof of {L}anglands' conjecture on
  {P}lancherel measures; complementary series for {$p$}-adic groups\fg},
  \emph{Ann. of Math. (2)} \textbf{132} (1990), no.~2, p.~273--330.

\bibitem{ShahidiDuke}
\bysame , {\og Twisted endoscopy and reducibility of induced representations
  for {$p$}-adic groups\fg}, \emph{Duke Math. J.} \textbf{66} (1992), no.~1,
  p.~1--41.

\bibitem{SilbergerKnappStein}
{\scshape A.~J. Silberger} -- {\og The {K}napp-{S}tein dimension theorem for
  {$p$}-adic groups\fg}, \emph{Proc. Amer. Math. Soc.} \textbf{68} (1978),
  no.~2, p.~243--246.

\bibitem{SilbergerRgroupe}
\bysame , {\og The {K}napp-{S}tein dimension theorem for {$p$}-adic groups\fg},
  \emph{Proc. Amer. Math. Soc.} \textbf{68} (1978), no.~2, p.~243--246.

\bibitem{Solleveld2}
{\scshape M.~Solleveld} -- {\og On the classification of irreducible
  representations of affine {H}ecke algebras with unequal parameters\fg},
  \emph{Represent. Theory} \textbf{16} (2012), p.~1--87.

\bibitem{Solleveld3}
\bysame , {\og Topological {K}-theory of affine {H}ecke algebras\fg},
  \emph{Ann. K-Theory} \textbf{3} (2018), no.~3, p.~395--460.

\bibitem{Solleveld1}
\bysame , {\og On completions of {H}ecke algebras\fg}, in \emph{Representations
  of reductive {$p$}-adic groups}, Progr. Math., vol. 328,
  Birkh\"{a}user/Springer, Singapore, 2019, p.~207--262.

\bibitem{Waldspurger}
{\scshape J.-L. Waldspurger} -- {\og La formule de {P}lancherel pour les
  groupes {$p$}-adiques (d'apr\`es {H}arish-{C}handra)\fg}, \emph{J. Inst.
  Math. Jussieu} \textbf{2} (2003), no.~2, p.~235--333.

\bibitem{WaldspurgerGGP2}
\bysame , {\og Une formule int\'{e}grale reli\'{e}e \`a la conjecture locale de
  {G}ross-{P}rasad, 2e partie: extension aux repr\'{e}sentations
  temp\'{e}r\'{e}es\fg}, Ast\'{e}risque, no. 346, 2012, Sur les conjectures de
  Gross et Prasad. I, p.~171--312.

\bibitem{Wassermann}
{\scshape A.~Wassermann} -- {\og Une d\'{e}monstration de la conjecture de
  {C}onnes-{K}asparov pour les groupes de {L}ie lin\'{e}aires connexes
  r\'{e}ductifs\fg}, \emph{C. R. Acad. Sci. Paris S\'{e}r. I Math.}
  \textbf{304} (1987), no.~18, p.~559--562.

\end{thebibliography}
\end{document}